\documentclass[12pt]{amsart}

\usepackage{amssymb}
\usepackage{mathrsfs}
\usepackage{xypic}
\usepackage{enumerate}
\usepackage{graphicx}
\usepackage[text={6in,9in},centering]{geometry}
\usepackage{hyperref}
\usepackage{relsize}
\usepackage[english]{babel}
\usepackage{booktabs}
\usepackage{longtable}
\usepackage{makecell}
\usepackage{multirow}
\usepackage{caption}

\hypersetup{
    colorlinks = true,
    linkcolor = blue,
    anchorcolor = blue,
    citecolor = blue
}

\numberwithin{equation}{section}

\theoremstyle{plain}
  \newtheorem{thm}{Theorem}[section]
  \newtheorem*{thm*}{Theorem}
  \newtheorem{cor}[thm]{Corollary}
  \newtheorem{lem}[thm]{Lemma}
  \newtheorem{prop}[thm]{Proposition}
  
\theoremstyle{definition}
  \newtheorem{defn}{Definition}
  \newtheorem{exmp}{Example}
  
  \newtheorem{claim}{Claim}
\theoremstyle{remark}
  \newtheorem{rmk}{Remark}

\newcommand{\id}[0]{\mathrm{id}}
\newcommand{\e}[0]{\mathrm{e}}

\newcommand{\diam}[0]{\mathrm{diam}}
\newcommand{\dif}[0]{\mathrm{d}}
\newcommand{\supp}[0]{\operatorname{supp}}
\newcommand{\dimbox}[0]{\dim_{\mathrm{B}}}
\newcommand{\dimh}[0]{\dim_{\mathrm{H}}}
\newcommand{\dimp}[0]{\dim_{\mathrm{P}}}
\newcommand{\vect}[1]{\boldsymbol{#1}}

\newcommand{\entlow}{\underline{h}}
\newcommand{\entup}{\overline{h}}
\newcommand{\enttop}{h_{\mathrm{top}}}
\newcommand{\enttopbow}{\enttop^{\mathrm{B}}}
\newcommand{\enttoppac}{\enttop^{\mathrm{P}}}
\newcommand{\entbow}{h^{\mathrm{B}}}
\newcommand{\entpac}{h^{\mathrm{P}}}
\newcommand{\prelow}{\underline{P}}
\newcommand{\preup}{\overline{P}}
\newcommand{\prebpp}{Q^{\mathrm{B}}}
\newcommand{\prebow}{P^{\mathrm{B}}}
\newcommand{\prebw}{P^{\mathrm{BW}}}
\newcommand{\preppp}{Q^{\mathrm{P}}}
\newcommand{\prepac}{P^{\mathrm{P}}}

\newcommand{\msrbpp}{\mathscr{M}}
\newcommand{\msrppp}{\mathscr{N}}
\newcommand{\msrbow}{\mathscr{R}}
\newcommand{\msrpac}{\mathscr{P}}
\newcommand{\leb}{\mathscr{L}}

\newcommand{\R}{\mathbb{R}}

\newcommand{\N}{\mathbb{N}}

\newcommand{\red}[1]{{\color{red}#1}}
\newcommand{\blue}[1]{{\color{blue}#1}}

\begin{document}

\title{Nonautonomous Dynamical Systems II: Variational principles}

\author{Zhuo Chen}
\address{School of Mathematical Sciences, East China Normal University, No. 500, Dongchuan Road, Shanghai 200241, P. R. China}
\email{10211510056@stu.ecnu.edu.cn, charlesc-hen@outlook.com}

\author{Jun Jie Miao}
\address{School of Mathematical Sciences,  Key Laboratory of MEA(Ministry of Education) \& Shanghai Key Laboratory of PMMP,  East China Normal University, Shanghai 200241, China}

\email{jjmiao@math.ecnu.edu.cn}




\begin{abstract}
Let $\vect{X}=\{X_k\}_{k=0}^\infty$ be a sequence of compact metric spaces $X_{k}$ and $\vect{T}=\{T_k\}_{k=0}^\infty$ a sequence of continuous  mappings $T_{k}: X_{k} \to X_{k+1}$. The pair $(\vect{X},\vect{T})$ is called a nonautonomous dynamical system.

In this paper, we study measure-theoretic entropies and pressures, Bowen  and packing  topological  entropies and pressures on $(\vect{X},\vect{T})$, and we prove that they  are  invariant under equiconjugacies of nonautonomous dynamical systems. By establishing Billingsley type theorems for Bowen and packing topological pressures, we obtain their variational principles, that is, given a non-empty compact subset $K \subset X_{0}$ and an equicontinuous sequence $\vect{f}= \{f_k\}_{k=0}^\infty$  of functions $f_k : X_k\to \mathbb{R}$, we have that 
  $$
\prebow(\vect{T},\vect{f},K)=\sup\{\prelow_{\mu}(\vect{T},\vect{f}): \mu \in M(X_{0}), \mu(K)=1\},
$$
and for $\|\vect{f}\|<+\infty$ and $\prepac(\vect{T},\vect{f},K)>\|\vect{f}\|$,
$$
\prepac(\vect{T},\vect{f},K)=\sup\{\preup_{\mu}(\vect{T},\vect{f}): \mu \in M(X_{0}), \mu(K)=1\},
$$
where  $\prelow_{\mu} $ and $\preup_{\mu} $,  $\prebow $ and $\prepac$ denote measure-theoretic lower and upper pressures, Bowen and packing topological pressure, respectively. The Billingsley type theorems and variational principles for Bowen and packing topological entropies are direct consequences of the ones for Bowen and packing topological pressures.

\end{abstract}

\maketitle

\section{Introduction}\label{sect:intro}



Given a compact metric space $(X,d)$, we  write $C(X,\mathbb{R})$ for the Banach space of real-valued continuous functions defined on $X$ equipped with the sup-norm $\|\cdot\|_{\infty}$ and $M(X)$ for the set of all Borel probability measures on $X$.

\subsection{Topological dynamical systems}\label{ssect:tdsetmds}
Given a compact metric space $(X,d)$ and a continuous mapping $T:X \to X$,   the pair $(X,T)$ is called  a \emph{topological dynamical system (TDS)}.
Let $(X,\mathscr{F},\mu)$ be a probability space and let $T:(X,\mathscr{F}) \to (X,\mathscr{F})$ be  \emph{measure-preserving}, i.e., $\mu(T^{-1}A)=\mu(A)$ for all $A \in \mathscr{F}$.  We call the quadruple $(X,\mathscr{F},\mu,T)$ a \emph{measure-preserving dynamical system (MPDS)}. It is well known that the variational principles for entropies and pressures link the two systems, TDS and MPDS, together since given a TDS $(X,T)$, it is normal to construct a measure-preserving dynamical system $(X,\mathscr{F},\mu,T)$ by choosing a  Borel probability measure $\mu$ on $X$ so that $T$ is measure-preserving, where $\mu$ is  called  a \emph{$T$-invariant measure}. We write $M(X,T)$ for the set of all $T$-invariant Borel probability measures on $X$.
We refer the readers to \cite{Walters1982} for details.

In the 1950s, measure-theoretic entropy was first introduced by Kolmogorov \cite{Kolmogorov1958}  and modified to its modern form by Sinai \cite{Sinai1959} as a conjugacy invariant in measure-preserving dynamical systems.
By using Bowen metrics, Dinaburg \cite{Dinaburg1970} and Bowen \cite{Bowen1971} independently gave an analogous notion called topological entropy, which is an invariant under topological conjugacy on topological dynamical systems.

Let $h_{\mu}(T)$ and $\enttop(T)$ denote the measure-theoretic entropy and the topological entropy, respectively. Their fundamental relation is given by the following variational principle,
$$
\enttop(T)=\sup\{h_{\mu}(T): \mu \in M(X,T)\},
$$
which was first established by Goodman \cite{Goodman1971} in 1971.

The topological pressure is a non-trivial  generalization of topological entropy, which, roughly speaking, measures the orbit complexity of the iterated mapping on the potential functions. It is also a crucial invariant in topological dynamical systems.
Let $P(T,f)$ denote the topological pressure of TDS $(X,T)$ for a potential function $f \in C(X,\mathbb{R})$. Since the topological entropy may be regarded as a special case of the topological pressure, that is, $\enttop(T)= P(T,0)$, it is natural to consider the variational principle for the topological pressure.
In 1975, Walters \cite{Walters1975} proved that the variational principle also holds for the topological pressure of $(X, T)$, that is, for all $f \in C(X,\R)$,
$$
P(T,f)=\sup\Big\{h_{\mu}(T)+\int_{X}f\dif{\mu}: \mu \in M(X,T)\Big\}.
$$
We refer readers to \cite{AKM1965, Barreira2006, Bowen1973, Cao&Feng&Huang2008,Huang&Ye&Zhang2011,Mummert2006, Ruelle1972, Zhao&Cheng2014} for details and other relevant studies.

These classic results have various important applications in ergodic theory, dynamical systems, and fractal geometry; see \cite{Falconer1997,Katok&Hasselblatt1995,Pesin1997,Walters1982} for the background reading.
In particular, Bowen \cite{Bowen1973} introduced a type of the topological entropies on subsets $Z$ of $(X,T)$, denoted by $\enttopbow(T,Z)$, where the notion resembles the Hausdorff dimension of the subset $Z$.  In 1984, Pesin and Pitskel' \cite{Pesin&Pitskel1984} generalized this idea to the topological pressures on subsets $Z$ of $(X,T)$ for potential functions $f \in C(X,\mathbb{R})$, denoted by $\prebow(T,f,Z)$.
  It was also proved in \cite{Bowen1973,Pesin&Pitskel1984} that their dimension-like pressure (entropy) on the entire space $X$ coincides with the conventional topological pressure (entropy), that is, $\prebow(T,f,X)=P(T,f)$ ($\enttopbow(T,X)=\enttop(T)$).
  This aroused people's interest to extend the variational principle to the ones for the entropy and pressure on subsets.

In 1983,  Brin and Katok \cite{Brin&Katok1983} first introduced the measure-theoretic local entropy,  which led to the study of  Billingsley type theorems in dynamical systems. In 2008, Ma and Wen in \cite{Ma&Wen2008} obtained the following Billingsley type theorem for the lower local entropy.
 \begin{thm*}
  Given a TDS $(X,T)$, a Borel subset $E$ of $X$ and $s > 0$, let $\mu$ be a Borel probability measure on $X$.
  \begin{enumerate}[(1)]
    \item If $\entlow_{\mu}(T,x) \leq s$ for all $x \in E$, then $\enttopbow(T,E) \leq s$;
    \item If $\entlow_{\mu}(T,x) \geq s$ for all $x \in E$ and $\mu(E)>0$, then $\enttopbow(T,E) \geq s$.
  \end{enumerate}
\end{thm*}

Billingsley type theorems are very useful for the study of variational principles.
Ma and Wen \cite{Ma&Wen2008} applied the Billingsley type theorem above to obtain a variational inequality
relating the Bowen entropy and the measure-theoretic lower entropy.
In 2012, Feng and Huang \cite{Feng&Huang2012} introduced the measure-theoretic lower and upper entropies and the packing topological entropy  and proved the following variational principles. See Section \ref{sect:defetprop} for the definitions.
\begin{thm*}
Given a TDS $(X,T)$, if  $K \subset X$ is non-empty and compact, then
$$
    \enttopbow(T,K)=\sup\{\entlow_{\mu}(T): \mu \in M(X), \mu(K)=1\},
$$
$$
    \enttoppac(T,K)=\sup\{\entup_{\mu}(T): \mu \in M(X), \mu(K)=1\}.
$$
\end{thm*}
Later, Tang, Cheng and Zhao \cite{Tang&Cheng&Zhao2015} generalized the Billingsley type theorem and the variational principle for topological entropies to those for Pesin-Pitskel' pressure.
Recently, Zhong and Chen \cite{Zhong&Chen2023} further extended the Billingsley type theorem to the packing pressure on $(X, T)$, and  they obtained the  variational principle for the  packing pressure under certain conditions. There are many interesting results on the
 Billingsley type theorems and variational principles for  entropies and pressures,
 and we refer the readers to
\cite{Barreira1996,Huang&Yi2007,Misiurewicz1976,Ollagnier&Pinchon1982,Tang&Cheng&Zhao2015,Wang&Chen2012,Zhang2009} and references therein.

\subsection{Nonautonomous dynamical systems}\label{ssect:NDS}
Let $(\vect{X},\vect{d})=\{(X_{k},d_{k})\}_{k=0}^{\infty}$ be a sequence of compact metric spaces $X_{k}$
endowed with metrics $d_{k}$ for all integers $k\geq 0$, and let $\vect{T}=\{T_{k}\}_{k=0}^{\infty}$ be a sequence of
continuous mappings $T_{k}:X_{k} \to X_{k+1}$. We call the pair $(\vect{X},\vect{T})$ a \emph{nonautonomous dynamical system (NDS)},
and we sometimes write $(\vect{X},\vect{d},\vect{T})$ to emphasize the dependence on the metrics.
If $\vect{X}$ is constant, i.e., $X_{k}=X$ for all $k \geq 0$, we call $(\vect{X},\vect{T})$ a \emph{nonautonomous dynamical system with an identical space} and denote it by $(X,\vect{T})$.  We refer the readers to \cite{Cheban2020,Kawan2014, Kloeden&Rasmussen2011} for details.
Note that  if $\vect{X}$ and $\vect{T}$ are both constant, i.e., $X_{k}=X$ and $T_{k}=T$ for all $k \geq 0$,
then the nonautonomous dynamical system $(\vect{X},\vect{T})$ degenerates into the classic topological dynamical system $(X,T)$. Therefore, both $(X,T)$ and  $(X,\vect{T})$ may be regarded as  special cases of   $(\vect{X},\vect{T})$.

Currently, there is a large amount of literature focusing on nonautonomous dynamical systems with an identical space  $(X,\vect{T})$.
In 1996, the topological entropy of $(X,\vect{T})$ was introduced by Kolyada and Snoha \cite{Kolyada&Snoha1996}. 
In 2018, Biś \cite{Bis2018} introduced the measure-theoretic lower and upper local entropies of  $(X,\vect{T})$  with respect to  a Borel probability measure $\mu \in M(X)$, and obtained two Billingsley type theorems for the classic topological entropy.
In 2024, Nazarian Sarkooh \cite{NazarianSarkooh2024} studied the measure-theoretic lower pressure and Bowen pressure of $(X,\vect{T})$ and obtained the Billingsley type theorem and the following  variational principle for Bowen pressure (see Subsection \ref{ssect:prebppetentbow} for the definitions).
\begin{thm*}
  Given an NDS $(X,\vect{T})$ with an identical space and a non-empty compact subset $K \subset X$, let $f \in C(X,\mathbb{R})$. Then
  $$
  \prebow(\vect{T},f,K)=\sup\{\prelow_{\mu}(\vect{T},f): \mu \in M(X), \mu(K)=1\}.
  $$
\end{thm*}
We refer the readers to \cite{Cheban2020, Huang&Wen&Zeng2008,Kloeden&Rasmussen2011, Kong&Cheng&Li2015,Liu&Qiao&Xu2020,ZLXZ2012} for various related studies on $(X,\vect{T})$ and references therein.

Nonautonomous dynamical systems with an identical space $(X, \vect{T})$ may also be viewed from another aspect of  the so called ``nonautonomous dynamical systems   of continuous type'', which is also known as \emph{nonautonomous flows}; see \cite{Cheban2020,Kloeden&Rasmussen2011} for details.
  These systems are highly related to differential equations and are therefore strongly motivated by
  equations of mathematical physics. One way of studying nonautonomous flows is via approximating
  them by nonautonomous dynamical systems with an identical space.
Given a nonautonomous flow $\varphi:[0,+\infty) \times X \to X$ on a compact metric space $X$, by restricting the flow to $\mathbb{N} \times X$, one may obtain a discrete time nonautonomous dynamical system $(X,\vect{T})$.

In 2008,  Xu and Zhou \cite{Xu&Zhou2018} studied  nonautonomous flows on a compact metric space $X$, and  they introduced the measure-theoretic lower entropy with respect to  a Borel probability measure and the Bowen topological entropy and established a variational principle for the first time in the field of NDSs.
Further results for nonautonomous flows may be found in \cite{Zhang&Liu2022,Zhang&Zhu2023}.

For general nonautonomous dynamical systems $(\vect{X},\vect{T})$, there is little progress made on this topic due to the complexity of their structure. In 1999, Kolyada, etc \cite{KMS1999} studied the topological entropy for nonautonomous dynamical systems of piecewise monotone mappings on intervals.
In 2014, Kawan \cite{Kawan2014} generalized the measure-theoretic entropy to nonautonomous dynamical systems and obtained a variational inequality
that the measure-theoretic entropies are bounded above by the topological entropy. In 2015, Kawan and Latushkin \cite{Kawan&Latushkin2015} proved that the variational principle for the topological entropy holds for a  class of special symbolic dynamical systems.

In this paper, we explore the Billingsley type theorems and variational principles for the Bowen type and the packing type topological entropies and pressures on the general nonautonomous dynamical system $(\vect{X},\vect{T})$. Moreover, we show these entropies and pressures are invariant under equiconjugacies.

  \subsection{Nonautonomous dynamics in fractal geometry}

Fractal geometry is closely related to dynamical systems, and entropies and pressures are important tools to study the dimension theory of various fractal sets.  In particular,  the attractors of iterated function systems (IFSs) such as self-similar sets and self-affine sets  may also be viewed as the  repellers of expansive dynamical systems since the contractions in IFSs are local inverses of the expanding dynamics; see \cite{BHR19,DasSim17, Falconer1997,Feng23, FH08, Pesin1997} for details.
For example, given a self-similar iterated function system $\{S_i\}_{i=1}^{m}$ satisfying the strong separation condition where each $S_i:[0,1]\to[0,1]$ is given by $S_i(x)= r\cdot x+a_i\ (x \in [0,1])$ with $0<r<1$, let $F$ be the corresponding self-similar set. Let $T$ be the inverse mapping of $S_i$ on each $S_{i}([0,1])$. Then  $(F, T)$ is an expansive topological dynamical system.  We have the following equality
\begin{equation}\label{eq:dimetentsss}
\dimh{F}=\dimbox{F}=\frac{\enttop(T)}{\log{r}}=\frac{\enttopbow(T,F)}{\log{r}}
\end{equation}
where $\dimh$  and $\dimbox$ denote Hausdorff  and  box dimensions, respectively. See \cite{Falconer2014,Hutch81, Pesin1997}  for details.

There is a typical generalization of IFSs called \emph{nonautonomous iterated function systems}, and their attractors are called \emph{nonautonomous fractals} (see~\cite{GM}). The nonautonomous iterated function systems have been well studied in fractal geometry. A special case of such fractals is called Moran sets, which was first studied by Moran~\cite{Moran} in 1946. The dimension theory of nonautonomous fractals has been studied extensively, and we refer the readers to~\cite{Gu&Miao2022,GM,HRWW2000,R-G&Urbanski2012, Wen00, Wen01} for various studies. Since different iterated function systems are applied at each iterating step in the construction of nonautonomous fractals, the number of basic intervals as well as the contraction ratios may vary according to the systems at each step. Therefore the geometric structures of nonautonomous fractals are more complex than self-similar sets, and in general, the following dimension inequality holds
$$
\dimh F \leq  \dimp F \leq {\overline{\dim}}_{\rm B} F,
$$
where $\dimp$ and ${\overline{\dim}}_{\rm B}$ denote the packing and upper box dimensions, respectively. Note that the inequality may hold strictly for nonautonomous fractals $F$.

Inspired by recent progress in nonautonomous iterated function systems (see \cite{Gu&Miao2022,GM, R-G&Urbanski2012}), it is reasonable to construct the following nonautonomous dynamical systems as a counterpart of nonautonomous iterated function systems. Let $I \subset \mathbb{R}^{n}$ be a compact set with non-empty interior. Given a sequence of positive integers $\{m_k\}_{k=1}^\infty$ where $m_k\geq2$ for each integer $k>0$, suppose that $I_{k,i} \subset I\ (i=1,2,\ldots,m_{k})$ are $m_{k}$ pairwise disjoint compact subsets, and that $S_{k,i}: I \to I_{k,i}$ is a surjective contraction for each $i=1,2, \ldots, m_k$. We write $\Xi_{k}=\{S_{k,1},S_{k,2},\cdots,S_{k,m_{k}}\}$. Then  $\Xi=\{\Xi_{k}\}_{k=1}^{\infty}$   forms a nonautonomous iterated function system (NIFS), and it has a nonautonomous attractor $F$ given by
     $$
    F=\bigcap_{k=1}^{\infty}\bigcup_{\substack{1 \leq i_{j} \leq m_{j} \\ 1 \leq j \leq k}}S_{i_{1}} \circ \cdots \circ S_{i_{k}}(I).
    $$
 See \cite{GM,Wen00} for details.

Similarly, we may give an alternative description from a dynamical point of view.  For each $k \in \mathbb{N} $, let $T_{k}$ be the continuous mapping from $\bigcup_{i=1}^{m_{k+1}}I_{k+1,i}$ onto $I$ defined by
    $$
    T_{k}(x)=S_{k+1,i}^{-1}(x),        \qquad (x \in I_{k+1,i})
    $$
 for $i=1,2,\cdots,m_{k+1}$. We write  $\vect{T} =\{T_{k}\}_{k=0}^{\infty}$.  We may obtain $F$ by using $\vect{T}$, that is,
    $$
F=\bigcap_{k=0}^{\infty}\vect{T}^{-k}(I) =\bigcap_{k=0}^{\infty}(T_{k}\circ\ldots\circ T_0)^{ -1}(I),
$$
and  we call  $F$ the \emph{repeller} of  $\vect{T}$ in this case. Note that every  exterior point of $F$   falls out of $\bigcup_{i=1}^{m_{k}}I_{k,i}$ at some  $k$-th  iteration, and  the   iterating process for the point  stops  at this step. To define the corresponding  nonautonomous dynamical system,   we restrict $\vect{T} $ to the repeller $F$ and write  $X_k = \vect{T}^{k}F$. Therefore  $(\vect{X} ,\vect{T} )$ is a nonautonomous dynamical system, where $\vect{X} =\{\vect{T}^{k}F\}_{k=0}^{\infty}$.     Note that it may not   necessarily happen that  $F =\vect{T}^{k}F$.  Therefore it is interesting to  understand the interaction between nonautonomous iterated functions systems and nonautonomous dynamical systems.

Inspired by these previous works, we study the properties of entropies and pressures for nonautonomous dynamical systems  $(\vect{X} ,\vect{T} )$ and obtain various Billingsley type theorems and variational principles for these quantities in this paper. In Section \ref{sect:defetprop}, we give the  definitions of different types of entropies and pressures and state our main conclusions on   Billingsley type theorems (Theorems \ref{thm:btypethmbowen}, \ref{thm:btypethmpacking} and their corollaries)
  and  variational principles (Theorems \ref{thm:varprinciplebowen}, \ref{thm:varprinciplepacking} and their corollaries) for these entropies and pressures.
To study  these   Billingsley type theorems and  variational principles, we  explore the properties of these entropies and pressures in Section \ref{sect:proptoppreetent}. Next, we prove that  these entropies and pressures are  invariant under equiconjugacies in Section \ref{sect:invarequiconj}, and we obtain a consequence that the Bowen pressure and packing pressure are `topological', that is, they are metric irrelevant if the metrics endowed are uniformly equivalent.
In Section \ref{sect:pfbthms}, the Billingsley type theorems are proven, and we give the proofs  for the variational principles in Section \ref{sect:pfvps}.

\section{Entropies, Pressures and Main Conclusions}\label{sect:defetprop}
In the rest of the paper, we always write $(\vect{X},\vect{T})$ for a  nonautonomous dynamical system (NDS), that is,
$ \vect{X}=\{ X_{k}  \}_{k=0}^{\infty} $ is  a sequence of compact metric spaces $X_{k}$, and   $\vect{T}=\{T_{k}\}_{k=0}^{\infty}$ is a sequence of
continuous mappings $T_{k}:X_{k} \to X_{k+1}$. We use  $(X,\vect{T})$ for a nonautonomous dynamical system with an identical space and $(X,T)$ for the classic topological dynamical system.

  \subsection{Bowen Metrics and Bowen Balls}\label{ssect:BB}

In this subsection, we   introduce the Bowen metrics and Bowen balls which are the key essence in the study of  nonautonomous dynamical systems.

Given an NDS $(\vect{X},\vect{d},\vect{T})$, for each integer $k\geq 0$, we write
  $$
  \vect{T}_{k}^{j}=T_{k+(j-1)} \circ \cdots \circ T_{k}: X_{k} \to X_{k+j}
  $$
for $ j = 1,2,3,\cdots$, and we adopt the convention that  $ \vect{T}_{k}^{0}=\id_{X_{k}}$
where $\id_{X_{k}}:X_{k} \to X_{k}$ is the identity mapping.
Since the mappings $T_{k}$ are not necessarily bijective, we write $\vect{T}_{k}^{-j} = (\vect{T}_{k}^{j})^{-1}$ for the preimage of subsets of $X_{k+j}$ under $\vect{T}_{k}^{j}$.
If the mappings $T_{k},\cdots,T_{k+(j-1)}$ are bijective, we also use $\vect{T}_{k}^{-j} : X_{k+j} \to X_{k}$ for the inverse of $\vect{T}_{k}^{j}$, that is,
\begin{equation}\label{def_TInv}
\vect{T}_{k}^{-j}=(T_{k+(j-1)} \circ \cdots \circ T_{k})^{-1}= T_{k}^{-1} \circ \cdots \circ T_{k+(j-1)}^{-1}.
\end{equation}
For simplicity, we often write $\vect{T}^{j}=\vect{T}_{0}^{j}$ for $j \in \mathbb{Z}$.

Given $k \in \mathbb{N}$, for $n=1,2,3,\ldots,$ we define the {\emph{$n$-th Bowen metric at level $k$}} or the {\emph{$n$-th Bowen metric} on $X_{k}$} by
$$
d_{k,n}^{\vect{T}}(x,y)=\max_{0 \leq j \leq n-1}d_{k+j}(\vect{T}_{k}^{j}x,\vect{T}_{k}^{j}y)
$$
for all $x,y \in X_{k}$. It is routine to check that each $d_{k,n}^{\vect{T}}$ is a metric on $X_{k}$ topologically equivalent to $d_{k}$ for every $k\in \mathbb{N}$.

 Given $\varepsilon>0$ and $x \in X_{k}$, let $B_{k,n}^{\vect{T}}(x,\varepsilon)$ denote the
\emph{(open) Bowen ball} with center  $x$ and  radius $\varepsilon$ at level $k$, i.e.,
$$
B_{k,n}^{\vect{T}}(x,\varepsilon)=\{y \in X_{k}: d_{k,n}^{\vect{T}}(x,y)<\varepsilon\},
$$
and let $\overline{B}_{k,n}^{\vect{T}}(x,\varepsilon)$ denote the \emph{closed Bowen ball} with
 center $x$ and  radius $\varepsilon$ at level $k$.

Alternatively, the open and closed Bowen balls may be  given by
  \begin{equation}\label{eq:bowenballopen}
    B_{k,n}^{\vect{T}}(x,\varepsilon)=\bigcap_{j=0}^{n-1}\vect{T}_{k}^{-j}B_{X_{k+j}}(\vect{T}_{k}^{j}x,\varepsilon),
 \qquad
    \overline{B}_{k,n}^{\vect{T}}(x,\varepsilon)=\bigcap_{j=0}^{n-1}\vect{T}_{k}^{-j}\overline{B}_{X_{k+j}}(\vect{T}_{k}^{j}x,\varepsilon).
  \end{equation}

To study the Billingsley type theorems and variational principles of pressures, we need potentials for pressures. Let
$$
\vect{C}(\vect{X},\mathbb{R})=\prod_{k=0}^{\infty}C(X_{k},\mathbb{R})
$$
be the collection of all sequences of continuous functions $f_k: X_k \to \mathbb{R}$.  We often write $\vect{0}$ and $\vect{1} \in \vect{C}(\vect{X},\mathbb{R})$ for the sequence of constant $0$ functions and constant $1$ functions, respectively, and $a\vect{1} \in \vect{C}(\vect{X},\mathbb{R})$ for the sequence of constant $a$ functions where $a \in \mathbb{R}$.
  Given $\vect{f}=\{f_{k}\}_{k=0}^{\infty}$ and $\vect{g}=\{g_{k}\}_{k=0}^{\infty} \in \vect{C}(\vect{X},\mathbb{R})$, we write $\vect{f} \preceq \vect{g}$
  if $f_{k} \leq g_{k}$ for all $k \in \mathbb{N}$.

 Given $\vect{f}\in \vect{C}(\vect{X},\mathbb{R})$, we write
\begin{equation}\label{eq_fnorm}
\|\vect{f}\|=\sup_{k \in \mathbb{N}}\left\{\|f_{k}\|_{\infty}=\max_{x \in X_{k}}\lvert f_{k}(x) \rvert\right\}.
\end{equation}
  It is clear that $\|\vect{f}\|<+\infty$ implies that $\vect{f}$ is uniformly bounded. Let $\vect{C}_{b}(\vect{X},\mathbb{R})$ denote the collection of all uniformly bounded function sequences  in $\vect{C}(\vect{X},\mathbb{R})$, i.e.,
  $$\vect{C}_{b}(\vect{X},\mathbb{R})=\{\vect{f}\in \vect{C}(\vect{X},\mathbb{R}): \|\vect{f}\|<+\infty\}.$$

Given $\vect{f} \in \vect{C}(\vect{X},\mathbb{R})$, we say $\vect{f}$ is \emph{equicontinuous} if for every $\varepsilon>0$,
  there exists $\delta>0$ such that for all $k \in \mathbb{N}$ and all $x^{\prime},x^{\prime\prime} \in X_{k}$ satisfying
  $d_{k}(x^{\prime},x^{\prime\prime})<\delta$, we have that
  $$
  \lvert f_{k}(x^{\prime})-f_{k}(x^{\prime\prime}) \rvert<\varepsilon.
  $$
Note that if $\vect{X}$ is constant, i.e., $X_k=X$ for all $k \in \mathbb{N}$, then by the compactness of $X$, the equicontinuity of $\vect{f}$ coincides with the conventional definition of equicontinuity.
  In particular, for the dynamical systems  with an identical space $(X,\vect{T})$ and the TDSs $(X,T)$, we usually require $f_{k}=f$ for all $k \in \N$ where $f \in C(X,\R)$, in which case $\vect{f}=\{f\}_{k=0}^{\infty}$ is clearly equicontinuous.

Given $\vect{f} \in \vect{C}(\vect{X},\mathbb{R})$, for  $k, n \in \mathbb{N}$, we write
  \begin{equation}\label{eq:sumknTf}
    S_{k,n}^{\vect{T}}\vect{f}=\sum_{j=0}^{n-1} f_{k+j} \circ \vect{T}_{k}^{j}.
  \end{equation}
Note that each $S_{k,n}^{\vect{T}}\vect{f}$ is a continuous real-valued function on $X_{k}$.

For $k \in \mathbb{N}$, write $\vect{X}_{k}=\{X_{j}\}_{j=k}^{\infty}$  and $\vect{T}_{k}=\{T_{j}\}_{j=k}^{\infty}$. Then $(\vect{X}_{k},\vect{T}_{k})$ is  also a nonautonomous dynamical system. Note that  $(\vect{X}_{0},\vect{T}_{0})$ is the original system $(\vect{X},\vect{T})$. For $k>0$,  we write $\vect{X}'=\vect{X}_{k}=\{X_{j+k}\}_{j=0}^{\infty}$ and $\vect{T}'=\vect{T}_{k}=\{T_{j+k}\}_{j=0}^{\infty}$, and  it is sufficient to consider the new NDS $(\vect{X}',\vect{T}')$ with initial space at index $0$. 

Therefore, in this paper,  we only   focus on the NDS $(\vect{X},\vect{T})$ with initial space starting from $k=0$, that is,  $\vect{X} =\{X_{k}\}_{k=0}^{\infty}$  and $\vect{T} =\{T_k\}_{k=0}^\infty.$
For simplicity, we write $d_{n}^{\vect{T}}=d_{0,n}^{\vect{T}}$,
$$
B_{n}^{\vect{T}}(x,\varepsilon)=B_{0,n}^{\vect{T}}(x,\varepsilon), \qquad
 \overline{B}_{n}^{\vect{T}}(x,\varepsilon)=\overline{B}_{0,n}^{\vect{T}}(x,\varepsilon),
$$
 and
  \begin{equation}\label{eq:sum0nTf}
    S_{n}^{\vect{T}}\vect{f}=S_{0,n}^{\vect{T}}\vect{f}=\sum_{j=0}^{n-1} f_{j} \circ \vect{T}^{j}.
  \end{equation}

  \subsection{Measure-theoretic entropies and pressures}\label{ssect:msrentetpre}
  Given an NDS $(\vect{X},\vect{T})$, let $\mu$ be a Borel probability measure on $X_{0}$. Following the idea of Brin and Katok \cite{Brin&Katok1983}, we define the measure-theoretic entropies and pressures of the NDS $(\vect{X},\vect{T})$. Note that only the measurability of the mappings of $\vect{T}$ is required in the definition of measure-theoretic entropies.

  \begin{defn}\label{def:msrent}
    Given $\mu \in M(X_{0})$ and $x \in X_{0}$, the \emph{measure-theoretic lower}
    and \emph{upper local entropies of $\vect{T}$ with respect to $\mu$ at $x$} are defined respectively by
    $$
    \entlow_{\mu}(\vect{T},x)=\lim_{\varepsilon \to 0}\entlow_{\mu}(\vect{T},x,\varepsilon),\qquad \entup_{\mu}(\vect{T},x)=\lim_{\varepsilon \to 0}\entup_{\mu}(\vect{T},x,\varepsilon),
    $$
    where for all $\varepsilon>0$,
    $$
    \entlow_{\mu}(\vect{T},x,\varepsilon)=\varliminf_{n \to \infty}\frac{-\log{\mu(B_{n}^{\vect{T}}(x,\varepsilon))}}{n},\qquad \entup_{\mu}(\vect{T},x,\varepsilon)=\varlimsup_{n \to \infty}\frac{-\log{\mu(B_{n}^{\vect{T}}(x,\varepsilon))}}{n}.
    $$
Note that we set  $\log{0}=-\infty$ if $\mu(B_{n}^{\vect{T}}(x,\varepsilon))=0$. The \emph{measure-theoretic lower} and \emph{upper entropies of $\vect{T}$ with respect to $\mu$} are
    defined respectively by
    $$\entlow_{\mu}(\vect{T})=\int_{X_{0}}\entlow_{\mu}(\vect{T},x)\dif{\mu(x)}, \qquad \entup_{\mu}(\vect{T})=\int_{X_{0}}\entup_{\mu}(\vect{T},x)\dif{\mu(x)}.$$
  \end{defn}

The measure-theoretic pressures of $\vect{T}$ are generalizations of measure-theoretic entropies, which are  mappings  from $\vect{C}(\vect{X},\mathbb{R})$ to $\mathbb{R}\cup \{\pm\infty\}$.

  \begin{defn}\label{def:msrpre}
    Given $\mu \in M(X_{0})$ and $\vect{f} \in \vect{C}(\vect{X},\mathbb{R})$, for each $x \in X_{0}$, the \emph{measure-theoretic lower} and \emph{upper local pressures of  $\vect{T}$ for $\vect{f}$  with respect to   $\mu$ at $x$}
    are defined respectively by
    $$\prelow_{\mu}(\vect{T},\vect{f},x)=\lim_{\varepsilon \to 0}\prelow_{\mu}(\vect{T},\vect{f},x,\varepsilon),\qquad \preup_{\mu}(\vect{T},\vect{f},x)=\lim_{\varepsilon \to 0}\preup_{\mu}(\vect{T},\vect{f},x,\varepsilon),$$
    where for all $\varepsilon>0$,
    $$\prelow_{\mu}(\vect{T},\vect{f},x,\varepsilon)=\varliminf_{n \to \infty}\frac{-\log{\mu(B_{n}^{\vect{T}}(x,\varepsilon))}+S_{n}^{\vect{T}}\vect{f}(x)}{n},$$
    $$\preup_{\mu}(\vect{T},\vect{f},x,\varepsilon)=\varlimsup_{n \to \infty}\frac{-\log{\mu(B_{n}^{\vect{T}}(x,\varepsilon))}+S_{n}^{\vect{T}}\vect{f}(x)}{n}.$$
 The \emph{measure-theoretic lower} and \emph{upper pressures of $\vect{T}$ for $\vect{f}$  with respect to $\mu$} are  defined respectively by
    $$\prelow_{\mu}(\vect{T},\vect{f})=\int_{X_{0}}\prelow_{\mu}(\vect{T},\vect{f},x)\dif{\mu(x)},\qquad \preup_{\mu}(\vect{T},\vect{f})=\int_{X_{0}}\preup_{\mu}(\vect{T},\vect{f},x)\dif{\mu(x)}.$$
  \end{defn}
As we can see from the definition, the measure-theoretic entropies are a special case of  the measure theoretic pressures, that is,
\begin{equation}\label{hupre}
\entlow_{\mu}(\vect{T})=\prelow_{\mu}(\vect{T},\vect{0}), \qquad \entup_{\mu}(\vect{T})=\preup_{\mu}(\vect{T},\vect{0}).
\end{equation}

Conversely, the measure-theoretic pressures may be represented in terms of entropies under certain conditions.
  \begin{prop}\label{prop:msrpreetent}
    Let  $\vect{f} \in \vect{C}(\vect{X},\mathbb{R})$ and $\mu \in M(X_{0})$.
    For $x \in X_{0}$, suppose additionally that one of the following is true:
    \begin{enumerate}[(a)]
      \item $\entlow_{\mu}(\vect{T},x)=\entup_{\mu}(\vect{T},x)$;
      \item $f_{k} \circ \vect{T}^{k}(x)$ converges to some real number $a \in \R$.
    \end{enumerate}
    Then
            $$\prelow_{\mu}(\vect{T},\vect{f},x)=\entlow_{\mu}(\vect{T},x)+\varliminf_{n \to \infty}\frac{1}{n}S_{n}^{\vect{T}}\vect{f}(x),$$
            $$\preup_{\mu}(\vect{T},\vect{f},x)=\entup_{\mu}(\vect{T},x)+\varlimsup_{n \to \infty}\frac{1}{n}S_{n}^{\vect{T}}\vect{f}(x).$$
    
    Moreover, if either $\entlow_{\mu}(\vect{T},x)=\entup_{\mu}(\vect{T},x)$ for $\mu$-a.e. $x \in X_{0}$ or $f_{k} \circ \vect{T}^{k} \to f$ $\mu$-a.e. for some $f \in L^{1}(X_{0},\R)$, then
            $$\prelow_{\mu}(\vect{T},\vect{f})=\entlow_{\mu}(\vect{T})+\int_{X_{0}}\varliminf_{n \to \infty}\frac{1}{n}S_{n}^{\vect{T}}\vect{f}\dif{\mu},$$
            $$\preup_{\mu}(\vect{T},\vect{f})=\entup_{\mu}(\vect{T})+\int_{X_{0}}\varlimsup_{n \to \infty}\frac{1}{n}S_{n}^{\vect{T}}\vect{f}\dif{\mu}.$$
  \end{prop}

  \begin{rmk}
    If $(\vect{X},\vect{T})$ degenerates into a topological dynamical system $(X,T)$,  i.e., $X_{k}=X$ and $T_{k}=T$ for all $k \geq 0$, and if $\mu$ is a $T$-invariant Borel probability measure,  then
    $\entlow_{\mu}(\vect{T},x)=\entup_{\mu}(\vect{T},x)=h_{\mu}(T,x)$
        for $\mu$-a.e. $x \in X$, and
        $$\entlow_{\mu}(\vect{T})=\entup_{\mu}(\vect{T})=h_{\mu}(T),$$
        where $h_{\mu}(T)$ is the classic measure-theoretic entropy; see \cite{Brin&Katok1983, Walters1982} for details.
        In this case, for every $\vect{f}=\{f_k\}_{k=0}^\infty$ such that $f_k=f$ for all $k\geq0 $ where $f \in C(X, \mathbb{R})$,
        combining Proposition \ref{prop:msrpreetent} with Birkhoff's ergodic theorem (see \cite[Thm.1.14]{Walters1982}), we have that
    $$
    \prelow_{\mu}(\vect{T},\vect{f})=\preup_{\mu}(\vect{T},\vect{f})=h_{\mu}(T)+\int_{X}f\dif{\mu}.
    $$
    \end{rmk}

The measure-theoretic entropies and pressures are crucial to the study of topological entropies and pressures, and the following simple facts are frequently used in our proofs. The existence of local pressures and entropies is also guaranteed by these facts.
  \begin{prop}\label{prop:localepsilon}
    Given $\vect{f} \in \vect{C}(\vect{X},\mathbb{R})$,
    \begin{enumerate}[(1)]
      \item For $\varepsilon_{1}>\varepsilon_{2}>0$, we have that
            $$
            \prelow_{\mu}(\vect{T},\vect{f},x,\varepsilon_{1}) \leq \prelow_{\mu}(\vect{T},\vect{f},x,\varepsilon_{2}),\quad \quad \preup_{\mu}(\vect{T},\vect{f},x,\varepsilon_{1}) \leq \preup_{\mu}(\vect{T},\vect{f},x,\varepsilon_{2})
            $$
 for all $x \in X_{0}$. 

      \item For each $x \in X_{0}$, we have that
            $$\prelow_{\mu}(\vect{T},\vect{f},x)=\sup_{\varepsilon > 0}\prelow_{\mu}(\vect{T},\vect{f},x,\varepsilon), \quad   \quad \preup_{\mu}(\vect{T},\vect{f},x)=\sup_{\varepsilon > 0}\preup_{\mu}(\vect{T},\vect{f},x,\varepsilon).$$

    \end{enumerate}
  \end{prop}

Note that the measure-theoretic entropies have the similar conclusion to Proposition \ref{prop:localepsilon}.

  \subsection{Bowen pressure and entropy}\label{ssect:prebppetentbow}
In \cite{Pesin&Pitskel1984}, Pesin and Pitskel' used an approach similar to Hausdorff measures and dimensions to define a type of pressure on topological dynamical systems, and they studied the relevant variational principle. Inspired by their work, we define a type of topological pressures on nonautonomous dynamical systems by constructing certain Hausdorff measures with  Bowen balls; see \cite{Falconer2014,Rogers1998} for details of Hausdorff measures.

Given an NDS $(\vect{X},\vect{T})$ and a subset $Z \subset X_{0}$, we say that a collection $\{B_{n_{i}}^{\vect{T}}(x_{i},\varepsilon)\}_{i\in \mathcal{I}}$ of Bowen balls is a \textit{$(N,\varepsilon)$-cover of $Z$} if $\bigcup_{i\in \mathcal{I}}B_{n_{i}}^{\vect{T}}(x_{i},\varepsilon) \supseteq Z$ where $n_{i} \geq N$ for each $i \in \mathcal{I}$ .

Given  $\vect{f} \in \vect{C}(\vect{X},\mathbb{R})$ and $s \in \mathbb{R}$, for  reals $N >0$ and $\varepsilon>0$, we define
\begin{equation}\label{eq:defmsrbow}
\msrbow_{N,\varepsilon}^{s}(\vect{T},\vect{f},Z)=\inf\Big\{\sum_{i=1}^{\infty} \exp{\big(-n_{i}s+S_{n_{i}}^{\vect{T}}\vect{f}(x_{i})\big)}  \Big\},
\end{equation}
where the infimum is taken over all countable $(N,\varepsilon)$-covers $\{B_{n_{i}}^{\vect{T}}(x_{i},\varepsilon)\}_{i=1}^\infty$ of $Z$.
Note that $\msrbow_{N,\varepsilon}^{s}(\vect{T},\vect{f}, \cdot)$ is an outer measure on $X_0$ constructed by Method I in \cite{Rogers1998}, and  such measures are not Borel regular. This causes one of the main difficulties in studying the variational principles for the pressures.

Since $\msrbow_{N,\varepsilon}^{s}(\vect{T},\vect{f},Z)$ increases as $N$ tends to $\infty$ for every given $Z \subset X_{0}$, we write
$$
\msrbow_{\varepsilon}^{s}(\vect{T},\vect{f},Z)=\lim_{N \to \infty} \msrbow_{N,\varepsilon}^{s}(\vect{T},\vect{f},Z).
$$
Note that if $t > s$, then $\msrbow_{\varepsilon}^{t}(\vect{T},\vect{f},Z)=0$ whenever $\msrbow_{\varepsilon}^{s}(\vect{T},\vect{f},Z)<\infty$. Thus, there is a critical value of $s$ at which $\msrbow_{\varepsilon}^{s}(\vect{T},\vect{f},Z)$ `jumps' from $\infty$ to $0$. Formally, the critical value is denoted by
\begin{equation}\label{def_PBep}
    \prebow(\vect{T},\vect{f},Z,\varepsilon)=\inf\{s: \msrbow_{\varepsilon}^{s}(\vect{T},\vect{f},Z)=0\}
                                            =\sup\{s: \msrbow_{\varepsilon}^{s}(\vect{T},\vect{f},Z)=+\infty\}.
\end{equation}

The following  facts  are  direct  consequences of  the  definitions.
  \begin{prop}\label{prop:msrbowepsilon}
Given $\vect{f} \in \vect{C}(\vect{X},\mathbb{R})$, $Z \subset X_{0}$ and $s \in \mathbb{R}$, for  $\varepsilon_{1}>\varepsilon_{2}>0$, we have that
    $$
    \msrbow_{N,\varepsilon_{1}}^{s}(\vect{T},\vect{f},Z) \leq \msrbow_{N,\varepsilon_{2}}^{s}(\vect{T},\vect{f},Z)
    $$
    for all $N >0$, and
    $$
    \msrbow_{\varepsilon_{1}}^{s}(\vect{T},\vect{f},Z) \leq \msrbow_{\varepsilon_{2}}^{s}(\vect{T},\vect{f},Z).
    $$
Moreover
    $$\prebow(\vect{T},\vect{f},Z,\varepsilon_{1}) \leq \prebow(\vect{T},\vect{f},Z,\varepsilon_{2}).$$
  \end{prop}

By the monotonicity of $\prebow(\vect{T},\vect{f},Z,\varepsilon)$ with respect to  $\varepsilon$, we are able to define the Bowen type of topological pressure on nonautonomous dynamical systems.
  \begin{defn}\label{def:prebpp}
    Given an NDS $(\vect{X},\vect{T})$, $\vect{f} \in \vect{C}(\vect{X},\mathbb{R})$  and $Z \subset X_{0}$, we call
$$
\prebow(\vect{T},\vect{f},Z)=\lim_{\varepsilon \to 0}\prebow(\vect{T},\vect{f},Z,\varepsilon)
$$
 the \emph{Bowen-Pesin-Pitskel' topological pressure} (\emph{Bowen pressure} for short) \emph{of $\vect{T}$ for $\vect{f}$ on $Z$}. We call
$$
\enttopbow(\vect{T},Z)=\prebow(\vect{T},\vect{0},Z)
$$
the \emph{Bowen topological entropy} (\emph{Bowen entropy} for short)  \emph{of $\vect{T}$  on $Z$},   where $\vect{0}$ is the sequence of zero functions.
  \end{defn}

The following Billingsley type theorem reveals the connection between the measure-theoretic lower local pressures and the Bowen pressure, which is also very useful in the study of variational principle for the Bowen pressure.
  \begin{thm}\label{thm:btypethmbowen}
    Given an NDS $(\vect{X},\vect{T})$, a Borel  $E\subset X_{0}$ and $s \in \mathbb{R}$,
let $\mu$ be a Borel probability measure on $X_{0}$ and  $\vect{f} \in \vect{C}(\vect{X},\mathbb{R})$.
    \begin{enumerate}[(1)]
      \item If $\prelow_{\mu}(\vect{T},\vect{f},x) \leq s$ for all $x \in E$, then $\prebow(\vect{T},\vect{f},E) \leq s$;
      \item If $\vect{f}$ is equicontinuous and  $\prelow_{\mu}(\vect{T},\vect{f},x) \geq s$ for all $x \in E$ and $\mu(E)>0$, then $\prebow(\vect{T},\vect{f},E) \geq s$.
    \end{enumerate}
  \end{thm}
Since the Bowen entropy is given by the pressure for $\vect{f}=\vect{0}$, we immediately have a Billingsley type conclusion on the Bowen entropy.
  \begin{cor}\label{cor:btypethmbowen}
    Given an NDS $(\vect{X},\vect{T})$,  a Borel  $E\subset X_{0}$  and $s > 0$, let $\mu$ be a Borel probability measure on $X_{0}$.
    \begin{enumerate}[(1)]
      \item If $\entlow_{\mu}(\vect{T},x) \leq s$ for all $x \in E$, then $\enttopbow(\vect{T},E) \leq s$;
      \item If $\entlow_{\mu}(\vect{T},x) \geq s$ for all $x \in E$ and $\mu(E)>0$, then $\enttopbow(\vect{T},E) \geq s$.
    \end{enumerate}
  \end{cor}

Next, we state our first main conclusion on the variational principle for the Bowen pressure.
  \begin{thm}\label{thm:varprinciplebowen}
      Given an NDS $(\vect{X},\vect{T})$ and a non-empty compact $K \subset X_{0}$, let $\vect{f} \in \vect{C}(\vect{X},\mathbb{R})$ be equicontinuous. Then
      $$\prebow(\vect{T},\vect{f},K)=\sup\{\prelow_{\mu}(\vect{T},\vect{f}): \mu \in M(X_{0}), \mu(K)=1\}.$$
  \end{thm}

The following  variational principle for the Bowen entropy is a direct consequence of Theorem \ref{thm:varprinciplebowen}.
  \begin{cor}\label{cor:varprinciplebowen}
    Given an NDS $(\vect{X},\vect{T})$ and a non-empty compact  $K \subset X_{0}$,
    $$\enttopbow(\vect{T},K)=\sup\{\entlow_{\mu}(\vect{T}): \mu \in M(X_{0}), \mu(K)=1\}.$$
  \end{cor}

  \subsection{Packing  pressure and entropy}\label{ssect:prepacetentpac}
 Coverings and packings play a dual role in many areas of mathematics.
Bowen pressure is defined in terms of Hausdorff measures by using covers, and there is another typical measure construction by using packings, namely, packing measure,
that is in a sense `dual' to Hausdorff measure. Hausdorff and  packing measures and dimensions are fundamental concepts in fractal geometry \cite{Falconer2014,Tricot1982}.
Thus, it is natural to look for a type of pressure defined in terms of `packings' by large collections of disjoint Bowen balls of large iterations and small
radii with centres in the set under consideration. Such analogues of entropy and pressure have been considered on topological dynamical systems; see \cite{Feng&Huang2012,Wang&Chen2012,Zhong&Chen2023}. In this subsection, we extend these notions to nonautonomous dynamical systems.

Given an NDS $(\vect{X},\vect{T})$ and a subset $Z \subset X_{0}$, we say that a collection $\{\overline{B}_{n_{i}}^{\vect{T}}(x_{i},\varepsilon)\}_{i\in \mathcal{I}}$ of closed Bowen balls is a \textit{$(N,\varepsilon)$-packing of $Z$} if $\{\overline{B}_{n_{i}}^{\vect{T}}(x_{i},\varepsilon)\}_{i\in \mathcal{I}}$ is disjoint where $x_{i} \in Z$ and $n_{i} \geq N$ for all $i \in \mathcal{I}$.

Given $\vect{f} \in \vect{C}(\vect{X},\R)$ and $s \in \mathbb{R}$, for reals $N>0$ and $\varepsilon>0$, we define
  \begin{equation}\label{eq:defmsrPack}
      \msrpac_{N,\varepsilon}^{s}(\vect{T},\vect{f},Z)=\sup\Big\{\sum_{i=1}^{\infty}\exp{\left(-n_{i}s+S_{n_{i}}^{\vect{T}}\vect{f}(x_{i})\right)} \Big\},
  \end{equation}
where the supremum is taken over all countable $(N,\varepsilon)$-packings $\{\overline{B}_{n_{i}}^{\vect{T}}(x_{i},\varepsilon)\}_{i=1}^\infty$ of $Z$. Since $\msrpac_{N,\varepsilon}^{s}(\vect{T},\vect{f},Z)$ is non-increasing as $N$ tends to $\infty$,
 we write
$$
\msrpac_{\infty,\varepsilon}^{s}(\vect{T},\vect{f},Z)=\lim_{N \to \infty} \msrpac_{N,\varepsilon}^{s}(\vect{T},\vect{f},Z).
$$
Note that $\msrpac_{\infty,\varepsilon}^{s}$ is not a measure, of which the problem is similar to that encountered with the classic packing measures; see \cite{Falconer2014}. Hence, we modify the definition by decomposing $Z$ into a countable collection of sets and define
\begin{equation}\label{def_pes}
\msrpac_{\varepsilon}^{s}(\vect{T},\vect{f},Z)=\inf\Big\{\sum_{i=1}^{\infty}\msrpac_{\infty,\varepsilon}^{s}(\vect{T},\vect{f},Z_{i}): \bigcup_{i=1}^{\infty}Z_{i} \supseteq Z\Big\}.
\end{equation}
Similarly, we denote the jump value of $s$  by
\begin{equation}\label{def_PP}
\prepac(\vect{T},\vect{f},Z,\varepsilon) =\inf\{s:\msrpac_{\varepsilon}^{s}(\vect{T},\vect{f},Z)=0\}
                         =\sup\{s: \msrpac_{\varepsilon}^{s}(\vect{T},\vect{f},Z)=+\infty\}.
\end{equation}

The following monotone properties are straightforward from the definitions.
  \begin{prop}\label{prop:msrpacepsilon}
    Given $\vect{f} \in \vect{C}(\vect{X},\mathbb{R})$, $Z \subset X_{0}$ and $s \in \mathbb{R}$, for $\varepsilon_{1}>\varepsilon_{2}>0$, we have that
    $$
    \msrpac_{N,\varepsilon_{1}}^{s}(\vect{T},\vect{f},Z) \leq \msrpac_{N,\varepsilon_{2}}^{s}(\vect{T},\vect{f},Z)
    $$
for all $N >0$, and
    $$
    \msrpac_{\varepsilon_{1}}^{s}(\vect{T},\vect{f},Z) \leq \msrpac_{\varepsilon_{2}}^{s}(\vect{T},\vect{f},Z).
    $$
Moreover
    $$
\prepac(\vect{T},\vect{f},Z,\varepsilon_{1}) \leq \prepac(\vect{T},\vect{f},Z,\varepsilon_{2}).
$$
  \end{prop}
Since  $\prepac(\vect{T},\vect{f},Z,\varepsilon)$ is monotone with respect to  $\varepsilon$, we may define the packing pressure as follows.
  \begin{defn}\label{def:prepac}
Given an NDS $(\vect{X},\vect{T})$, $\vect{f} \in \vect{C}(\vect{X},\mathbb{R})$ and $Z \subset X_{0}$, we define the \emph{packing topological pressure} (\emph{packing pressure} for short) \emph{of $\vect{T}$ for $\vect{f}$ on  $Z$} by
    $$\prepac(\vect{T},\vect{f},Z)=\lim_{\varepsilon \to 0}\prepac(\vect{T},\vect{f},Z,\varepsilon).$$
 We define the \emph{packing topological entropy} (\emph{packing entropy} for short) \emph{of $\vect{T}$ on $Z$} by
    $$\enttoppac(\vect{T},Z)=\prepac(\vect{T},\vect{0},Z),$$
    where $\vect{0}$ is the sequence of zero functions.
  \end{defn}

We also have a Billingsley type theorem for the packing  pressure.
  \begin{thm}\label{thm:btypethmpacking}
    Given an NDS $(\vect{X},\vect{T})$,  a Borel  $E\subset X_{0}$  and $s \in \mathbb{R}$, let  $\mu$  be  a Borel probability measure on $X_{0}$ and $\vect{f} \in \vect{C}(\vect{X},\mathbb{R})$.
    \begin{enumerate}[(1)]
      \item If $\preup_{\mu}(\vect{T},\vect{f},x) \leq s$ for all $x \in E$, then $\prepac(\vect{T},\vect{f},E) \leq s$;
      \item If $\vect{f} \in \vect{C}(\vect{X},\mathbb{R})$ is equicontinuous and  $\preup_{\mu}(\vect{T},\vect{f},x) \geq s$ for all $x \in E$ and $\mu(E)>0$, then $\prepac(\vect{T},\vect{f},E) \geq s$.
    \end{enumerate}
  \end{thm}

As an immediate consequence of the above conclusion, we have the following Billingsley type theorem for the packing   entropy.
  \begin{cor}
    Given an NDS $(\vect{X},\vect{T})$,  a Borel  $E\subset X_{0}$  and $s > 0$, let $\mu$ be a Borel probability measure on $X_{0}$.
    \begin{enumerate}[(1)]
      \item If $\entup_{\mu}(\vect{T},x) \leq s$ for all $x \in E$, then $\enttoppac(\vect{T},E) \leq s$;
      \item If $\entup_{\mu}(\vect{T},x) \geq s$ for all $x \in E$ and $\mu(E)>0$, then $\enttoppac(\vect{T},E) \geq s$.
    \end{enumerate}
  \end{cor}

It is more difficult to explore the variational principle for the  packing  pressure, and the following conclusion is obtained under extra assumptions on $\vect{f}$ and $\prepac$.
  \begin{thm}\label{thm:varprinciplepacking}
Given an NDS $(\vect{X},\vect{T})$ and a non-empty compact $K \subset X_{0}$,
let $\vect{f} \in \vect{C}(\vect{X},\mathbb{R})$ be equicontinuous with  $\|\vect{f}\|<+\infty$.
      If $\prepac(\vect{T},\vect{f},K)>\|\vect{f}\|$, then
      $$\prepac(\vect{T},\vect{f},K)=\sup\{\preup_{\mu}(\vect{T},\vect{f}): \mu \in M(X_{0}), \mu(K)=1\}.$$
  \end{thm}
  Similarly, we have the variational principle for the  packing  entropy as a direct consequence of Theorem \ref{thm:varprinciplepacking}.
  \begin{cor}\label{cor:varprinciplepacking}
    Given an NDS $(\vect{X},\vect{T})$ and a non-empty compact  $K \subset X_{0}$,
    $$\enttoppac(\vect{T},K)=\sup\{\entup_{\mu}(\vect{T}): \mu \in M(X_{0}), \mu(K)=1\}.$$
  \end{cor}

\section{Properties of Topological Pressures and Entropies}\label{sect:proptoppreetent}
In this section, we provide some equivalent definitions of Bowen and packing pressures, which simplifies the study of variational principles.

\subsection{Two covering lemmas}
  First, we introduce two covering lemmas. The first one is a version of the classic $5r$-covering lemma (see \cite[Thm.2.1]{Mattila1995}).
\begin{lem}\label{coveringlem}
    Let $(X,d)$ be a compact metric space and $\mathcal{B}=\{B(x_{i},r_{i})\}_{i \in \mathcal{I}}$
    a family of closed (or open) balls in $X$.
    Then there exists a finite or countable subfamily $\mathcal{B}^{\prime}=\{B(x_{i},r_{i})\}_{i \in \mathcal{I}^{\prime}}$
    of pairwise disjoint balls in $\mathcal{B}$ such that
    $$\bigcup_{B \in \mathcal{B}}B \subset \bigcup_{i \in \mathcal{I}^{\prime}}B(x_{i},5r_{i}).$$
  \end{lem}
Note that  the author of \cite{Mattila1995} assumed that $(X,d)$ is boundedly compact
(all closed bounded subsets are compact), and the conclusion holds for $\mathcal{B}$ of closed balls with $\sup\{\diam(B): B \in \mathcal{B}\}<+\infty$. In this paper, we always assume that $X$ is a compact metric space. Since the compactness of $X$ implies that $X$ is boundedly compact, the conclusion  also holds for open balls by  the same  argument as in \cite[Thm.2.1]{Mattila1995}.

The second covering lemma is similar to the classic $3r$-covering lemma (see Federer \cite[2.8.4-6]{Federer1969}),   and we adopt it for covers by Bowen balls in nonautonomous dynamical systems. We add a proof here for the convenience of readers.
For  $\varepsilon>0$, we write
\begin{equation}\label{def_coloB}
  \mathcal{B}(\varepsilon)=\{B_{n}^{\vect{T}}(x,\varepsilon):x \in X_{0}, 1 \leq n \in \mathbb{N}\}.
\end{equation}
 \begin{lem}\label{dn3rcovlem}
Given $(\vect{X},\vect{T})$ and $\varepsilon>0$, let $\mathcal{B}(\varepsilon)$ be given by \eqref{def_coloB}. Then for every family $\mathcal{F} \subset \mathcal{B}(\varepsilon)$,
    there exists a subfamily $\mathcal{G} \subset \mathcal{F}$ consisting of pairwise disjoint Bowen balls such that
    $$\bigcup_{B \in \mathcal{F}}B \subset \bigcup_{B_{n}^{\vect{T}}(x,\varepsilon) \in \mathcal{G}}B_{n}^{\vect{T}}(x,3\varepsilon).$$
  \end{lem}
  \begin{rmk}
    (1) The subfamily $\mathcal{G}$, in contrast to $\mathcal{B}^{\prime}$ in Lemma \ref{coveringlem},
    is not necessarily countable.  (2) The indices $n$ of the Bowen balls $B_{n}^{\vect{T}}(x,\varepsilon) \in \mathcal{F}$ may not  be identical.
  \end{rmk}
  \begin{proof}
    Given $\mathcal{F} \subset \mathcal{B}(\varepsilon)$, let $(\Omega,\subset)$ denote the partially ordered set
    consisting of all subfamilies $\mathcal{F}^{\prime}$ of $\mathcal{F}$ with the following properties:
    \begin{enumerate}[(a)]
      \item $\mathcal{F}^{\prime}$ consists of disjoint balls from $\mathcal{F}$;
      \item If a ball $B_{n}^{\vect{T}}(x,\varepsilon) \in \mathcal{F}$ intersects some ball
            from $\mathcal{F}^{\prime}$, then there exists $B_{m}^{\vect{T}}(y,\varepsilon) \in \mathcal{F}^{\prime}$
            such that $m \leq n$ and $B_{n}^{\vect{T}}(x,\varepsilon) \cap B_{m}^{\vect{T}}(y,\varepsilon) \neq \emptyset$.
    \end{enumerate}

    The set $\Omega$ is non-empty since it contains the family consisting of a single ball $B_{m}^{\vect{T}}(y,\varepsilon) \in \mathcal{F}$
    where
    $$m=\min\{n: B_{n}^{\vect{T}}(x,\varepsilon) \in \mathcal{F}\}.$$
    Let $\Lambda \subset \Omega$ be a linearly ordered subset. Then $\bigcup_{\mathcal{F}^{\prime} \in \Lambda}\mathcal{F}^{\prime}$
    belongs to $\Omega$ and is an upper bound of $\Lambda$.
    By Zorn's lemma, there exists a maximal element $\mathcal{G} \in \Omega$.

    We claim that for every $B\in\mathcal{F}$, there exists a ball  $B'\in \mathcal{G}$ such that $B\cap B'\neq \emptyset$.
Suppose that there exists a ball $B_{m}^{\vect{T}}(y,\varepsilon) \in \mathcal{F}$
    where
    $$m = \min\{n: B_{n}^{\vect{T}}(x,\varepsilon)\ \text{does not intersect any ball in}\ \mathcal{G}\}.$$
    It is clear that the family $\mathcal{G} \cup \{B_{m}^{\vect{T}}(y,\varepsilon)\}$
    satisfies (a) and (b) and hence belongs to $\Omega$. This contradicts the maximality of $\mathcal{G}$.

    Since $\mathcal{F} \subset \mathcal{B}(\varepsilon)$, by the triangle inequality, every $B \in \mathcal{F}$ is contained in some $B_{n}^{\vect{T}}(x,3\varepsilon)$ where $n$ and $x$ satisfy $B_{n}^{\vect{T}}(x,\varepsilon) \in \mathcal{G}$, and the conclusion follows.
  \end{proof}

  \subsection{Equivalent definitions of Bowen  pressure and entropy}\label{ssect:equivdefbow}
  Given $\vect{f} \in \vect{C}(\vect{X},\mathbb{R})$, for all $s \in \mathbb{R}$, $N >0$ and $\varepsilon>0$,   we define
  \begin{equation}\label{eq:defmsrbpp}
      \msrbpp_{N,\varepsilon}^{s}(\vect{T},\vect{f},Z)=\inf\Big\{\sum_{i=1}^{\infty}\exp{\Big(-n_{i}s+\sup_{y \in B_{n_{i}}^{\vect{T}}(x_{i},\varepsilon)}{S_{n_{i}}^{\vect{T}}\vect{f}(y)}\Big)}\Big\},
  \end{equation}
where the infimum is taken over all countable $(N,\varepsilon)$-covers $\{B_{n_{i}}^{\vect{T}}(x_{i},\varepsilon)\}_{i=1}^\infty$ of $Z$.  Similarly, we write
  $$\msrbpp_{\varepsilon}^{s}(\vect{T},\vect{f},Z)=\lim_{N \to \infty} \msrbpp_{N,\varepsilon}^{s}(\vect{T},\vect{f},Z),$$
  and
\begin{equation} \label{def_QBT}
\prebpp(\vect{T},\vect{f},Z,\varepsilon)=\inf\{s: \msrbpp_{\varepsilon}^{s}(\vect{T},\vect{f},Z)=0\}
    =\sup\{s: \msrbpp_{\varepsilon}^{s}(\vect{T},\vect{f},Z)=+\infty\}.
\end{equation}
Since $\msrbpp_{N,\varepsilon}^{s}(\vect{T},\vect{f},Z)$ is non-decreasing as $N$ tends to $\infty$, the quantities $\msrbpp_{\varepsilon}^{s} $ and $\prebpp$  are well defined.

It would be ideal that $\prebow(\vect{T},\vect{f},Z)=\lim_{\varepsilon \to 0}\prebpp(\vect{T},\vect{f},Z,\varepsilon).$ However, in general, we do not have the convergence of $\prebpp(\vect{T},\vect{f},Z,\varepsilon)$ due to the lack of the monotonicity of $\msrbpp_{\varepsilon}^{s}(\vect{T},\vect{f},Z)$ with respect to $\varepsilon$. Under an extra assumption, we obtain that the limit of $\prebpp$ exists and  equals  the Bowen pressure.
  \begin{prop}\label{equivdefprebpplem}
    If $\vect{f} \in \vect{C}(\vect{X},\mathbb{R})$ is equicontinuous, then $\prebpp(\vect{T},\vect{f},Z,\varepsilon)$ converges as $\varepsilon$ tends to $  0$ and
    \begin{equation}\label{eq:prebowequivprebpp}
      \prebow(\vect{T},\vect{f},Z)=\lim_{\varepsilon \to 0}\prebpp(\vect{T},\vect{f},Z,\varepsilon).
    \end{equation}
  \end{prop}
  \begin{proof}
For all $s \in \mathbb{R}$ and all $\alpha>0$,  we first show that
    \begin{equation}\label{eq:msrbowequivbpp}
      \msrbow_{N,\varepsilon}^{s}(\vect{T},\vect{f},Z) \leq \msrbpp_{N,\varepsilon}^{s}(\vect{T},\vect{f},Z) \leq \msrbow_{N,\varepsilon}^{s-\alpha}(\vect{T},\vect{f},Z)
    \end{equation}
 for all $N >0$ and  all sufficiently small $\varepsilon>0$.  By \eqref{eq:defmsrbow} and \eqref{eq:defmsrbpp}, the left inequality holds for all $N >0$ and $\varepsilon>0$.

It remains to show the right inequality. Given $s \in \mathbb{R}$ and $\alpha>0$, since $\vect{f}$ is equicontinuous,
    there exists $\varepsilon_{0}>0$ such that for all $0<\varepsilon<\varepsilon_{0}$ and all $x^{\prime},x^{\prime\prime} \in X_{0}$     with $d_{j}(\vect{T}^{j}x^{\prime},\vect{T}^{j}x^{\prime})<\varepsilon$,    we have that
    $$
\lvert f_{j}(\vect{T}^{j}x^{\prime})-f_{j}(\vect{T}^{j}x^{\prime\prime}) \rvert < \alpha
$$
for all integers $j>$0. Thus, given $N >0$ and $\varepsilon$ such that $0<\varepsilon<\varepsilon_{0}$, for every countable $(N,\varepsilon)$-cover $\{B_{n_{i}}^{\vect{T}}\left(x_{i},\varepsilon\right)\}_{i=1}^{\infty}$  of $Z$, we have that
$$
\sup_{y \in B_{n_{i}}^{\vect{T}}(x_{i},\varepsilon)}S_{n_{i}}^{\vect{T}}\vect{f}(y) \leq S_{n_{i}}^{\vect{T}}\vect{f}(x_{i})+n_{i}\alpha
$$
for all $i\in\mathbb{N}$, and it follows that
    $$
    \sum_{i=1}^{\infty}\exp{\Big(-n_{i}s + \sup_{y \in B_{n_{i}}^{\vect{T}}(x_{i},\varepsilon)}S_{n_{i}}^{\vect{T}}\vect{f}(y)\Big)} \leq \sum_{i=1}^{\infty}\exp{\left(-n_{i}(s-\alpha)+S_{n_{i}}^{\vect{T}}\vect{f}(x_{i})\right)}.
    $$
    Hence we have that
    $$
    \msrbpp_{N,\varepsilon}^{s}(\vect{T},\vect{f},Z) \leq \msrbow_{N,\varepsilon}^{s-\alpha}(\vect{T},\vect{f},Z)
    $$
    for all $N  >0$ and $0<\varepsilon<\varepsilon_{0}$. Therefore the inequality \eqref{eq:msrbowequivbpp} holds.

Fix $s \in \mathbb{R}$, and arbitrarily choose $\alpha>0$.   It follows from \eqref{eq:msrbowequivbpp} that
    $$
    \msrbow_{\varepsilon}^{s}(\vect{T},\vect{f},Z) \leq \msrbpp_{\varepsilon}^{s}(\vect{T},\vect{f},Z) \leq \msrbow_{\varepsilon}^{s-\alpha}(\vect{T},\vect{f},Z)
    $$
 for all sufficiently small $\varepsilon>0$. By  \eqref{def_PBep}  and  \eqref{def_QBT}, this implies that
    $$
    \prebow(\vect{T},\vect{f},Z,\varepsilon) \leq \prebpp(\vect{T},\vect{f},Z,\varepsilon) \leq \prebow(\vect{T},\vect{f},Z,\varepsilon)+\alpha,
    $$
    and we obtain that
    $$
    \varlimsup_{\varepsilon \to 0}\prebpp(\vect{T},\vect{f},Z,\varepsilon) \leq \lim_{\varepsilon \to 0}\prebow(\vect{T},\vect{f},Z,\varepsilon)+\alpha \leq \varliminf_{\varepsilon \to 0}\prebpp(\vect{T},\vect{f},Z,\varepsilon)+\alpha.
    $$
    By the arbitrariness of $\alpha>0$, we conclude that $\prebpp(\vect{T},\vect{f},Z,\varepsilon)$ converges and
    $$
    \lim_{\varepsilon \to 0}\prebpp(\vect{T},\vect{f},Z,\varepsilon)=\lim_{\varepsilon \to 0}\prebow(\vect{T},\vect{f},Z,\varepsilon)=\prebow(\vect{T},\vect{f},Z).
    $$
  \end{proof}

Finally, we give another equivalent form of the Bowen pressure and entropy. Given $\vect{f}\in \vect{C}(\vect{X},\mathbb{R})$
    and $Z \subset X_{0}$, for all $s \in \mathbb{R}$, we write
    $$
    \msrbow^{s}(\vect{T},\vect{f},Z)=\lim_{\varepsilon \to 0}\msrbow_{\varepsilon}^{s}(\vect{T},\vect{f},Z),
    $$
where the limit follows from the monotonicity of $\msrbow_{\varepsilon}^{s}(\vect{T},\vect{f},Z)$ in Proposition \ref{prop:msrbowepsilon}.  The following proposition shows that $\msrbow^{s}(\vect{T},\vect{f},Z)$ has a jump value which is  identical to the Bowen pressure.
  \begin{prop}\label{equivdefentbowlem}
    Given $\vect{f} \in \vect{C}(\vect{X},\mathbb{R})$,
   $$ 
\prebow(\vect{T},\vect{f},Z)=\inf\{s:\msrbow^{s}(\vect{T},\vect{f},Z)=0\}=\sup\{s:\msrbow^{s}(\vect{T},\vect{f},Z)=+\infty\}.
    $$
 In particular,
    $$
    \enttopbow(\vect{T},Z)=\inf\{s:\msrbow^{s}(\vect{T},\vect{0},Z)=0\}=\sup\{s:\msrbow^{s}(\vect{T},\vect{0},Z)=+\infty\}.
    $$
  \end{prop}
  \begin{proof}
We write
$$
s_0=\inf\{s:\msrbow^{s}(\vect{T},\vect{f},Z)=0\}=\sup\{s:\msrbow^{s}(\vect{T},\vect{f},Z)=+\infty\}.
$$

We first show $s_{0}\geq \prebow(\vect{T},\vect{f},Z)$. For every $t<\prebow(\vect{T},\vect{f},Z)$, by Definition \ref{def:prebpp}, Proposition \ref{prop:msrbowepsilon} and \eqref{def_PBep},  it is clear that $t<\prebow(\vect{T},\vect{f},Z,\varepsilon)$ and $\msrbow_{\varepsilon}^{t}(\vect{T},\vect{f},Z)=+\infty$
    for all sufficiently small $\varepsilon>0$. This implies that $\msrbow^{t}(\vect{T},\vect{f},Z)=+\infty$, and it follows from the definition of $s_{0}$ that  $t \leq s_{0}$. Hence $s_{0}\geq\prebow(\vect{T},\vect{f},Z)$  by the arbitrariness of $t$.

    For each $t>\prebow(\vect{T},\vect{f},Z)$, by Definition \ref{def:prebpp} and Proposition \ref{prop:msrbowepsilon}, we have that $\prebow(\vect{T},\vect{f},Z,\varepsilon)<t$ for all $\varepsilon>0$. By \eqref{def_PBep},  it is clear that $\msrbow_{\varepsilon}^{t}(\vect{T},\vect{f},Z)=0$ for all $\varepsilon>0$.
    Let $\varepsilon$ tend to $0$, and we have $\msrbow^{t}(\vect{T},\vect{f},Z)=0$, which implies that $t \geq s_{0}$. Hence $s_{0}\leq\prebow(\vect{T},\vect{f},Z)$ follows by the arbitrariness of $t$.
  \end{proof}

Note that it is essentially stated in Proposition \ref{equivdefentbowlem} that
    $$
    \inf\{s:\lim_{\varepsilon\to0}\msrbow_{\varepsilon}^{s}(\vect{T},\vect{f},Z)=0\} =\lim_{\varepsilon\to0}\inf\{s:\msrbow_{\varepsilon}^{s}(\vect{T},\vect{f},Z)=0\}.
    $$
    In particular, when $(\vect{X},\vect{T})$ reduces to the TDS $(X,T)$ and $\vect{f}=\vect{0}$,
    this definition reduces directly to the original one of Bowen  entropy given by Bowen in \cite{Bowen1973}.

  \subsection{Bowen  pressure via weighted measures}\label{ssect:equivdefbw}
We introduce an alternative characterization of Bowen pressure,
 which is inspired by  weighted Hausdorff measures from geometric measure theory (see \cite{Federer1969, Mattila1995}).

Given $(\vect{X},\vect{T})$ and  $\vect{f} \in \vect{C}(\vect{X},\mathbb{R})$, let $g:X_{0} \to \mathbb{R}$ be a bounded function. For all $s \in \mathbb{R}$, $N>0$ and $\varepsilon>0$, we write
  \begin{equation}\label{def_WBPP}
    \begin{aligned}
      \mathscr{W}_{N,\varepsilon}^{s}(\vect{T},\vect{f},g)=\inf\Big\{&\sum_{i=1}^{\infty}c_{i}\exp{\Big(-n_{i}s+\sup_{y \in B_{n_{i}}^{\vect{T}}(x_{i},\varepsilon)}{S_{n_{i}}^{\vect{T}}\vect{f}(y)}\Big)}: \\
                                                            & \sum_{i=1}^{\infty}c_{i}\chi_{B_{n_{i}}^{\vect{T}}(x_{i},\varepsilon)} \geq g, 0<c_{i}<+\infty, x_{i} \in X_{0}, n_{i} \geq N\Big\},
    \end{aligned}
  \end{equation}
  where $\chi_{A}$ denotes the characteristic function of a set $A \subset X_{0}$, i.e.,
  \begin{equation*}
    \chi_{A}(x)=\left\{
      \begin{array}{ll}
        1, & \text{if}\ x \in A, \\
        0, & \text{if}\ x \in X_{0} \setminus A.
      \end{array}
      \right.
  \end{equation*}

We summarize the key properties of $\mathscr{W}_{N,\varepsilon}^{s}$  in the following proposition.
  \begin{prop}\label{prop:prebwprop}
Given $s \in \mathbb{R}$, $N>0$ and $\varepsilon>0$, let  $f,g \in C(X_{0},\mathbb{R})$.
    \begin{enumerate}[(1)]
      \item For all $a \geq 0$,
            $$\mathscr{W}_{N,\varepsilon}^{s}(\vect{T},\vect{f},ag) = a\mathscr{W}_{N,\varepsilon}^{s}(\vect{T},\vect{f},g).$$
      \item If $f \leq g$, then
            $$\mathscr{W}_{N,\varepsilon}^{s}(\vect{T},\vect{f},f) \leq \mathscr{W}_{N,\varepsilon}^{s}(\vect{T},\vect{f},g).$$
           In particular, if $f \leq 0$, then $$\mathscr{W}_{N,\varepsilon}^{s}(\vect{T},\vect{f},f)=0.$$
      \item $$\mathscr{W}_{N,\varepsilon}^{s}(\vect{T},\vect{f},f+g) \leq \mathscr{W}_{N,\varepsilon}^{s}(\vect{T},\vect{f},f)+\mathscr{W}_{N,\varepsilon}^{s}(\vect{T},\vect{f},g).$$
      \item $$0 \leq \mathscr{W}_{N,\varepsilon}^{s}(\vect{T},\vect{f},g) \leq \|g\|_{\infty}\mathscr{W}_{N,\varepsilon}^{s}(\vect{T},\vect{f},\chi_{\supp{g}}),$$
            where $\supp{g}$ denotes the support of $g$, i.e., $\supp{g}=\overline{\{x \in X_{0}: g(x) \neq 0\}}$.
    \end{enumerate}
  \end{prop}

Next, we construct another type of topological pressure by using $\mathscr{W}_{N,\varepsilon}^{s}$.
Given $Z \subset X_{0}$, we write
$$
\mathscr{W}_{\varepsilon}^{s}(\vect{T},\vect{f},Z)=\lim_{N \to \infty} \mathscr{W}_{N,\varepsilon}^{s}(\vect{T},\vect{f},\chi_{Z}),
$$
    and
\begin{equation} \label{def_PBW}
\prebw(\vect{T},\vect{f},Z,\varepsilon) = \inf\{s: \mathscr{W}_{\varepsilon}^{s}(\vect{T},\vect{f},Z)=0\}                                              =  \sup\{s: \mathscr{W}_{\varepsilon}^{s}(\vect{T},\vect{f},Z)=+\infty\}.
\end{equation}

  \begin{defn}\label{def:prebw}
Let $\vect{f} \in \vect{C}(\vect{X},\mathbb{R})$  be equicontinuous. We define
$$
\prebw(\vect{T},\vect{f},Z)=\lim_{\varepsilon \to 0}\prebw(\vect{T},\vect{f},Z,\varepsilon).
$$
  \end{defn}

In the following conclusion, we obtain not only the convergence of $\prebw(\vect{T},\vect{f},Z,\varepsilon)$ but also the equivalence of  $\prebw$ and $\prebow$ for  equicontinuous potentials $\vect{f} \in \vect{C}(\vect{X},\mathbb{R})$.
In \cite{CM1}, we proved the following conclusion.
 \begin{prop}\label{prop:equivprebwetprebpp}
    Given an NDS $(\vect{X},\vect{T})$ and $Z \subset X_{0}$, if $\vect{f} \in \vect{C}(\vect{X},\mathbb{R})$ is equicontinuous,
    then $\prebw(\vect{T},\vect{f},Z)$ exists and
    $$\prebow(\vect{T},\vect{f},Z)=\prebw(\vect{T},\vect{f},Z).$$
  \end{prop}

  \subsection{An equivalent definition of packing pressure}\label{ssect:equivdefpac}

Similar to the Bowen pressure, we introduce an equivalent definition of the packing pressure.

  Given $\vect{f} \in \vect{C}(\vect{X},\mathbb{R})$, for each $s \in \mathbb{R}$, $N>0$  and $\varepsilon>0$, we define
  \begin{equation}\label{eqdef_pp}
      \msrppp_{N,\varepsilon}^{s}(\vect{T},\vect{f},Z)=\sup\Big\{\sum_{i=1}^{\infty}\exp{\Big(-n_{i}s+\sup_{y \in \overline{B}_{n_{i}}^{\vect{T}}(x_{i},\varepsilon)}{S_{n_{i}}^{\vect{T}}\vect{f}(y)}\Big)} \Big\},
  \end{equation}
where the supremum is taken over all countable $(N,\varepsilon)$-packings $\{\overline{B}_{n_{i}}^{\vect{T}}(x_{i},\varepsilon)\}_{i=1}^\infty$ of $Z$, and write
 $$
\msrppp_{\infty,\varepsilon}^{s}(\vect{T},\vect{f},Z)=\lim_{N \to \infty} \msrppp_{N,\varepsilon}^{s}(\vect{T},\vect{f},Z).
$$
Again, following a standard procedure, we define
  $$
  \msrppp_{\varepsilon}^{s}(\vect{T},\vect{f},Z)=\inf\Big\{\sum_{i=1}^{\infty}\msrppp_{\infty,\varepsilon}^{s}(\vect{T},\vect{f},Z_{i}): \bigcup_{i=1}^{\infty}Z_{i} \supseteq Z \Big\}
  $$
and write
\begin{equation}\label{def_QP}
    \preppp(\vect{T},\vect{f},Z,\varepsilon) =\inf\{s: \msrppp_{\varepsilon}^{s}(\vect{T},\vect{f},Z)=0\} =\sup\{s: \msrppp_{\varepsilon}^{s}(\vect{T},\vect{f},Z)=+\infty\}.
\end{equation}

By a similar argument to classic packing measures (see \cite{Tricot1982,Mattila1995}), it is straightforward to verify that $\msrpac_{\infty,\varepsilon}^{s}$ and $\msrppp_{\infty,\varepsilon}^{s}$ are not measures, that is, they are not countably subadditive.
The following property is the disadvantages of $\msrpac_{\infty,\varepsilon}^{s}$ and $\msrppp_{\infty,\varepsilon}^{s}$  but has some unexpected consequences.
\begin{prop}\label{prop_NPclosure}
   Given $s \in \mathbb{R}$ and $\varepsilon>0$, for all $Z \subset X_{0}$,
we have that
 $$
\msrppp_{N,\varepsilon}^{s}(\vect{T},\vect{f},Z)=\msrppp_{N,\varepsilon}^{s}(\vect{T},\vect{f},\overline{Z}),\qquad
\msrpac_{N,\varepsilon}^{s}(\vect{T},\vect{f},Z)=\msrpac_{N,\varepsilon}^{s}(\vect{T},\vect{f},\overline{Z})
$$
for all  $N >0$, and
$$
\msrppp_{\infty,\varepsilon}^{s}(\vect{T},\vect{f},Z)=\msrppp_{\infty,\varepsilon}^{s}(\vect{T},\vect{f},\overline{Z}), \qquad \msrpac_{\infty,\varepsilon}^{s}(\vect{T},\vect{f},Z)=\msrpac_{\infty,\varepsilon}^{s}(\vect{T},\vect{f},\overline{Z}).
$$

  \end{prop}

Although $\msrpac_{\infty,\varepsilon}^{s}$ and $\msrppp_{\infty,\varepsilon}^{s}$ are not countably subadditive, they still have the monotonicity in sets.

  \begin{prop}\label{prop_Nesinc}
    Given $s \in \mathbb{R}$ and $\varepsilon>0$, for $Z_{1} \subset Z_{2} \subset X_{0}$, then
 $$
\msrppp_{N,\varepsilon}^{s}(\vect{T},\vect{f},Z_{1}) \leq \msrppp_{N,\varepsilon}^{s}(\vect{T},\vect{f},Z_{2}), \quad \msrpac_{N,\varepsilon}^{s}(\vect{T},\vect{f},Z_{1}) \leq \msrpac_{N,\varepsilon}^{s}(\vect{T},\vect{f},Z_{2})
$$
for all $N >0$, and
    $$\msrppp_{\infty,\varepsilon}^{s}(\vect{T},\vect{f},Z_{1}) \leq \msrppp_{\infty,\varepsilon}^{s}(\vect{T},\vect{f},Z_{2}), \quad \msrpac_{\infty,\varepsilon}^{s}(\vect{T},\vect{f},Z_{1}) \leq \msrpac_{\infty,\varepsilon}^{s}(\vect{T},\vect{f},Z_{2}).$$
  \end{prop}

Next, we show that the packing pressure may alternatively be given by the limit of $ \preppp(\vect{T},\vect{f},Z,\varepsilon)$.
  \begin{prop}\label{equivdefprepaclem}
    If $\vect{f} \in \vect{C}(\vect{X},\mathbb{R})$ is equicontinuous, then $\preppp(\vect{T},\vect{f},Z,\varepsilon)$ converges as $\varepsilon \to 0$ and
    \begin{equation}\label{eq:prepacequivpreppp}
      \prepac(\vect{T},\vect{f},Z)=\lim_{\varepsilon \to 0}\preppp(\vect{T},\vect{f},Z,\varepsilon).
    \end{equation}
  \end{prop}
  \begin{proof}
    We first show that for every fixed $s \in \mathbb{R}$ and $\alpha>0$, the following inequality
    \begin{equation}\label{eq:msrpacequivmsrppp}
      \msrpac_{N,\varepsilon}^{s}(\vect{T},\vect{f},Z) \leq \msrppp_{N,\varepsilon}^{s}(\vect{T},\vect{f},Z) \leq \msrpac_{N,\varepsilon}^{s-\alpha}(\vect{T},\vect{f},Z)
    \end{equation}
    holds for all $N >0$ and all sufficiently small $\varepsilon>0$. The left inequality  is trivial, and we only need to show the right inequality.

Fix $s \in \mathbb{R}$ and $\alpha>0$. Since   $\vect{f}$ is equicontinuous, there exists $\varepsilon_{0}>0$ such that for every $N >0$, $0<\varepsilon<\varepsilon_{0}$ and every countable $(N,\varepsilon)$-packing $\{\overline{B}_{n_{i}}^{\vect{T}}(x_{i},\varepsilon)\}_{i=1}^{\infty}$ of $Z$, we have
    $$
    \sup_{y \in \overline{B}_{n_{i}}^{\vect{T}}(x_{i},\varepsilon)}S_{n_{i}}^{\vect{T}}\vect{f}(y) \leq S_{n_{i}}^{\vect{T}}\vect{f}(x_{i})+n_{i}\alpha.
    $$
It implies that
    $$
    \sum_{i=1}^{\infty}\exp{\Big(-n_{i}s+\sup_{y \in \overline{B}_{n_{i}}^{\vect{T}}(x_{i},\varepsilon)}S_{n_{i}}^{\vect{T}}\vect{f}(y)\Big)} \leq \sum_{i=1}^{\infty}\exp{\left(-n_{i}(s-\alpha)+S_{n_{i}}^{\vect{T}}\vect{f}(x_{i})\right)},
    $$
    and we obtain that
$$
\msrppp_{N,\varepsilon}^{s}(\vect{T},\vect{f},Z) \leq \msrpac_{N,\varepsilon}^{s-\alpha}(\vect{T},\vect{f},Z)
$$
for all $N >0$ and all sufficiently small $\varepsilon>0$.
Therefore, the inequality \eqref{eq:msrpacequivmsrppp} holds.

By \eqref{eq:msrpacequivmsrppp}, as $N$ tends to $\infty$,  we have that
    $$
\msrpac_{\infty,\varepsilon}^{s}(\vect{T},\vect{f},Z) \leq \msrppp_{\infty,\varepsilon}^{s}(\vect{T},\vect{f},Z) \leq \msrpac_{\infty,\varepsilon}^{s-\alpha}(\vect{T},\vect{f},Z)
$$
 for all sufficiently small $\varepsilon>0$.
Since this holds for all  subsets of $X_{0}$, we apply it to all members of any given countable cover $\{Z_{i}\}_{i=1}^{\infty}$ of $Z$ and  obtain that
    $$\msrpac_{\varepsilon}^{s}(\vect{T},\vect{f},Z) \leq \msrppp_{\varepsilon}^{s}(\vect{T},\vect{f},Z) \leq \msrpac_{\varepsilon}^{s-\alpha}(\vect{T},\vect{f},Z).$$
By \eqref{def_PP} and \eqref{def_QP}, this implies  that
    $$\prepac(\vect{T},\vect{f},Z,\varepsilon) \leq \preppp(\vect{T},\vect{f},Z,\varepsilon) \leq \prepac(\vect{T},\vect{f},Z,\varepsilon)+\alpha.$$
By Proposition \ref{prop:msrpacepsilon},    it follows that
    $$\varlimsup_{\varepsilon \to 0}\preppp(\vect{T},\vect{f},Z,\varepsilon) \leq \lim_{\varepsilon \to 0}\prepac(\vect{T},\vect{f},Z,\varepsilon)+\alpha \leq \varliminf_{\varepsilon \to 0}\preppp(\vect{T},\vect{f},Z,\varepsilon)+\alpha.$$
    Since $\alpha>0$ is arbitrarily chosen, we have that
    $$\prepac(\vect{T},\vect{f},Z)=\lim_{\varepsilon \to 0}\preppp(\vect{T},\vect{f},Z,\varepsilon)=\lim_{\varepsilon \to 0}\prepac(\vect{T},\vect{f},Z,\varepsilon),$$
   and the conclusion follows.
  \end{proof}

  \subsection{Properties of topological pressures and entropies}\label{ssect:propdepreetent}
In this subsection, we list some basic properties of Bowen and packing pressures and entropies. In particular,   the pressures $\prebow$ and $\prepac$ have the similar properties of fractal dimensions; see \cite{Falconer2014} for properties of fractal dimensions.  Therefore, these pressures  may be regarded as a type of dimensions for nonautonomous dynamical systems.

    We first give the relations between Bowen and packing pressures (entropies).
    \begin{prop}\label{prop:prebowleqprepac}
      Given $Z \subset X_{0}$ and $\vect{f} \in \vect{C}(\vect{X},\mathbb{R})$,
      if $\vect{f}$ is equicontinuous, then
      $$\prebow(\vect{T},\vect{f},Z) \leq \prepac(\vect{T},\vect{f},Z).$$
 In particular,
      $$\enttopbow(\vect{T},Z) \leq \enttoppac(\vect{T},Z).$$
    \end{prop}
    \begin{proof}
      For every $N >0$ and $\varepsilon>0$, we choose a $(N, \varepsilon)$-packing $\{\overline{B}_{N}^{\vect{T}}(x_{i},\varepsilon)\}_{i\in \mathcal{I}}$ of $Z$ with maximal cardinality, that is, for every $x\in Z$, there exists $i\in \mathcal{I}$ such that
      $$
      \overline{B}_{N}^{\vect{T}}(x_{i},\varepsilon)\cap \overline{B}_{N}^{\vect{T}}(x,\varepsilon)\neq \emptyset.
      $$
Hence, for every $y \in Z$, there exists  $i\in \mathcal{I}$ such that $d_{N}^{\vect{T}}(x_{i},y) \leq 2\varepsilon$, which implies that the collection $\{B_{N}^{\vect{T}}(x_{i},3\varepsilon)\}_{i \in \mathcal{I}}$ is a cover of $Z$. Since $X_0$ is compact, it is clear that $\mathcal{I}$ is finite.

 For sufficiently small $\varepsilon>0$, by \eqref{eq:defmsrbow} and \eqref{eq:defmsrPack},
      we have
      $$
      \msrbow_{N,3\varepsilon}^{s}(\vect{T},\vect{f},Z) \leq \sum_{i\in \mathcal{I}}\exp{\left(-Ns+S_{N}^{\vect{T}}\vect{f}(x_{i})\right)} \leq \msrpac_{N,\varepsilon}^{s}(\vect{T},\vect{f},Z)
      $$
      for all $s \in \mathbb{R}$. Letting  $N\to \infty$, it implies that
      \begin{equation}\label{eq:msrbowleqmsrpac}
        \msrbow_{3\varepsilon}^{s}(\vect{T},\vect{f},Z) \leq \msrpac_{\infty,\varepsilon}^{s}(\vect{T},\vect{f},Z).
      \end{equation}
Since $\msrbow_{3\varepsilon}^{s}$ is an outer measure, for any given countable cover $\{Z_{i}\}_{i=1}^{\infty}$ of $Z$, we apply \eqref{eq:msrbowleqmsrpac} to $Z_i$ and obtain that
      \begin{align*}
        \msrbow_{3\varepsilon}^{s}(\vect{T},\vect{f},Z) &\leq \inf\Big\{\sum_{i=1}^{\infty}\msrbow_{3\varepsilon}^{s}(\vect{T},\vect{f},Z_{i}): \bigcup_{i=1}^{\infty}Z_{i} \supseteq Z\Big\} \\
         &\leq \inf\Big\{\sum_{i=1}^{\infty}\msrpac_{\infty,\varepsilon}^{s}(\vect{T},\vect{f},Z_{i}): \bigcup_{i=1}^{\infty}Z_{i} \supseteq Z\Big\} \\
         &= \msrpac_{\varepsilon}^{s}(\vect{T},\vect{f},Z).
      \end{align*}
For every  $s<\prebow(\vect{T},\vect{f},Z)$, by Proposition \ref{equivdefentbowlem}, we have $\msrbow^{s}(\vect{T},\vect{f},Z) = +\infty$, and it follows that $\msrbow_{3\varepsilon}^{s} (\vect{T},\vect{f},Z) >0$ for  all sufficiently small $\varepsilon>0$. Since       $\msrpac_{\varepsilon}^{s}(\vect{T},\vect{f},Z) \geq \msrbow_{3\varepsilon}^{s}(\vect{T},\vect{f},Z)$,
by \eqref{def_PP} and Proposition \ref{prop:msrpacepsilon}, we have that
$$
\prepac(\vect{T},\vect{f},Z) \geq \prepac(\vect{T},\vect{f},Z,\varepsilon) \geq s.
$$
      By the arbitrariness of $s<\prebow(\vect{T},\vect{f},Z)$, we conclude that
      $\prebow(\vect{T},\vect{f},Z) \leq \prepac(\vect{T},\vect{f},Z).$
    \end{proof}

    Next, it directly follows from the definition  that $\prebow(\vect{T},\vect{f},\cdot)$ and $\prepac(\vect{T},\vect{f},\cdot)$ are increasing in sets.
    \begin{prop}\label{Prop_PBPP}
      Given $\vect{f} \in \vect{C}(\vect{X},\mathbb{R})$, if $Z_{1} \subset Z_{2} \subset X_{0}$, then
      $$\prebow(\vect{T},\vect{f},Z_{1}) \leq \prebow(\vect{T},\vect{f},Z_{2}),\quad \prepac(\vect{T},\vect{f},Z_{1}) \leq \prepac(\vect{T},\vect{f},Z_{2}).$$
    \end{prop}

    The following conclusion shows that $\prebow$ and $\prepac$ are countably stable.
    \begin{prop}
      Given $\vect{f} \in \vect{C}(\vect{X},\mathbb{R})$, if $Z=\bigcup_{i=1}^{\infty}Z_{i}$, then
      $$\prebow(\vect{T},\vect{f},Z)=\sup_{i =1,2,\ldots }\prebow(\vect{T},\vect{f},Z_{i}), \quad \prepac(\vect{T},\vect{f},Z)=\sup_{i  =1,2,\ldots}\prepac(\vect{T},\vect{f},Z_{i}).$$
    \end{prop}
    \begin{proof}
      ``$\leq$'' follows from subadditivity of $\msrbow_{\varepsilon}^{s}(\vect{T},\vect{f},\cdot)$ and $\msrpac_{\varepsilon}^{s}(\vect{T},\vect{f},\cdot)$,
      while ``$\geq$'' is by monotonicity for every $Z_{i} \subset Z$.
    \end{proof}

Since  the mappings $\prebow(\vect{T},\cdot,Z), \prepac(\vect{T},\cdot,Z) :\vect{C}(\vect{X},\mathbb{R}) \to \mathbb{R}\cup\{\pm\infty\}$ are generalizations of the classic pressure in topological dynamical systems, they still keep the similar properties to the  classic pressure.
Recall that the entropies equal the pressures for $\vect{f}=\vect{0}$.    If  $\vect{f}$ is constant, we have the following equalities.
    \begin{prop}
      Given $Z \subset X_{0}$, if $\vect{f} = a\vect{1}$ where $a \in \mathbb{R}$, then
      $$\prebow(\vect{T},\vect{f},Z) = \enttopbow(\vect{T},Z)+a, \quad \prepac(\vect{T},\vect{f},Z) = \enttoppac(\vect{T},Z)+a.$$
    \end{prop}

The following conclusion shows the positivity property of the pressures.
    \begin{prop}\label{prop:positivity}
      Given $Z \subset X_{0}$, if $\vect{f} \preceq \vect{g}$, then
      $$\prebow(\vect{T},\vect{f},Z) \leq \prebow(\vect{T},\vect{g},Z), \quad \prepac(\vect{T},\vect{f},Z) \leq \prepac(\vect{T},\vect{g},Z).$$
      In particular,  for $\vect{f} \succeq \vect{0}$,
      $$0 \leq \enttopbow(\vect{T},Z) \leq \prebow(\vect{T},\vect{f},Z), \quad 0 \leq \enttoppac(\vect{T},Z) \leq \prepac(\vect{T},\vect{f},Z).$$
    \end{prop}

Note that the value of $\prebow(\vect{T},\vect{f},Z)$ and $ \prepac(\vect{T},\vect{f},Z)$ may vary dramatically in  $\mathbb{R} \cup \{\pm\infty\}$,
but if we restrict  $\prebow(\vect{T},\cdot,Z)$ and $\prepac(\vect{T},\cdot,Z)$  to $\vect{C}_{b}(\vect{X},\mathbb{R})$,  these ranges of values may behave differently. For $\vect{f} \in \vect{C}_{b}(\vect{X},\mathbb{R})$, we write
    $$
    \inf{\vect{f}} =\inf_{k \in \mathbb{N}}{(\inf{f_{k}})}\qquad \text{and}\qquad \sup{\vect{f}}=\sup_{k \in \mathbb{N}}{(\sup{f_{k}})}.
    $$

    \begin{prop}
      Given $Z \subset X_{0}$, if $\vect{f} \in \vect{C}_{b}(\vect{X},\mathbb{R})$, then
              $$\enttopbow(\vect{T},Z) + \inf{\vect{f}} \leq \prebow(\vect{T},\vect{f},Z) \leq \enttopbow(\vect{T},Z) + \sup{\vect{f}},$$
              $$\enttoppac(\vect{T},Z) + \inf{\vect{f}} \leq \prepac(\vect{T},\vect{f},Z) \leq \enttoppac(\vect{T},Z) + \sup{\vect{f}}.$$
Moreover, $\prebow(\vect{T},\cdot,Z)$ and $\prepac(\vect{T},\cdot,Z)$ are either finite or constantly $\infty$ on $\vect{C}_{b}(\vect{X},\mathbb{R})$.
    \end{prop}

\section{Invariance under Equiconjugacy}\label{sect:invarequiconj}

  \subsection{Equisemiconjugacies and equiconjugacies}
  Let $(\vect{X},\vect{T})$ and $(\vect{Y},\vect{R})$ be two NDSs, where $X_{k}$ and $Y_{k}$ are endowed with
  the metrics $d_{X_{k}}$ and $d_{Y_{k}}$ respectively throughout this section.

  We call a sequence $\vect{\pi}=\{\pi_{k}\}_{k=0}^{\infty}$ of surjective continuous mappings $\pi_{k}:X_{k} \to Y_{k}$ a \emph{semiconjugacy} or a \emph{factor mapping sequence} from $(\vect{X},\vect{T})$ to $(\vect{Y},\vect{R})$
  if $R_{k} \circ \pi_{k}=\pi_{k+1} \circ T_{k}$  for every $k \in \mathbb{N}$,
  i.e., the diagram
  \begin{equation*}
    \xymatrix{
      (\vect{X},\vect{T}): \ar[d]^{\vect{\pi}:} & X_{0} \ar[r]^{T_{0}} \ar[d]_{\pi_{0}} & X_{1} \ar[r]^{T_{1}} \ar[d]_{\pi_{1}} & X_{2} \ar[r]^{T_{2}} \ar[d]_{\pi_{2}} & \cdots \ar[r]^{T_{k-1}} \ar @{} [d] |{\cdots} & X_{k} \ar[r]^{T_{k}} \ar[d]_{\pi_{k}} & \cdots \ar @{} [d] |{\cdots}\\
      (\vect{Y},\vect{R}):                      & Y_{0} \ar[r]^{R_{0}}                  & Y_{1} \ar[r]^{R_{1}}                  & Y_{2} \ar[r]^{R_{2}}                  & \cdots \ar[r]^{R_{k-1}}                       & Y_{k} \ar[r]^{R_{k}}                  & \cdots
    }
  \end{equation*}
  commutes. If there is a semiconjugacy from $(\vect{X},\vect{T})$ to $(\vect{Y},\vect{R})$,
  we say that $(\vect{Y},\vect{R})$ is a \emph{factor} of $(\vect{X},\vect{T})$.
Moreover, we call
  $\vect{\pi}$ a \emph{conjugacy} from $(\vect{X},\vect{T})$ to $(\vect{Y},\vect{R})$ if the continuous mappings $\pi_{k}$ are homeomorphisms for all $k\geq 0$,
in which case the sequence $\vect{\pi}^{-1}$ of the inverses $\pi_{k}^{-1}$ is a semiconjugacy from $(\vect{Y},\vect{R})$ to $(\vect{X},\vect{T})$.

Even for systems $(X,\vect{T})$ and $(Y,\vect{R})$ each with an identical space, a conjugacy is not sufficient to preserve the topological entropies; see  Example  \ref{ex_nc} at the end of this section. Hence  it is reasonable to define equisemiconjugacy and equiconjugacy to study the invariance of entropies and pressures in nonautonomous dynamical systems . 
  \begin{defn}
    Let $\vect{\pi}$ be a semiconjugacy from $(\vect{X},\vect{T})$ to $(\vect{Y},\vect{R})$.
    We say $\vect{\pi}$ is an \emph{equisemiconjugacy} if it is equicontinuous, i.e.,
    for every $\varepsilon>0$, there exists $\delta>0$ such that for all $k \in \mathbb{N}$
    and all $x^{\prime},x^{\prime\prime} \in X_{k}$ satisfying  $d_{X_{k}}(x^{\prime},x^{\prime\prime})<\delta $, we have that
$$
 d_{Y_{k}}(\pi_{k}x^{\prime},\pi_{k}x^{\prime\prime})<\varepsilon.
$$
We say $\vect{\pi}$ is an \emph{equiconjugacy}    if  $\vect{\pi}$ is a conjugacy and  both $\vect{\pi}$ and $\vect{\pi}^{-1}$ are   equicontinuous.
  \end{defn}

Let $\vect{\pi}$ be a semiconjugacy from $(\vect{X},\vect{T})$ to $(\vect{Y},\vect{R})$.
    For every $\vect{g} \in \vect{C}(\vect{Y},\mathbb{R})$, we write $\vect{\pi}^{*}\vect{g}=\{g_{k} \circ \pi_{k}\}_{k=0}^{\infty}$.
  It is clear that $\vect{\pi}^{*}\vect{g} \in \vect{C}(\vect{X},\mathbb{R})$.
  Moreover, $\vect{\pi}^{*}\vect{g}$ is equicontinuous if $\vect{g} \in \vect{C}(\vect{Y},\mathbb{R})$ is equicontinuous and $\vect{\pi}$ is an equisemiconjugacy.

  \subsection{Invariance of topological pressures and entropies}
  We show the invariance of Bowen and packing pressures and entropies under equiconjugacies.

  \begin{thm}\label{thm:invar}
    Let $\vect{g} \in \vect{C}(\vect{X},\mathbb{R})$ be equicontinuous and $Z \subset X_{0}$.
        \begin{enumerate}[(1)]
      \item Suppose that $\vect{\pi}$ is an equisemiconjugacy from $(\vect{X},\vect{T})$ to $(\vect{Y},\vect{R})$.
            Then
            $$\prebow(\vect{T},\vect{\pi}^{*}\vect{g},Z) \geq \prebow(\vect{R},\vect{g},\pi_{0}(Z)),\quad \prepac(\vect{T},\vect{\pi}^{*}\vect{g},Z) \geq \prepac(\vect{R},\vect{g},\pi_{0}(Z)).$$
            In particular,
            $$\enttopbow(\vect{T},Z) \geq \enttopbow(\vect{R},\pi_{0}(Z)),\quad \enttoppac(\vect{T},Z) \geq \enttoppac(\vect{R},\pi_{0}(Z)).$$
      \item Suppose that $\vect{\pi}$ is an equiconjugacy.
            Then
            $$\prebow(\vect{T},\vect{\pi}^{*}\vect{g},Z) = \prebow(\vect{R},\vect{g},\pi_{0}(Z)),\quad \prepac(\vect{T},\vect{\pi}^{*}\vect{g},Z) = \prepac(\vect{R},\vect{g},\pi_{0}(Z)).$$
            In particular,
            $$\enttopbow(\vect{T},Z) = \enttopbow(\vect{R},\pi_{0}(Z)),\quad \enttoppac(\vect{T},Z) = \enttoppac(\vect{R},\pi_{0}(Z)).$$
    \end{enumerate}
  \end{thm}
  \begin{proof}
First, We prove $\prebow(\vect{T},\vect{\pi}^{*}\vect{g},Z) \geq \prebow(\vect{R},\vect{g},\pi_{0}(Z))$.

Since  $\vect{\pi}$ is an equisemiconjugacy,  for every $\varepsilon>0$, there exists  $\delta>0$
    such that for all integers $k \geq 0$, $n >0$ and all $x^{\prime},x^{\prime\prime} \in X_{k}$ satisfying $d_{k,n}^{\vect{T}}(x^{\prime},x^{\prime\prime})<\delta $, we have that
$$
  d_{k,n}^{\vect{R}}(\pi_{k}x^{\prime},\pi_{k}x^{\prime\prime})<\varepsilon.
$$
This implies that
$$
\pi_{k}(B_{k,n}^{\vect{T}}(x,\delta)) \subset B_{k,n}^{\vect{R}}(\pi_{k}x,\varepsilon)
$$
 for every  $x \in X_{k}$.     In particular, for all integers $n >0$ and all  $x \in X_{0}$,
    \begin{equation}\label{eq:notebowballs}
      \pi_{0}(B_{n}^{\vect{T}}(x,\delta)) \subset B_{n}^{\vect{R}}(\pi_{0}x,\varepsilon).
    \end{equation}

    Fix $s \in \mathbb{R}$, $N>0$ and $\varepsilon>0$. We choose $\delta>0$ as above such that $\delta \to 0$ as $\varepsilon \to 0$.
    Given a countable $(N,\delta)$-cover $\{B_{n_{i}}^{\vect{T}}(x_{i},\delta)\}_{i=1}^{\infty}$ of $Z$, it is clear that $\{B_{n_{i}}^{\vect{R}}(\pi_{0}x_{i},\varepsilon)\}_{i=1}^{\infty}$  is a $(N,\varepsilon)$-cover of $\pi_{0}(Z)$.
    Hence, by \eqref{eq:defmsrbow}, we have that
    \begin{align*}
      \msrbow_{N,\varepsilon}^{s}(\vect{R},\vect{g},\pi_{0}(Z)) &\leq \sum_{i=1}^{\infty}\exp{\left(-n_{i}s + S_{n_{i}}^{\vect{R}}\vect{g}(\pi_{0}x_{i})\right)} \\
                                                                        &= \sum_{i=1}^{\infty}\exp{\left(-n_{i}s + S_{n_{i}}^{\vect{T}}\vect{\pi}^{*}\vect{g}(x_{i})\right)}.
    \end{align*}
    By \eqref{eq:notebowballs}, it is clear that
    $$
    \msrbow_{N,\varepsilon}^{s}(\vect{R},\vect{g},\pi_{0}(Z)) \leq \msrbow_{N,\delta}^{s}(\vect{T},\vect{\pi}^{*}\vect{g},Z),
    $$
and we obtain that
    $$
    \msrbow_{\varepsilon}^{s}(\vect{R},\vect{g},\pi_{0}(Z)) \leq \msrbow_{\delta}^{s}(\vect{T},\vect{\pi}^{*}\vect{g},Z).
    $$
By \eqref{def_PBep},    it follows that
    $$
    \prebow(\vect{R},\vect{g},\pi_{0}(Z),\varepsilon) \leq \prebow(\vect{T},\vect{\pi}^{*}\vect{g},Z,\delta).
    $$
    Since $\delta$ tends to $0$ as $\varepsilon$ goes to  $0$ , we have that  $\prebow(\vect{R},\vect{g},\pi_{0}(Z)) \leq \prebow(\vect{T},\vect{\pi}^{*}\vect{g},Z)$.

Next, we  prove $\prepac(\vect{T},\vect{\pi}^{*}\vect{g},Z) \geq \prepac(\vect{R},\vect{g},\pi_{0}(Z))$.

For every $s<\prepac(\vect{R},\vect{g},\pi_{0}(Z))$, by Proposition \ref{prop:msrpacepsilon}, there exists $\varepsilon _0>0$ such that for all $\varepsilon $ satisfying $0<\varepsilon  <\varepsilon _0$, we have that
$$
\prepac(\vect{R},\vect{g},\pi_{0}(Z),\varepsilon )>s.
$$
By \eqref{def_PP}, we have that
$$
\msrpac_{\varepsilon}^{s}(\vect{R},\vect{g},\pi_{0}(Z))=+\infty.
$$

Given a countable cover $\{Z_{i}\}_{i=1}^{\infty}$ of $Z$, it is clear that  $\{\pi_{0}(Z_{i})\}_{i=1}^{\infty}$  is a cover of $\pi_{0}(Z)$.
   By the definition of $\msrpac_{\varepsilon}^{s}$, we have
$$
\sum_{i=1}^{\infty}\msrpac_{\infty,\varepsilon}^{s}(\vect{R},\vect{g}, \pi_{0}(Z_{i}))=+\infty.
$$
Since $\vect{\pi}$ is an equisemiconjugacy,  by the same argument as before,     there exists $\delta>0$  such that for all    $n \in \mathbb{N}$ and all $x \in X_{0}$,
    $$
\pi_{0}(\overline{B}_{n}^{\vect{T}}(x,\delta)) \subset \overline{B}_{n}^{\vect{R}}(\pi_{0}x,\varepsilon).
$$

Fix $N >0$.     For each $i$, let $\{\overline{B}_{n_{i,l}}^{\vect{R}}(y_{i,l},\varepsilon)\}_{l=1}^{\infty}$ be a $(N,\varepsilon)$-packing of $\pi_0(Z_i)$,  and choose $x_{i,l} \in \pi_{0}^{-1}(y_{i,l}) \cap Z_{i}$. Clearly the collection  $\{\overline{B}_{n_{i,l}}^{\vect{T}}(x_{i,l},\delta)\}_{l=1}^{\infty}$ is a $(N,\delta)$-packing of $Z_i$. Combining this with  \eqref{eq:defmsrPack}, we have that
    \begin{align*}
      \msrpac_{N,\delta}^{s}(\vect{T},\vect{\pi}^{*}\vect{g},Z_{i}) &\geq \sum_{l=1}^{\infty}\exp{\left(-n_{i,l}s + S_{n_{i,l}}^{\vect{T}}\vect{\pi}^{*}\vect{g}(x_{i,l})\right)} \\
                                                                    &= \sum_{l=1}^{\infty}\exp{\left(-n_{i,l}s + S_{n_{i,l}}^{\vect{R}}\vect{g}(\pi_{0}x_{i,l})\right)} \\
                                                                    &= \sum_{l=1}^{\infty}\exp{\left(-n_{i,l}s + S_{n_{i,l}}^{\vect{R}}\vect{g}(y_{i,l})\right)}.
    \end{align*}
It immediately follows that
$$
\msrpac_{N,\delta}^{s}(\vect{T},\vect{\pi}^{*}\vect{g},Z_{i}) \geq \msrpac_{N,\varepsilon}^{s}(\vect{R},\vect{g},\pi_0(Z_{i}))
$$
for all $i$.
Letting $N $ tend  to  $\infty$ and summing the inequality over $i$, we obtain that
    $$
    \sum_{i=1}^{\infty}\msrpac_{\infty,\delta}^{s}(\vect{T},\vect{\pi}^{*}\vect{g},Z_{i}) \geq \sum_{i=1}^{\infty}\msrpac_{\infty,\varepsilon}^{s}(\vect{R},\vect{g},\pi_{0}(Z_{i})) = +\infty.
    $$
    By the arbitrariness of $\{Z_{i}\}_{i=1}^{\infty}$, we have $\msrpac_{\delta}^{s}(\vect{T},\vect{\pi}^{*}\vect{g},Z)=+\infty$,
and by \eqref{def_PP} it implies that $\prepac(\vect{T},\vect{\pi}^{*}\vect{g},Z,\delta)>s $ for all $s<\prepac(\vect{R},\vect{g},\pi_{0}(Z))$. It follows that
 $$
\prepac(\vect{R},\vect{g},\pi_{0}(Z),\varepsilon) \leq \prepac(\vect{T},\vect{\pi}^{*}\vect{g},Z,\delta).
$$
Letting $\varepsilon \to 0$, we obtain that  $\prepac(\vect{R},\vect{g},\pi_{0}(Z)) \leq \prepac(\vect{T},\vect{\pi}^{*}\vect{g},Z)$.

    (2) Since $\vect{\pi}$ is an equiconjugacy from $(\vect{X},\vect{T})$ to $(\vect{Y},\vect{R})$,
it is clear that  $\vect{\pi}^{-1}=\{\pi_{k}^{-1}\}_{k=0}^{\infty}$ is an equisemiconjugacy from $(\vect{Y},\vect{R})$ to $(\vect{X},\vect{T})$.
The conclusion follows by  applying   (1) on  both $\vect{\pi}$ and $\vect{\pi}^{-1}$.
  \end{proof}
  \begin{cor}
    Let $\vect{\pi}$ be an equisemiconjugacy from $(\vect{X},\vect{T})$ to $(\vect{Y},\vect{R})$ and $W \subset Y_{0}$.
    Assume that $\vect{g} \in \vect{C}(\vect{Y},\mathbb{R})$ is equicontinuous.
    Then
    $$
    \prebow(\vect{T},\vect{\pi}^{*}\vect{g},\pi_{0}^{-1}(W)) \geq \prebow(\vect{R},\vect{g},W),\quad \prepac(\vect{T},\vect{\pi}^{*}\vect{g},\pi_{0}^{-1}(W)) \geq \prepac(\vect{R},\vect{g},W).
    $$
    In particular,
    $$
    \enttopbow(\vect{T},\pi_{0}^{-1}(W)) \geq \enttopbow(\vect{R},W),\quad \enttoppac(\vect{T},\pi_{0}^{-1}(W)) \geq \enttoppac(\vect{R},W).
    $$
  \end{cor}

  Another immediate consequence of Theorem \ref{thm:invar} is that the topological pressures ($\prebow$ and $\prepac$) and entropies ($\enttopbow$ and $\enttoppac$)
 are indeed `topological', or in other words, metric irrelevant in the following sense. For each integer $k\geq 0$, let  $d_{k}$  and  $d_{k}^{\prime}$ be two   metrics on   $X_{k}$. We say $\vect{d}=\{d_{k}\}_{k=0}^{\infty}$ and $\vect{d}^{\prime}=\{d_{k}^{\prime}\}_{k=0}^{\infty}$
  on $\vect{X}$ are  \emph{uniformly equivalent} if the identity mappings
  $$
  \vect{\id}=\{\id_{X_{k}}:(X_{k},d_{k})\to(X_{k},d_{k}^{\prime})\}_{k=0}^{\infty}\quad \text{and}\quad \vect{\id}^{\prime}=\{\id_{X_{k}}^{\prime}:(X_{k},d_{k}^{\prime})\to(X_{k},d_{k})\}_{k=0}^{\infty}
  $$
  are both equicontinuous.


  Consider the NDSs $(\vect{X},\vect{d},\vect{T})$ and $(\vect{X},\vect{d}^{\prime},\vect{T})$.
  We write $h^{\mathrm{B}}$, $\prebow$, $h^{\mathrm{P}}$ and $\prepac$ with subscripts $\vect{d}$ and $\vect{d}^{\prime}$
  to emphasize the dependence on metrics.
  \begin{thm}\label{thm:pretopuniequiv}
    Given an NDS $(\vect{X},\vect{T})$, let $\vect{d}$ and $\vect{d}^{\prime}$ be two sequences of metrics
    $d_{k}$ and $d_{k}^{\prime}$ inducing the same topologies on $X_{k}$ for each $k \in \mathbb{N}$. If $\vect{d}$ and $\vect{d}^{\prime}$ are uniformly equivalent, then for all equicontinuous $\vect{f} \in \vect{C}(\vect{X},\mathbb{R})$, we have that
    $$
    \prebow_{\vect{d}}(\vect{T},\vect{f},Z) = \prebow_{\vect{d}^{\prime}}(\vect{T},\vect{f},Z),\quad \prepac_{\vect{d}}(\vect{T},\vect{f},Z) = \prepac_{\vect{d}^{\prime}}(\vect{R},\vect{f},Z).
    $$
    In particular,
    $$\entbow_{\vect{d}}(\vect{T},Z) = \entbow_{\vect{d}^{\prime}}(\vect{T},Z),\quad \entpac_{\vect{d}}(\vect{T},Z) = \entpac_{\vect{d}^{\prime}}(\vect{T},Z).$$
  \end{thm}
  \begin{proof}
    Obviously, compositions of $\id_{X_{k}}$ and $T_{k}$ are commutative.
    \begin{equation*}
      \xymatrix{
        (\vect{X},\vect{d},\vect{T}): \ar[d]^{\vect{\id}:}          & (X_{0},d_{0}) \ar[r]^{T_{0}} \ar[d]_{\id_{X_{0}}}     & (X_{1},d_{1}) \ar[r]^{T_{1}} \ar[d]_{\id_{X_{1}}}          & \cdots \ar[r]^{T_{k-1}} \ar @{} [d] |{\cdots} & (X_{k},d_{k}) \ar[r]^{T_{k}} \ar[d]_{\id_{X_{k}}}          & \cdots \ar @{} [d] |{\cdots}\\
        (\vect{X},\vect{d}^{\prime},\vect{T}):                      & (X_{0},d_{0}^{\prime}) \ar[r]^{T_{0}}                  & (X_{1},d_{1}^{\prime}) \ar[r]^{T_{1}}                  & \cdots \ar[r]^{T_{k-1}}                       & (X_{k},d_{k}^{\prime}) \ar[r]^{T_{k}}                  & \cdots
      }
    \end{equation*}
Since $\vect{d}$ and $\vect{d}^{\prime}$
    are uniformly equivalent, $\vect{\id}$ is an equiconjugacy from $(\vect{X},\vect{d},\vect{T})$ to $(\vect{X},\vect{d}^{\prime},\vect{T})$,    and the conclusion follows from Theorem \ref{thm:invar}.
  \end{proof}

  \subsection{Invariance of measure-theoretic entropies}
  Generally speaking, measure-theoretic entropies are not preserved under (topological) conjugacies,
  even for equiconjugacies or in the autonomous cases.   With extra requirement on measures, we are able to show that measure-theoretic lower and upper entropies are  also equiconjugacy invariants.
  \begin{thm}\label{thm:invarmsrent}
    Let $\vect{\pi}$ be a semiconjugacy from $(\vect{X},\vect{T})$ to $(\vect{Y},\vect{R})$.
    Given $\mu \in M(X_{0})$, let $\nu=\pi_{0,*}\mu$ be the pushforward measure of $\mu$
    under $\pi_{0}$, that is, $$\nu(A)=\mu(\pi_{0}^{-1}A)\quad \text{for all Borel subsets}\ A \subset X_{0}.$$
    \begin{enumerate}[(1)]
      \item Suppose that $\vect{\pi}$ is an equisemiconjugacy.
            Then
            $$\entlow_{\mu}(\vect{T}) \geq \entlow_{\nu}(\vect{R}),\quad \entup_{\mu}(\vect{T}) \geq \entup_{\nu}(\vect{R}).$$
      \item Furthermore, suppose that $\vect{\pi}$ is an equiconjugacy.
            Then
            $$\entlow_{\mu}(\vect{T}) = \entlow_{\nu}(\vect{R}),\quad \entup_{\mu}(\vect{T}) = \entup_{\nu}(\vect{R}).$$
    \end{enumerate}
  \end{thm}
  \begin{proof}
    (1) Recall \eqref{eq:notebowballs} from the proof of Theorem \ref{thm:invar} that for every $\varepsilon>0$,
    there exists $\delta>0$ such that for all integer  $n >0$ and all $x \in X_{0}$,
    $$\pi_{0}(B_{n}^{\vect{T}}(x,\delta)) \subset B_{n}^{\vect{R}}(\pi_{0}x,\varepsilon).$$
    Hence, for all $\varepsilon>0$, there exists $\delta>0$ such that for all integer  $n >0$
    and all $y \in Y_{0}$, by choosing $x \in \pi_{0}^{-1}y$, we have
    $$\pi_{0}(B_{n}^{\vect{T}}(x,\delta)) \subset B_{n}^{\vect{R}}(y,\varepsilon).$$
    Fix $\varepsilon>0$ and  $\delta<\varepsilon$. It follows that
    $$\nu(B_{n}^{\vect{R}}(y,\varepsilon)) \geq \nu(\pi_{0}(B_{n}^{\vect{T}}(x,\delta))) = \mu(\pi_{0}^{-1}(\pi_{0}(B_{n}^{\vect{T}}(x,\delta)))) \geq \mu(B_{n}^{\vect{T}}(x,\delta)).$$
    Thus, for all integer  $n >0$, $y \in Y_{0}$ and every $x \in \pi_{0}^{-1}y$,
    \begin{equation}
      \frac{-\log{\nu(B_{n}^{\vect{R}}(y,\varepsilon))}}{n} \leq \frac{-\log{\mu(B_{n}^{\vect{T}}(x,\delta))}}{n}.
    \end{equation}
    Letting $n \to \infty$, we obtain that  $\entlow_{\mu}(\vect{T},x,\delta) \geq \entlow_{\nu}(\vect{R},y,\varepsilon)$
    and $\entup_{\mu}(\vect{T},x,\delta) \geq \entup_{\nu}(\vect{R},y,\varepsilon)$, and this implies
    $$\entlow_{\mu}(\vect{T},x) \geq \entlow_{\nu}(\vect{R},y) \quad \text{and} \quad \entup_{\mu}(\vect{T},x) \geq \entup_{\nu}(\vect{R},y),$$
    since $\delta$ tends to $0$ as $\varepsilon$ goes to $ 0$, where $x \in \pi_{0}^{-1}y$ for all $y \in Y_{0}$.

    Define $g: Y_{0} \to [0,+\infty]$ by
    $$g(y)=\entlow_{\nu}(\vect{R},y).$$
    It is clear that $g$ is a nonnegative Borel measurable function on $Y_{0}$.
    Thus,
    \begin{align*}
      \entlow_{\mu}(\vect{T}) &= \int_{X_{0}}\entlow_{\mu}(\vect{T},x)\dif{\mu(x)} \\
                              &\geq \int_{X_{0}}\entlow_{\nu}(\vect{R},\pi_{0}(x))\dif{\mu(x)} \\
                              &= \int_{X_{0}}(g \circ \pi_{0})\dif{\mu} \\
                              &= \int_{Y_{0}}g\dif{\nu} \\
                              &= \entlow_{\nu}(\vect{R}).
    \end{align*}
    The argument for $\entup_{\mu}(\vect{T}) \geq \entup_{\nu}(\vect{R})$ is completely identical.

    (2) Since $\vect{\pi}$ is an equiconjugacy from $(\vect{X},\vect{T})$ to $(\vect{Y},\vect{R})$, the inverse
    $\vect{\pi}^{-1}=\{\pi_{k}^{-1}\}_{k=0}^{\infty}$ is an equisemiconjugacy from $(\vect{Y},\vect{R})$ to $(\vect{X},\vect{T})$.
    Applying   (1) on $\vect{\pi}$ and $\vect{\pi}^{-1}$, we have the desired equalities.
  \end{proof}

Finally, we end  with an example which shows that the conjugacy is not sufficient to preserve the topological entropies. Hence  the equiconjugacy is appropriate for the  study the invariance of entropies and pressures in nonautonomous dynamical systems.
  \begin{exmp} \label{ex_nc}
For each integer  $k\geq 0$, we write $I_{k}=[0,2^{k}] \subset \mathbb{R},$
and  let $    T_{k}:I_{k} \to I_{k+1}$ be given by
$$
T_k(x) =2x.
$$
We give three different metrics  $d_e$, $d_u$ and $d_b$ on each $I_k$.   Let $ d_{e} $ be the  standard Euclidean metric restricted on $I_k$.  We write
$$
d_u(x,y)=\frac{1}{2^k} d_e(x,y), \qquad
d_b(x,y)= \frac{d_e(x,y)}{1+d_e(x,y)}
$$
for all $x,y\in I_k$. Then  we have three NDSs $(\vect{I},\vect{d}_{e},\vect{T})$,  $(\vect{I},\vect{d}_{u},\vect{T})$ and $(\vect{I},\vect{d}_{b},\vect{T})$.

By simple calculation, the topological entropies are given by
\begin{eqnarray*}
&&\entbow_{\vect{d}_{e}}(\vect{T},I_{0})=\entpac_{\vect{d}_{e}}(\vect{T},I_{0})=\log{2},  \\
&&\entbow_{\vect{d}_{u}}(\vect{T},I_{0})=    \entpac_{\vect{d}_{u}}(\vect{T},I_{0})=0, \\
&&\entbow_{\vect{d}_{b}}(\vect{T},I_{0}) =\entpac_{\vect{d}_{b}}(\vect{T},I_{0}) = \log{2}.
\end{eqnarray*}
Let $\leb$ be the Lebesgue measure on $I_{0}$. Then the measure-theoretic entropies are given by
\begin{eqnarray*}
&&\entlow_{\leb,\vect{d}_{e}}(\vect{T})=\entup_{\leb,\vect{d}_{e}}(\vect{T})=\log{2}, \\
&&\entlow_{\leb,\vect{d}_{u}}(\vect{T})=\entup_{\leb,\vect{d}_{u}}(\vect{T})=0,       \\
&&\entlow_{\leb,\vect{d}_{b}}(\vect{T})=\entup_{\leb,\vect{d}_{b}}(\vect{T})=\log{2}.
\end{eqnarray*}

Let  $\vect{\id}$ be the sequence of  identity mappings $\id_{I_{k}}$ on $I_{k}$.  Then  it is clear that
$\vect{\id}$ is an equiconjugacy from $(\vect{I},\vect{d}_{e},\vect{T})$ to $(\vect{I},\vect{d}_{b},\vect{T})$, however  $\vect{\id}$  is a conjugacy but not an equiconjugacy
 from $(\vect{I},\vect{d}_{e},\vect{T})$ to $(\vect{I},\vect{d}_{u},\vect{T})$. Therefore, conjugacy is not enough to guarantee the invariance of entropies and pressures.

\end{exmp}

\section{Billingsley Type Theorems}\label{sect:pfbthms}

In this section, we study the Billingsley type theorems for topological entropies and pressures which are similar to the Frostman's Lemma in fractal geometry (see \cite{Falconer1997}). These theorems relate measure-theoretic pressures to topological pressures.

First, we give the proof of the Billingsley type theorem for the Bowen pressure. 
  \begin{proof}[Proof of Theorem \ref{thm:btypethmbowen}]
    (1) Suppose that $\prelow_{\mu}(\vect{T},\vect{f},x) \leq s$ for all $x \in E$.  By Proposition \ref{equivdefentbowlem}, it is sufficient to prove that for every $t>s$,
$$
\msrbow^{t}(\vect{T},\vect{f},E)  <+\infty.
$$

 By Proposition \ref{prop:localepsilon},  we have that     for all $x \in E$ and all $\varepsilon>0$,
    $$
\prelow_{\mu}(\vect{T},\vect{f},x,\varepsilon) \leq \prelow_{\mu}(\vect{T},\vect{f},x) \leq s < t.
$$
Hence, for every $x \in E$ and every $\varepsilon>0$, there exists a strictly increasing sequence $\{n_{l}(x,\varepsilon)\}_{l=1}^{\infty}$
    such that for all integer $l>0$,
    $$\frac{-\log{\mu(B_{n_{l}}^{\vect{T}}(x,\varepsilon))}+S_{n_{l}}^{\vect{T}}\vect{f}(x)}{n_{l}} \leq t,$$
    which implies that
    $$\mu(B_{n_{l}}^{\vect{T}}(x,\varepsilon)) \geq \exp{\left(-n_{l}t + S_{n_{l}}^{\vect{T}}\vect{f}(x)\right)}.$$

    For each $N >0$ and $\varepsilon>0$, we write
    $$
\mathcal{F}=\{B_{n_{l}(x,\varepsilon)}^{\vect{T}}(x,\varepsilon): x \in E, n_{l}(x,\varepsilon) \geq N\},
$$
which is a cover of $E$.     By Lemma \ref{dn3rcovlem}, there exists a disjoint subfamily of $\mathcal{F}$, denoted by
    $\mathcal{F}^{\prime}=\{B_{n_{i}}^{\vect{T}}(x_{i},\varepsilon)\}_{i \in \mathcal{I}}$,
   such that
\begin{equation}\label{covere3e}
E \subset \bigcup_{B \in \mathcal{F}}B \subset \bigcup_{i \in \mathcal{I}}B_{n_{i}}^{\vect{T}}(x_{i},3\varepsilon),
\end{equation}
    and it is clear that
    $$
\mu(B_{n_{i}}^{\vect{T}}(x_{i},\varepsilon)) \geq \exp{\left(-n_{i}t + S_{n_{i}}^{\vect{T}}\vect{f}(x_{i})\right)}>0
$$
for all $i \in \mathcal{I}$. Since $\mu(X_{0})=1$ and $\mathcal{F}^{\prime}$ is disjoint, the index set $\mathcal{I}$  is  countable,  and we obtain that
    $$\sum_{i \in \mathcal{I}}\exp{\left(-n_{i}t + S_{n_{i}}^{\vect{T}}\vect{f}(x_{i})\right)}\leq \sum_{i \in \mathcal{I}}\mu(B_{n_{i}}^{\vect{T}}(x_{i},\varepsilon)) \leq 1.$$
By \eqref{eq:defmsrbow} and \eqref{covere3e}, it implies that
    $$
\msrbow_{N,3\varepsilon}^{t}(\vect{T},\vect{f},E) \leq \sum_{i \in \mathcal{I}}\exp{\left(-n_{i}t + S_{n_{i}}^{\vect{T}}\vect{f}(x_{i})\right)} \leq 1
$$
  for all $\varepsilon>0$ and  all  large  $N>0$. By Proposition \ref{prop:msrbowepsilon}, we have
$$
\msrbow^{t}(\vect{T},\vect{f},E) =\lim_{\varepsilon \to 0} \msrbow^{t}_\varepsilon (\vect{T},\vect{f},E)\leq 1 <+\infty,
$$
 and the conclusion  holds .

    (2) Suppose that $\prelow_{\mu}(\vect{T},\vect{f},x) \geq s$ for all $x \in E$ and  $\mu(E)>0$. It is sufficient to show that for every $t<s$,
$$
\prebow(\vect{T},\vect{f},E )\geq t,
$$

For every $l >0$,  we write
    $$
    E_{l}=\Big\{x \in E: \prelow_{\mu}(\vect{T},\vect{f},x,\frac{1}{l})>t \Big\}.
    $$
Since $\prelow_{\mu}(\vect{T},\vect{f},x) >t$,  by Proposition \ref{prop:localepsilon},  it is equivalent to write
    $$
    E_{l}=\Big\{x \in E: \prelow_{\mu}(\vect{T},\vect{f},x,\varepsilon)>t\ \text{for all}\ 0 < \varepsilon \leq \frac{1}{l } \Big\}.
    $$
It is clear that $\{E_{l}\}_{l=1}^{\infty}$ increases to $E$, and
$$
\lim_{l \to \infty}\mu(E_{l})=\mu(E)>0.
$$
Hence, there exists $L>0$ such that for all $l>L$,
$$
\mu(E_{l})>\frac{1}{2}\mu(E).
$$

For each fixed $l>L$, we write that for every integer $N>0$,
$$
E_{l,N}=\Big\{x \in E_{l}: \frac{  -\log{\mu(B_{n}^{\vect{T}}(x,\frac{1}{l}))} + S_{n}^{\vect{T}}\vect{f}(x)   }{n}>t\ \text{for all}\ n \geq N\Big\}.
$$
Since $\{E_{l,N}\}_{N=1}^{\infty}$ increases to $E_{l}$, we choose an integer $N_l>0$ such that
$$
\mu(E_{l,N_l})>\frac{1}{2}\mu(E_l)>\frac{1}{4}\mu(E)>0.
$$
It follows by Proposition~\ref{prop:localepsilon} that for all $x \in E_{l,N_l}$, we have that
$$
\frac{  -\log{\mu(B_{n}^{\vect{T}}(x,\varepsilon))} + S_{n}^{\vect{T}}\vect{f}(x)}{n}>t,
$$
and therefore
\begin{equation}\label{2}
      \mu(B_{n}^{\vect{T}}(x,\varepsilon)) \leq \exp{\left(-nt+S_{n}^{\vect{T}}\vect{f}(x)\right)}
    \end{equation}
for all  $0<\varepsilon\leq\frac{1}{l}$ and all  $n \geq N_l$.

For each $\varepsilon$ such that   $0<\varepsilon\leq\frac{1}{l}$ and $N \geq N_{l}$, arbitrarily choose a countable $(N,\frac{\varepsilon}{2})$-cover  $\left\{B_{n_{i}}^{\vect{T}}\left(x_{i},\frac{\varepsilon}{2}\right)\right\}_{i = 1}^{\infty}$ of $E_{l,N_l}$.
We assume that $B_{n_{i}}^{\vect{T}}\left(x_{i},\frac{\varepsilon}{2}\right) \cap E_{l,N_l} \neq \emptyset$ for each $i$. For each integer $i>0$, we choose $y_{i} \in E_{l,N_l} \cap B_{n_{i}}^{\vect{T}}\left(x_{i},\frac{\varepsilon}{2}\right)$, and  it is clear that  $B_{n_{i}}^{\vect{T}}\left(x_{i},\frac{\varepsilon}{2}\right) \subset B_{n_{i}}^{\vect{T}}(y_{i},\varepsilon)$.
    Combining this  with \eqref{2}, we have that
    \begin{align*}
      \sum_{i=1}^{\infty}\exp{\Big(-n_{i}t + \sup_{y \in B_{n_{i}}^{\vect{T}}(x_{i},\frac{\varepsilon}{2})} S_{n_{i}}^{\vect{T}}\vect{f}(y)\Big)} &\geq \sum_{i=1}^{\infty}\exp{\left(-n_{i}t+S_{n_{i}}^{\vect{T}}\vect{f}(y_{i})\right)} \\
  &\geq \sum_{i=1}^{\infty}\mu(B_{n_{i}}^{\vect{T}}(y_{i},\varepsilon)) \\
 &\geq \mu(E_{l,N_l}).
    \end{align*}
By  \eqref{eq:defmsrbpp}, it follows that
$$
\msrbpp_{N,\frac{\varepsilon}{2}}^{t}(\vect{T},\vect{f},E_{l,N_l}) \geq \mu(E_{l,N_l})
$$
 for all  $0<\varepsilon\leq\frac{1}{l}$ and  all $N \geq N_l$. Since $l>L$,   we have
$$
\msrbpp_{\frac{\varepsilon}{2}}^{t}(\vect{T},\vect{f},E_{l,N_l}) \geq \frac{1}{4}\mu(E)>0.
$$
By \eqref{def_QBT},  it follows that $\prebpp(\vect{T},\vect{f},E_{l,N_l}, \frac{\varepsilon}{2}) \geq t$ for all  $0<\varepsilon\leq\frac{1}{l}$. Since $\vect{f}$ is equicontinuous, by Proposition \ref{equivdefprebpplem}, we obtain that
$$
\prebow(\vect{T},\vect{f},E_{l,N_l})=\lim_{\varepsilon \to 0}\prebpp(\vect{T},\vect{f},E_{l,N_l}, \frac{\varepsilon}{2})   \geq t .
$$
Since $E_{l,N_l}\subset E$, by Proposition \ref{Prop_PBPP}, we have that
$$
\prebow(\vect{T},\vect{f},E )\geq \prebow(\vect{T},\vect{f},E_{l,N_l})\geq t,
$$
 which completes the proof.
  \end{proof}

Next, we  prove  the Billingsley type theorem for the packing pressure.  

  \begin{proof}[Proof of Theorem \ref{thm:btypethmpacking}]
    (1) Suppose that $\preup_{\mu}(\vect{T},\vect{f},x) \leq s$ for all $x \in E$. 

For every  $t>s$, by Proposition \ref{prop:localepsilon},   we have that
$$
\preup_{\mu}(\vect{T},\vect{f},x,\varepsilon) \leq \preup_{\mu}(\vect{T},\vect{f},x) \leq s < \frac{s+t}{2}
$$
    for all $x \in E$ and $\varepsilon>0$, and
    there exists $N>0$ such that for all $n \geq N$,
$$
\frac{-\log{\mu(B_{n}^{\vect{T}}(x,\varepsilon))}+S_{n}^{\vect{T}}\vect{f}(x)}{n}\leq \frac{s+t}{2}.
$$
Hence, for all $x \in E$ and all  $\varepsilon>0$,   there exists $N >0$ such that for all $n \geq N$,
\begin{equation}\label{inequgeeP}
\mu(B_{n}^{\vect{T}}(x,\varepsilon)) \geq \exp{\Big(-\frac{n(s+t)}{2} + S_{n}^{\vect{T}}\vect{f}(x)\Big)}.
\end{equation}

Fix $\varepsilon>0$. For each $N >0$, we write
    $$
E_{N}=\Big\{x \in E: \mu(B_{n}^{\vect{T}}(x,\varepsilon)) \geq \exp{\Big(-\frac{n(s+t)}{2} + S_{n}^{\vect{T}}\vect{f}(x)\Big)} \ \text{for all}\ n \geq N\Big\}.
$$
    It is clear that $E=\bigcup_{N=1}^{\infty}E_{N}$.

Fix $N >0$. For each $N^{\prime} \geq N$ and each countable $(N',\varepsilon)$-packing $\{\overline{B}_{n_{i}}^{\vect{T}}(x_{i},\varepsilon)\}_{i=1}^{\infty}$ of $E_N$,
by \eqref{inequgeeP}, we have that
    \begin{align*}
      \sum_{i=1}^{\infty}\exp{\left(-n_{i}t + S_{n_{i}}^{\vect{T}}\vect{f}(x_{i})\right)}  &= \sum_{i=1}^{\infty}\e^{-\frac{n_{i}(t-s)}{2}}\exp{\Big(-\frac{n_{i}(s+t)}{2} + S_{n_{i}}^{\vect{T}}\vect{f}(x_{i})\Big)} \\
 &\leq\e^{-\frac{N^{\prime}(t-s)}{2}}\sum_{i=1}^{\infty}\exp{\Big(-\frac{n_{i}(s+t)}{2} + S_{n_{i}}^{\vect{T}}\vect{f}(x_{i})\Big)} \\
   &\leq \e^{-\frac{N^{\prime}(t-s)}{2}}\sum_{i=1}^{\infty}\mu(B_{n_{i}}^{\vect{T}}(x_{i},\varepsilon)) \\
&\leq \e^{-\frac{N^{\prime}(t-s)}{2}}.
    \end{align*}
By \eqref{eq:defmsrPack}, it follows that
$$
\msrpac_{N^{\prime},\varepsilon}^{t}(\vect{T},\vect{f},E_{N}) \leq \e^{-\frac{N^{\prime}(t-s)}{2}}.
$$
By the definition of $\msrpac_{\infty,\varepsilon}^{t}$,  it is clear  that  $\msrpac_{\infty,\varepsilon}^{t}(\vect{T},\vect{f},E_{N})=0$.  Since $E=\bigcup_{N=1}^{\infty}E_{N}$,  by \eqref{def_pes}, we have that
$$
\msrpac_{\varepsilon}^{t}(\vect{T},\vect{f},E) \leq \sum_{N=1}^{\infty}\msrpac_{\infty, \varepsilon}^{t}(\vect{T},\vect{f},E_{N}) =0.
$$
By \eqref{def_PP},  we have that $\prepac(\vect{T},\vect{f},E,\varepsilon) \leq t$, and the inequality $\prepac(\vect{T},\vect{f},E) \leq t$ follows from Proposition \ref{prop:msrpacepsilon},

(2) Suppose that $\preup_{\mu}(\vect{T},\vect{f},x) \geq s$ for all $x \in E$ and $\mu(E)>0$. By  Proposition \ref{equivdefprepaclem} and \eqref{def_QP},   it is sufficient to show that for every $t<s$,  there exists $\varepsilon_0>0$ such that for all $0<\varepsilon <\varepsilon_0$, we have that
$$
\msrppp_{\frac{\varepsilon}{5}}^{t}(\vect{T},\vect{f},E)=+\infty.
$$
To show this,  by Proposition \ref{prop_NPclosure}, it suffices to prove that  for all Borel subsets $E^{*} \subset E$ satisfying $\mu(E^{*})>0$, we have that
 $$
\msrppp_{\infty, \frac{\varepsilon}{5} }^{t}(\vect{T},\vect{f},E^{*})=+\infty.
$$

    Fix $t<s$. For every $x \in E$, since $\preup_{\mu}(\vect{T},\vect{f},x)>t$, by Proposition~\ref{prop:localepsilon}, there exists an $\varepsilon_{0}>0$ such that for all $0<\varepsilon\leq\varepsilon_{0}$,
    $$
    \varlimsup_{n \to \infty}\frac{-\log{\mu(B_{n}^{\vect{T}}(x, \varepsilon  ))} + S_{n}^{\vect{T}}\vect{f}(x)}{n} > \frac{s+t}{2}.
    $$
Since  the LHS of the above inequality is monotone with respect to $\varepsilon$, there exists $N_{0}(x,\varepsilon_{0}) >0$ such that for all $0<\varepsilon\leq\varepsilon_{0}$ and $N \geq N_{0}(x,\varepsilon_{0})$,
    $$\sup_{n \geq N}\frac{-\log{\mu(B_{n}^{\vect{T}}(x, \varepsilon ))} + S_{n}^{\vect{T}}\vect{f}(x)}{n} > \frac{s+t}{2},$$
    which implies that there exists an integer $n \geq N$ for every given  $N >0$ such that for all $0<\varepsilon\leq\varepsilon_{0}$,
\begin{equation}\label{ineq_uleePP}
\mu(B_{n}^{\vect{T}}(x, \varepsilon )) < \exp{\Big(-\frac{n(s+t)}{2}+S_{n}^{\vect{T}}\vect{f}(x)\Big)}.
\end{equation}

Given a Borel subset $E^{*} \subset E$ with $\mu(E^{*})>0$, for each fixed integer  $l >0$, we write
    $$
E_{l}^{*}=\Big\{x \in E^{*}: \preup_{\mu}(\vect{T},\vect{f},x, \frac{1}{l}  ) > \frac{s+t}{2}\Big\}.
$$
By Proposition~\ref{prop:localepsilon}, it is equivalent to write
    $$
E_{l}^{*}=\Big\{x \in E^{*}: \preup_{\mu}(\vect{T},\vect{f},x, \varepsilon   ) > \frac{s+t}{2}\ \text{for all}\ 0<\varepsilon\leq\frac{1}{l}\Big\},
$$
and it is clear that $E_{l}^{*}$ increases to $E^{*}$.  Hence, for all sufficiently large integer $l>0$, we have $\mu(E_{l}^{*})>0$.
    For each integer $n>0$, we write
    $$
E_{l,n}^{*}=\Big\{x \in E_{l}^{*}: \mu(B_{n}^{\vect{T}}(x,  \frac{1}{l}  )) < \exp{\Big(-\frac{n(s+t)}{2}+S_{n}^{\vect{T}}\vect{f}(x)\Big)}\Big\}.
$$

For each fixed integer $M>0$, since $E_{l}^{*} \subset E$, by \eqref{ineq_uleePP},  we have that $E_{l}^{*}=\bigcup_{n=M}^{\infty}E_{l,n}^{*}$ and  $\mu\left(\bigcup_{n=M}^{\infty}E_{l,n}^{*}\right)=\mu(E_{l}^{*})>0$.
This implies that there exists $n \geq M$ such that
\begin{equation}\label{eq:msrElnstar}
    \mu(E_{l,n}^{*}) \geq \frac{1}{n(n+1)}\mu(E_{l}^{*}).
\end{equation}
Otherwise, we   have $\mu(E_{l,n}^{*}) < \frac{1}{n(n+1)}\mu(E_{l}^{*})$ for all $n \geq M$,  and it implies that
    $$
\mu(E_{l}^{*}) \leq \sum_{n=M}^{\infty}\mu(E_{l,n}^{*}) \leq \sum_{n=M}^{\infty}\frac{1}{n(n+1)}\mu(E_{l}^{*})=\frac{1}{M}\mu(E_{l}^{*}),
$$
 which contradicts the fact that $\mu(E_{l}^{*})>0$.

Choose an integer $l>0$ such that $\frac{1}{l}<\varepsilon_{0}$.
For each $0<\varepsilon\leq\varepsilon_{0}$ and $N>0$, let $M$ be an integer satisfying $M>N$.  Fix an integer $n \geq M$ satisfying \eqref{eq:msrElnstar}.
For the family $   \mathcal{B}_{l,n}=\{\overline{B}_{n}^{\vect{T}}\left(x, \frac{\varepsilon}{5}\right): x \in E_{l,n}^{*}\}   $, by Lemma \ref{coveringlem}, there exists a countable subfamily
    $   \{\overline{B}_{n}^{\vect{T}}\left(x_{i}, \frac{\varepsilon}{5}   \right)\}_{i=1}^{\infty}   $,
    which is a $(N,\frac{\varepsilon}{5})$-packing   of $E_{l,n}^{*}$, such that
    $$
E_{l,n}^{*} \subset \bigcup_{B \in \mathcal{B}_{l,n}}B \subset \bigcup_{i=1}^{\infty}\overline{B}_{n}^{\vect{T}}(x_{i}, \varepsilon).
$$
By \eqref{eqdef_pp} and \eqref{ineq_uleePP}, we have that
    \begin{align*}
       \msrppp_{N,\frac{\varepsilon}{5}}^{t}(\vect{T},\vect{f},E_{l,n}^{*}) &\geq \sum_{i=1}^{\infty}\exp{\Big(-nt +  \sup_{y \in \overline{B}_{n}^{\vect{T}}(x_{i},\frac{\varepsilon}{5})}S_{n}^{\vect{T}}\vect{f}(y)     \Big)} \\
 &\geq \e^{\frac{n(s-t)}{2}}\sum_{i=1}^{\infty}\exp{\Big(-\frac{n(s+t)}{2} + S_{n}^{\vect{T}}\vect{f}(x_{i})\Big)} \\
& \geq \e^{\frac{n(s-t)}{2}}\sum_{i=1}^{\infty}  \mu(B_{n}^{\vect{T}}(x_{i},\varepsilon   )) \\
&\geq \e^{\frac{n(s-t)}{2}}\mu(E_{l,n}^{*}) \\
 &\geq \frac{1}{n(n+1)} \e^{\frac{n(s-t)}{2}}\mu(E_{l}^{*}).
    \end{align*}

Since $E_{l,n}^{*}\subset E_{l}^{*}$, by Proposition \ref{prop_Nesinc}, it follows that for every $N>0$, there exists an integer $n>N$ such that
$$
 \msrppp_{N,\frac{\varepsilon}{5}}^{t}(\vect{T},\vect{f},E_{l}^{*}) \geq \frac{1}{n(n+1)} \e^{\frac{n(s-t)}{2}}\mu(E_{l}^{*}).
$$
Since $\lim_{n\to\infty}\frac{1}{n(n+1)}\e^{\frac{n(s-t)}{2}} =+\infty$ and $\mu(E_{l}^{*})>0$, we have that
$$
\msrppp_{\infty,\frac{\varepsilon}{5}}^{t}(\vect{T},\vect{f},E_{l}^{*})=\lim_{N \to \infty}\msrppp_{N,\frac{\varepsilon}{5}}^{t}(\vect{T},\vect{f},E_{l}^{*})=+\infty.
$$
Since $E_{l}^{*}\subset E^{*}$, by Proposition \ref{prop_Nesinc}, it follows that
 $$
\msrppp_{\infty,\frac{\varepsilon}{5}}^{t}(\vect{T},\vect{f},E^{*})=+\infty,
$$
 which completes the proof.
  \end{proof}

  \section{Variational Principles}\label{sect:pfvps}
  In this section, we study the variational principles for topological entropies and pressures.

    \subsection{Variational principle for the Bowen pressure}\label{ssect:pfvb}

    First, we show a variational inequality that the measure-theoretic pressure is bounded above by the Bowen pressure as a direct consequence of Theorem~\ref{thm:btypethmbowen}.
    \begin{lem}\label{lem:varneqprebow}
      Given an NDS $(\vect{X},\vect{T})$ and a non-empty Borel subset $E\subset X_{0}$,
      let $\mu$ be a Borel probability measure on $X_{0}$ with $\mu(E)=1$ and let $\vect{f} \in \vect{C}(\vect{X},\mathbb{R})$ be equicontinuous.
      Then
      \begin{equation}\label{eq:varneqprebow}
        \prelow_{\mu}(\vect{T},\vect{f}) \leq \prebow(\vect{T},\vect{f},E).
      \end{equation}
    \end{lem}
    \begin{proof}
      Since the conclusion holds if $\prelow_{\mu}(\vect{T},\vect{f})=-\infty$, we assume that $\prelow_{\mu}(\vect{T},\vect{f}) > -\infty$.

      Fix $\varepsilon>0$ and $s<\prelow_{\mu}(\vect{T},\vect{f})$.
      Since $s<\int_{X_{0}}\prelow_{\mu}(\vect{T},\vect{f},x)\dif{\mu(x)}$, there exists a Borel set $A_{s} \subset X_{0}$ with $\mu(A_{s})>0$ such that
      $\prelow_{\mu}(\vect{T},\vect{f},x) \geq s$ for all $x \in A_{s}$.
      Since $\mu(X_0\setminus E)=0$, it is clear that $\mu(E \cap A_{s})>0$.
      Since $\vect{f}$ is equicontinuous, by Theorem \ref{thm:btypethmbowen}(2), we have $\prebow(\vect{T},\vect{f},E \cap A_{s}) \geq s$.
      By Proposition \ref{Prop_PBPP}, it  follows  that
      $$
      \prebow(\vect{T},\vect{f},E) \geq \prelow_{\mu}(\vect{T},\vect{f},E \cap A_{s}) \geq s.
      $$
      Hence, by the arbitrariness of $s<\prelow_{\mu}(\vect{T},\vect{f})$, the inequality \eqref{eq:varneqprebow} holds.
    \end{proof}

    Next, we give a lemma which is  analogous to the classical Frostman's lemma in compact metric spaces, and
  we adapt the argument of Howroyd \cite[Thm.2]{Howroyd1995} (see also \cite[Thm.8.17]{Mattila1995}) to our setting, where they studied the properties of Hausdorff measures and dimensions by applying Frostman's lemma.
    \begin{lem}\label{dynflemw}
      Given an NDS $(\vect{X},\vect{T})$ and a non-empty compact subset $K\subset X_{0}$,
      suppose that $\mathscr{W}_{N,\varepsilon}^{s}(\vect{T},\vect{f},K)>0$ for $s \in \mathbb{R}$, $N>0$ and $\varepsilon>0$.
      Then there exists a Borel probability measure $\mu$ on $X_{0}$ and a constant $C>0$
      such that $\mu(K)=1$ and
      \begin{equation}\label{eq:Frostmanestm}
        \mu(B_{n}^{\vect{T}}(x,\varepsilon)) \leq \frac{1}{C}\exp{\Big(-sn+\sup_{y \in B_{n}^{\vect{T}}(x,\varepsilon)}S_{n}^{\vect{T}}\vect{f}(y)\Big)}
      \end{equation}
   for all $x \in X_{0}$ and all $n \geq N$.
    \end{lem}
    \begin{proof}
  We write
   $$
  C=\mathscr{W}_{N,\varepsilon}^{s}(\vect{T},\vect{f},K).
  $$

   Since $K \subset X_{0}$ is compact, for every $(N,\varepsilon)$-cover $\{B_{n_{i}}^{\vect{T}}(x_{i},\varepsilon)\}_{i\in \mathcal{I}}$,  there  exists a finite  subcover of $K$. Hence, $\mathscr{M}_{N,\varepsilon}^{s}(\vect{T},\vect{f},K)<+\infty$,  and it follows that $\mathscr{W}_{N,\varepsilon}^{s}(\vect{T},\vect{f},K) \leq \msrbpp_{N,\varepsilon}^{s}(\vect{T},\vect{f},Z)<+\infty$.

    We define a functional $\psi: C(X_{0},\mathbb{R}) \to \mathbb{R}$ by
  \begin{equation}\label{def_PsyW}
  \psi(g)=\frac{1}{C}\mathscr{W}_{N,\varepsilon}^{s}(\chi_{K} \cdot g)
  \end{equation}
for every $g \in C(X_{0},\mathbb{R})$.  By Proposition \ref{prop:prebwprop}, it is straightforward that $\psi$ satisfies the following properties:
      \begin{enumerate}[(1)]
        \item $\psi(ag)=a\psi(g)$ for any $a \geq 0$ and $g \in C(X_{0},\mathbb{R})$ (positive homogeneity);
        \item $\psi(f+g) \leq \psi(f)+\psi(g)$ for any $f,g \in C(X_{0},\mathbb{R})$ (subadditivity);
        \item $\psi(\mathbf{1})=1$, where $\mathbf{1} \in C(X)$ denotes the constant function $1$;
        \item $0 \leq \psi(g) \leq \|g\|_{\infty}$ for any $g \in C(X_{0},\mathbb{R})$;
        \item $\psi(f)=0$ for $f \in C(X_{0},\mathbb{R})$ with $f \leq 0$.
      \end{enumerate}

  Note that $\varPsi: t \mapsto t\psi(\mathbf{1})$ for all $t \in \mathbb{R}$  is a linear functional on $\mathbb{R}$, and that we may identify $\mathbb{R}$ as the subspace of real-valued constant functions of  $C(X_{0},\mathbb{R})$.  By the Hahn-Banach theorem (see \cite[Thm.5.6]{Folland1999}),
  we extend $\varPsi$ to a linear functional $\varLambda: C(X_{0},\mathbb{R}) \to \mathbb{R}$ with
  \begin{equation}\label{pro_LPg}
  \varLambda(\mathbf{1})=\psi(\mathbf{1})=1\quad \text{and}\ -\psi(-g) \leq \varLambda(g) \leq \psi(g),
  \end{equation}
      for all $g \in C(X_{0},\mathbb{R})$.
      It follows that $\varLambda(g) \geq -\psi(-g) \geq 0$ for all nonnegative continuous functions $g$ on $X_{0}$.
      Hence, $\varLambda$ is a positive linear functional on $C(X_{0},\mathbb{R})$ with $\varLambda(\mathbf{1})=1$.

  By the Riesz representation theorem (see \cite[Thm.7.2]{Folland1999}), there exists
      a unique Radon probability measure $\mu$ on $X_{0}$ such that
  $$
  \varLambda(g)=\int_{X_{0}}g\dif{\mu},
  $$
  for all $g \in C(X_{0},\mathbb{R})$.
  Recall 	 that a Radon measure  is a Borel measure that is finite on all compact sets, outer regular on all Borel sets,
      and inner regular on all open sets.

      Next, we verify that $\mu(K)=1$ and that $\mu$ satisfies \eqref{eq:Frostmanestm}.

      For every compact subset $F \subset X_{0} \setminus K$, there exists an open set $U \subset X_{0}$ such that $F \subset U$ and $U\cap K=\emptyset$. By the Urysohn lemma (see \cite[Thm.4.32]{Folland1999}),
      there exists $g \in C(X_{0},\mathbb{R})$ such that $0 \leq g \leq 1$, $g(x)=1$ for $x \in F$
      and $g(x)=0$ for $x \in X_{0}\setminus U$. Since $K\subset X_0\setminus U$, it is clear  that  $g(x)=0$ for $x \in K$. Since $g \cdot \chi_{K}=0$ for $x \in X_0$, by \eqref{def_PsyW} and Proposition \ref{prop:prebwprop}, we have that $\psi(g)=0$.
      Hence we have that
  $$
  \mu(F) \leq \varLambda(g) \leq \psi(g)=0,
  $$
  and it follows that  $\mu(X_{0} \setminus K)=0 $. Immediately, we obtain that  $\mu(K)=1$.

      Given $B_{n}^{\vect{T}}(x,\varepsilon)$ with $x \in X_{0}$ and $n \geq N$, arbitrarily choose a compact subset $F \subset B_{n}^{\vect{T}}(x,\varepsilon)$. By the Urysohn lemma,
      there exists $g \in C(X_{0},\mathbb{R})$ such that $0 \leq g \leq 1$, $g(y)=1$ for $y \in F$
      and $g(y)=0$ for $y \in X_{0} \setminus B_{n}^{\vect{T}}(x,\varepsilon)$. Hence we have that
  $$
  \chi_{B_{n}^{\vect{T}}(x,\varepsilon)}\geq g \cdot \chi_{K} .
  $$
   By \eqref{def_WBPP}, this implies that
  $$
  \mathscr{W}_{N,\varepsilon}^{s}(\vect{T},\vect{f}, g \cdot \chi_{K}) \leq  \exp{\Big(-ns+\sup_{y \in B_{n}^{\vect{T}}(x,\varepsilon)}S_{n}^{\vect{T}}\vect{f}(y)\Big)}.
  $$
  Meanwhile, by \eqref{def_PsyW} and \eqref{pro_LPg}, we have that
      $$
  \mu(F) \leq \varLambda(g) \leq \psi(g) = \frac{1}{C}\mathscr{W}_{N,\varepsilon}^{s}(\vect{T},\vect{f},g \cdot \chi_{K}).
  $$
  This implies that
  $$
  \mu(F) \leq \frac{1}{C}\exp{\Big(-ns+\sup_{y \in B_{n}^{\vect{T}}(x,\varepsilon)}S_{n}^{\vect{T}}\vect{f}(y)\Big)}
  $$
  for every compact  $F \subset B_{n}^{\vect{T}}(x,\varepsilon)$.

  Since $\mu$ is a Radon measure, the inequality \eqref{eq:Frostmanestm} follows immediately from the inner regularity of $\mu$, and we complete the proof.
    \end{proof}

  Now,  we are ready to prove the variational principle for the Bowen pressure.
    \begin{proof}[Proof of Theorem \ref{thm:varprinciplebowen}]
      By Lemma~\ref{lem:varneqprebow}, it remains to prove that
      \begin{equation}\label{eq:varprinbowpart2}
        \prebow(\vect{T},\vect{f},K) \leq \sup\{\prelow_{\mu}(\vect{T},\vect{f}): \mu \in M(X_{0}). \mu(K)=1\},
      \end{equation}
  We assume that $\prebow(\vect{T},\vect{f},K) > -\infty$. It is sufficient to prove that for every $\alpha>0$, there exists a Borel probability measure $\mu$ with $\mu(K)=1$ such that
  $$
        \prebow(\vect{T},\vect{f},K)-\alpha  \leq \prelow_{\mu}(\vect{T},\vect{f}).
  $$

    Fix $\alpha>0$.   Since $\vect{f}$ is equicontinuous, we have that
  $$
  \lim_{\varepsilon\to0}\varliminf_{n \to \infty}\frac{1}{n}\Big(S_{n}^{\vect{T}}\vect{f}(x)-\sup_{y \in B_{n}^{\vect{T}}(x,\varepsilon)}S_{n}^{\vect{T}}\vect{f}(y)\Big)=0
  $$
  for all $x \in X_{0}$, and it follows that for sufficiently small $\varepsilon>0$,
  \begin{equation}\label{ineq_LSA}
  \varliminf_{n \to \infty}\frac{1}{n}\Big(S_{n}^{\vect{T}}\vect{f}(x)-\sup_{y \in B_{n}^{\vect{T}}(x,\varepsilon)}S_{n}^{\vect{T}}\vect{f}(y)\Big) > -\frac{\alpha}{2},
  \end{equation}
   for all $x \in X_{0}$.

  We write $s=\prebow(\vect{T},\vect{f},K)-\frac{\alpha}{2}$.
  By Proposition \ref{prop:equivprebwetprebpp}, we have that
  $$
  \prebw(\vect{T},\vect{f},K)=\prebow(\vect{T},\vect{f},K)>s.
  $$
  By \eqref{def_PBW} and the definition of $\mathscr{W}_{\varepsilon}^{s}$,  there exist reals $\varepsilon_0>0$ and  $N_0 >0$ such that
  $$
  \mathscr{W}_{N,\varepsilon}^{s}(\vect{T},\vect{f},K)>0
  $$
  for all $0<\varepsilon<\varepsilon_0$ and all $N>N_0$.  Fix  $0<\varepsilon<\varepsilon_0$ and  $N>N_0$.
  By  Lemma \ref{dynflemw},  there exists $\mu \in M(X_{0})$ with $\mu(K)=1$
      such that for all $x \in X_{0}$ and $n \geq N$,
      $$\mu(B_{n}^{\vect{T}}(x,\varepsilon)) \leq \frac{1}{C}\exp{\Big(-ns+\sup_{y \in B_{n}^{\vect{T}}(x,\varepsilon)}S_{n}^{\vect{T}}\vect{f}(y)\Big)},$$
      where $C=\mathscr{W}_{N,\varepsilon}^{s}(\vect{T},\vect{f},K)$. Since $s=\prebow(\vect{T},\vect{f},K)-\frac{\alpha}{2}$, by Proposition \ref{prop:localepsilon} and \eqref{ineq_LSA}, we have that
      \begin{align*}
        \prelow_{\mu}(\vect{T},\vect{f},x) &\geq \prelow_{\mu}(\vect{T},\vect{f},x,\varepsilon) \\
                                          &\geq \varliminf_{n \to \infty}\frac{-1}{n}\Big(-ns+\sup_{y \in B_{n}^{\vect{T}}(x,\varepsilon)}S_{n}^{\vect{T}}\vect{f}(y)+\log{C}+S_{n}^{\vect{T}}\vect{f}(x)\Big) \\
                                          &= s + \varliminf_{n \to \infty} \frac{1}{n}\Big(S_{n}^{\vect{T}}\vect{f}(x)-\sup_{y \in B_{n}^{\vect{T}}(x,\varepsilon)}S_{n}^{\vect{T}}\vect{f}(y)\Big) \\
                                          &> s-\frac{\alpha}{2} \\
                                          &= \prebow(\vect{T},\vect{f},K)-\alpha
      \end{align*}
      for all $x \in X_{0}$. It follows that
      $$\prelow_{\mu}(\vect{T},\vect{f})=\int_{X_{0}}\prelow_{\mu}(\vect{T},\vect{f},x)\dif{\mu(x)}>\prebow(\vect{T},\vect{f},K)-\alpha,
  $$
  and we complete the proof.
    \end{proof}

    \subsection{Variational principle for the packing pressure}\label{ssect:pfvp}

    First, the following  variational inequality   is a direct consequence of Theorem~\ref{thm:btypethmpacking}.
    \begin{lem}\label{lem:varneqprepac}
      Given an NDS $(\vect{X},\vect{T})$ and a non-empty Borel subset $E\subset X_{0}$,
      let $\mu$ be a Borel probability measure on $X_{0}$ with $\mu(E)=1$ and $\vect{f} \in \vect{C}(\vect{X},\mathbb{R})$ be equicontinuous.
      Then
      \begin{equation}\label{eq:varneqprepac}
        \preup_{\mu}(\vect{T},\vect{f}) \leq \prepac(\vect{T},\vect{f},E).
      \end{equation}
    \end{lem}
    \begin{proof}
      Without loss of generality, we assume that $\preup_{\mu}(\vect{T},\vect{f},E) > -\infty$.

      Fix $s<\preup_{\mu}(\vect{T},\vect{f})$.
      There exists a Borel set $A_{s} \subset E$ with $\mu(A_{s})>0$
      such that $\preup_{\mu}(\vect{T},\vect{f},x)>s$ for all $x \in A_{s}$.
      Since $\mu(E)=1$, we have that $\mu(E \cap A_{s})>0$. Since  $\vect{f} \in \vect{C}(\vect{X},\mathbb{R})$ is equicontinuous, by Theorem \ref{thm:btypethmpacking}(2), we have $\prepac(\vect{T},\vect{f},E \cap A_{s}) \geq s$.
      By Proposition \ref{Prop_PBPP}, it follows that
      $$
      \prepac(\vect{T},\vect{f},E) \geq \prepac(\vect{T},\vect{f},E \cap A_{s}) \geq s.
      $$
      Since it holds for all $s<\preup_{\mu}(\vect{T},\vect{f})$, the inequality \eqref{eq:varneqprepac} follows.
    \end{proof}

    The variational principle for the packing pressure is much more difficult to study than the one for the Bowen pressure, and the following lemma plays a key role in the proof.
     
   \begin{lem}\label{badlem}
      Given  $\vect{f} \in \vect{C}_{b}(\vect{X},\mathbb{R})$ and $Z \subset X_{0}$, for  $\varepsilon>0$ and $s>\|\vect{f}\|$, assume that $\msrpac_{\infty,\varepsilon}^{s}(\vect{T},\vect{f},Z)=+\infty$.
      Then for real $a,b$ with $0 \leq a < b < +\infty$ and $M >0$,
      there exists a finite disjoint collection $\{\overline{B}_{m_{i}}^{\vect{T}}(x_{i},\varepsilon)\}_{i\in \mathcal{I}}$
      with $x_{i} \in Z$ and $m_{i} > M$ for all $i\in \mathcal{I} $ such that
      $$a < \sum_{i\in \mathcal{I} }\exp{\left(-m_{i}s+S_{m_{i}}^{\vect{T}}\vect{f}(x_{i})\right)} < b.$$
    \end{lem}
    \begin{proof}
  Since $s>\|\vect{f}\|$, there exists $N_{1}>M$ such that
  $$
  \exp{\left(-N_{1}s+N_{1}\|\vect{f}\|\right)}<b-a.
  $$
  Since $\msrpac_{\infty,\varepsilon}^{s}(\vect{T},\vect{f},Z)=+\infty$, there exists $N_2>N_1$ such that for all $N>N_2$,
      $\msrpac_{N,\varepsilon}^{s}(\vect{T},\vect{f},Z)>b$. By \eqref{eq:defmsrPack}, there exists a finite $(N,\varepsilon)$-packing $\{\overline{B}_{m_{i}}^{\vect{T}}(x_{i},\varepsilon)\}_{i\in \mathcal{I}_{0}}$ of $Z$
      such that
      $$
      \sum_{i\in \mathcal{I}_{0}}\exp{\left(-m_{i}s+S_{m_{i}}^{\vect{T}}\vect{f}(x_{i})\right)}>b.
      $$

     Since $s>\|\vect{f}\|$ and $m_i \geq N>N_1$ for every $i\in \mathcal{I}_{0}$, by \eqref{eq_fnorm}, we have that
  $$
  \exp{\left(-m_{i}s+S_{m_{i}}^{\vect{T}}\vect{f}(x_{i})\right)} \leq  \exp{\left(-N_{1}s+N_{1}\|\vect{f}\|\right)}<b-a.
  $$
 Hence there exists a subset  $\mathcal{I} \subset \mathcal{I}_0$ such that
      $$
      a < \sum_{i\in \mathcal{I}}\exp{\left(-m_{i}s+S_{m_{i}}^{\vect{T}}\vect{f}(x_{i})\right)} < b,
      $$
    which completes the proof of Lemma~\ref{badlem}.
    \end{proof}

Finally, we are ready to prove the variational principle for the packing pressure. 
    \begin{proof}[Proof of Theorem \ref{thm:varprinciplepacking}]
      By Lemma~\ref{lem:varneqprepac}, we only need to show that
  $$
        \prepac(\vect{T},\vect{f},K) \leq \sup\{\preup_{\mu}(\vect{T},\vect{f}): \mu \in M(X_{0}), \mu(K)=1\}.
  $$
   We assume that $\prepac(\vect{T},\vect{f},K)>-\infty$, and it is sufficient to prove that for every  $s<\prepac(\vect{T},\vect{f},K)$, there is  a Borel probability measure $\mu$ with $\mu(K)=1$   such that 
   $$
   \preup_{\mu}(\vect{T},\vect{f}) > s.
   $$

  Choose  $\varepsilon_{0}>0$ such that $s<\prepac(\vect{T},\vect{f},K,\varepsilon_{0})$. Fix $t$ such that $s < t < \prepac(\vect{T},\vect{f},K,\varepsilon_{0})$.    We construct the following sequences inductively:
      \begin{enumerate}[1.]
        \item a sequence $\{K_{i}\}_{i=1}^{\infty}$ of finite sets $K_{i} \subset K$;
        \item a sequence $\{n_{i}\}_{i=1}^{\infty}$ of integer-valued mappings $n_{i}: K_{i} \to \mathbb{N} \cap (0,+\infty)$;
        \item a sequence $\{\mu_{i}\}_{i=1}^{\infty}$ of finite measures $\mu_{i}$ each supported on $K_{i}$;
        \item a sequence $\{\gamma_{i}\}_{i=1}^{\infty}$ of positive numbers $\gamma_{i}$,
      \end{enumerate}
  and verify that they satisfy the following properties:
      \begin{enumerate}[(a)]
        \item\label{ass:a} for each integer $i >0$, the family $\mathcal{B}_i= \{\overline{B}(x,\gamma_{i}): x \in K_{i}\}$ is disjoint;
        \item\label{ass:b} there exists a positive number $0<\varepsilon<\varepsilon_{0}$ such that for all integers $i >0$ and every $x \in K_{i}$, we have
  $$      
                \overline{B}_{n_{i}(x)}^{\vect{T}}(z,\varepsilon) \bigcap \Big( \bigcup_{y \in K_{i}\setminus\{x\}}\overline{B}(y,\gamma_{i})\Big)= \emptyset
  $$      
   for all $z \in \overline{B}(x,\gamma_{i})$;
        \item\label{ass:c} for each integer $i >0$ and every $x \in K_{i+1}$, there exists $y \in K_{i}$ such that 
            $$
            \overline{B}(x,\gamma_{i+1}) \subset \overline{B}\left(y,\frac{\gamma_{i}}{2}\right);
            $$
        \item\label{ass:d} for each integer $i >0$ and every $x \in K_{i}$,
  $$      
                \mu_{i}(\overline{B}(x,\gamma_{i})) \leq \sum_{y \in K_{i+1}(x)}\exp{\Big(-n_{i+1}(y)s + S_{n_{i+1}(y)}^{\vect{T}}\vect{f}(y)\Big)} \leq \big(1+\frac{1}{2^{i+1}}\big)\mu_{i}(\overline{B}(x,\gamma_{i})),
  $$      
              where $K_{i+1}(x)=\overline{B}(x,\gamma_{i}) \cap K_{i+1}$.
      \end{enumerate}

        \textbf{Step 1.} Construct $K_{1}$, $n_{1}(\cdot)$, $\mu_{1}$ and $\gamma_{1}$ satisfying \eqref{ass:a} and \eqref{ass:b}.

  Since $t<\prepac(\vect{T},\vect{f},K,\varepsilon_{0})$, it is clear that  $\msrpac_{\varepsilon_{0}}^{t}(\vect{T},\vect{f},K)=+\infty$. Let
        $$
  H=\mathlarger{\bigcup_{\substack{G\ \text{is open} \\ \msrpac_{\varepsilon_{0}}^{t}(\vect{T},\vect{f},K \cap G)=0}}} G.
  $$
   Since $X_{0}$ is a compact metric space, $X_0$ is second-countable and therefore a Lindelöf space (see \cite{Munkres2000}). Recall that a Lindelöf space is a space for which every open cover admits a countable subcover. Since $K\subset X_0$ is compact, the cover
        $$
  \{K \cap G: G\ \text{is open, and}\ \msrpac_{\varepsilon_{0}}^{t}(\vect{T},\vect{f},K \cap G)=0\}
  $$
  of $K \cap H$ has a countable subcover. Recall that $\msrpac_{\varepsilon_{0}}^{t}$ is an outer measure, and  we have that
  $$
  \msrpac_{\varepsilon_{0}}^{t}(\vect{T},\vect{f},K \cap H)=0.
  $$

  Let $K^{\prime} = K \setminus H = K \cap (X_{0} \setminus H)$. Then for each open set $G \subset X_{0}$ such that  $K^{\prime} \cap G \neq \emptyset$, we have
  \begin{equation}\label{KpOP}
  \msrpac_{\varepsilon_{0}}^{t}(\vect{T},\vect{f},K^{\prime} \cap G)>0.
  \end{equation}
  Indeed, if  $\msrpac_{\varepsilon_{0}}^{t}(\vect{T},\vect{f},K^{\prime} \cap G)=0$ for an open set $G$,
        then
        $$\msrpac_{\varepsilon_{0}}^{t}(\vect{T},\vect{f},K \cap G) \leq \msrpac_{\varepsilon_{0}}^{t}(\vect{T},\vect{f},K^{\prime} \cap G) + \msrpac_{\varepsilon_{0}}^{t}(\vect{T},\vect{f},K \cap H)=0,$$
  which implies $G \subset H$, and we have $K^{\prime} \cap G=\emptyset$.

  Since  $K^{\prime} = K \setminus H$, it follows that
        $$\msrpac_{\varepsilon_{0}}^{t}(\vect{T},\vect{f},K) \leq \msrpac_{\varepsilon_{0}}^{t}(\vect{T},\vect{f},K^{\prime})+ \msrpac_{\varepsilon_{0}}^{t}(\vect{T},\vect{f},K \cap H) = \msrpac_{\varepsilon_{0}}^{t}(\vect{T},\vect{f},K^{\prime}).$$
  Since $s<t$, this implies that
  $$
  \msrpac_{\varepsilon_{0}}^{s}(\vect{T},\vect{f},K^{\prime}) \geq \msrpac_{\varepsilon_{0}}^{t}(\vect{T},\vect{f},K^{\prime}) = \msrpac_{\varepsilon_{0}}^{t}(\vect{T},\vect{f},K)=+\infty.
  $$
        By Lemma \ref{badlem}, for $a=1$, $b=2$ and   $M=1$, there exists a finite disjoint collection $\{\overline{B}_{m_{i}}^{\vect{T}}(x_i,\varepsilon_{0})\}_{i \in \mathcal{I}}$ with $x_i\in K'$ and $m_i>M$ for all $i \in \mathcal{I}$
        such that
  $$
          1 < \sum_{i\in \mathcal{I}}\exp{\left(-m_i s + S_{m_i}^{\vect{T}}\vect{f}(x_i)\right)} < 2.
  $$

  We write  $K_{1}=\{x_i : i\in \mathcal{I}\}$,  and we define the integer-valued function $n_1:K_1 \to \mathbb{N}$  by
  $$
  n_1(x_i)=m_i  \quad\textit{ for each } x_i \in K_1.
  $$
  Hence, we have that
        \begin{equation}\label{eq:*1}
          1 < \sum_{x \in K_{1}}\exp{\left(-n_{1}(x)s + S_{n_{1}(x)}^{\vect{T}}\vect{f}(x)\right)} < 2.
        \end{equation}
  We define the finite measure $\mu_{1}$ supported on $K_1$ by
  $$
  \mu_{1}=\sum_{x \in K_{1}}\exp{\left(-n_{1}(x)s + S_{n_{1}(x)}^{\vect{T}}\vect{f}(x)\right)}\delta_{x},
  $$
        where $\delta_{x}$ denotes the Dirac measure at $x$.

Since  $\{\overline{B}_{m_{i}}^{\vect{T}}(x_i,\varepsilon_{0})\}_{i \in \mathcal{I}}$ is disjoint, it is clear that  the family $\{\overline{B}_{m_{i}}^{\vect{T}}(x_{i},\varepsilon)\}_{i \in \mathcal{I}}$ is   disjoint for all  $0<\varepsilon<\varepsilon_{0}$,  while the sum in \eqref{eq:*1} remains unchanged. Moreover, by Proposition \ref{prop:msrpacepsilon},  we have that
  $$
  \prepac(\vect{T},\vect{f},K,\varepsilon) \geq \prepac(\vect{T},\vect{f},K,\varepsilon_{0}),
  $$
  and it follows  that  $s<t<\prepac(\vect{T},\vect{f},K,\varepsilon)$.
    Hence, we choose $\varepsilon$  and  $\gamma_{1}$ such that
    $$
    0 < \varepsilon < \min\Big\{ \frac{\varepsilon_{0}}{2}, \frac{1}{3}\min_{\substack{x,y \in K_{1} \\ x \neq y}}d_{X_{0}}(x,y) \Big\}
  \quad \textit{ and }\quad
  0< \gamma_{1}<\frac{1}{6}\min_{\substack{x,y \in K_{1} \\ x \neq y}}d_{X_{0}}(x,y),
    $$
  and for each $x \in K_{1}$,  we have that for all $z \in \overline{B}(x,\gamma_{1})$,

        \begin{equation}\label{eq:4.1}
          \left(\overline{B}(z,\gamma_{1})\cup\overline{B}_{n_{1}(x)}^{\vect{T}}(z,\varepsilon)\right) \bigcap \Big(\bigcup_{y \in K_{1}\setminus\{x\}}\left(\overline{B}(w,\gamma_{1})\cup\overline{B}_{n_{1}(y)}^{\vect{T}}(w,\varepsilon)\right)\Big)=\emptyset,
        \end{equation}
        where $w \in \overline{B}(y,\gamma_{1})$ for each $y \in K_{1}\setminus\{x\}$.

  It immediately follows from  \eqref{eq:4.1} that  $\mathcal{B}_{1}=\{\overline{B}(x,\gamma_{1}): x \in K_{1}\}$ is disjoint, and  for each $x \in K_{1}$,  we have that  for all $z \in \overline{B}(x,\gamma_{1})$,
  $$
  \overline{B}_{n_{1}(x)}^{\vect{T}}(z,\varepsilon) \bigcap \Big(\bigcup_{y \in K_{1}\setminus\{x\}}\overline{B}(y,\gamma_{1}) \Big)= \emptyset  .
  $$
  Therefore $K_{1}$, $n_{1}(\cdot)$, $\mu_{1}$ and $\gamma_{1}$  satisfy \eqref{ass:a} and \eqref{ass:b}.

  For each $x \in K_{1}$, we write  $G=B\left(x,\frac{1}{4}\gamma_{1}\right)$. Since $K_{1} \subset K^{\prime}\subset K$, it follows from \eqref{KpOP} that
        \begin{equation}\label{eq:4.1.1}
          \msrpac_{\varepsilon_{0}}^{t}\Big(\vect{T},\vect{f},K \cap B\left(x,\frac{\gamma_{1}}{4}\right)\Big) \geq \msrpac_{\varepsilon_{0}}^{t}\left(\vect{T},\vect{f},K^{\prime} \cap B\left(x,\frac{\gamma_{1}}{4}\right)\right)>0.
        \end{equation}

      \textbf{Step 2.} Construct $K_{2}$, $n_{2}(\cdot)$,  $\mu_{2}$ and $\gamma_{2}$, and verify that $K_{1}$, $n_{1}(\cdot)$, $\mu_{1}$ and $\gamma_{1}$ satisfy \eqref{ass:c} and \eqref{ass:d}.

  We construct these sequences locally and then glue them together
  since $\mathcal{B}_{1}=\{\overline{B}(x,\gamma_{1})\}_{x \in K_{1}}$ is disjoint.
        For each $x \in K_{1}$, we write
  $$
  F(x) = K \cap B\left(x,\frac{\gamma_{1}}{4}\right),
  $$
  and it is clear that
  $$
     \msrpac_{\varepsilon_{0}}^{t}\left(\vect{T},\vect{f},F(x)\right)=\msrpac_{\varepsilon_{0}}^{t}\left(\vect{T},\vect{f},K \cap B\left(x,\frac{\gamma_{1}}{4}\right)\right)>0.
  $$

  By the same argument as Step 1, we write
   $$
  H(x)=\mathlarger{\bigcup_{\substack{G\ \text{is open} \\ \msrpac_{\varepsilon_{0}}^{t}(\vect{T},\vect{f},F(x) \cap G)=0}}}G,
  $$
  and we have that $\msrpac_{\varepsilon_{0}}^{t}(\vect{T},\vect{f},F(x) \cap H(x))=0$. Set $F^{\prime}(x)=F(x) \setminus H(x)$, and it is clear that
  \begin{equation}\label{ineq_2ptKG>0}
  \msrpac_{\varepsilon_{0}}^{t}(\vect{T},\vect{f},F^{\prime}(x) \cap G)>0
  \end{equation}
  for every open set $G$ with $G \cap F^{\prime}(x) \neq \emptyset$. By \eqref{eq:4.1.1}, we have that
  $$
  \msrpac_{\varepsilon_{0}}^{t}(\vect{T},\vect{f},F^{\prime}(x))=\msrpac_{\varepsilon_0}^{t}(\vect{T},\vect{f},F(x))>0.
  $$
  Since $s<t$, it immediately follows that
        $$
  \msrpac_{\varepsilon_{0}}^{s}(\vect{T},\vect{f},F^{\prime}(x))=+\infty.
  $$
  By Lemma \ref{badlem},
  for $a= \mu_{1}(\{x\}),$ $b=(1+\frac{1}{2^{2}})\mu_{1}(\{x\})$ and $M=\max_{y \in K_{1}}n_{1}(y)$, there exists a finite disjoint collection $\{\overline{B}_{m_{i}}^{\vect{T}}(x_i,\varepsilon_{0})\}_{i \in \mathcal{I}}$ with $x_i\in F^{\prime}(x)$ and $m_i>M$ for all $i \in \mathcal{I}$ such that
  $$
  \mu_{1}(\{x\})< \sum_{i\in \mathcal{I}}\exp{\left(-m_i s + S_{m_i}^{\vect{T}}\vect{f}(x_i)\right)} < \Big(1+\frac{1}{2^{2}}\Big)\mu_{1}(\{x\}).
  $$

  We write  $K_{2}(x)=\{x_i : i\in \mathcal{I}\}$ and  define the integer-valued function $n_2:K_{2}(x) \to \mathbb{N}$  by
  \begin{equation}\label{def_n2x}
  n_2(x_i)=m_i  \qquad \textit{ for   } x_i \in K_{2}(x).
  \end{equation}
  It is clear that $K_{2}(x) \subset F^{\prime}(x)$ is finite and
  $$
  K_{2}(x) \subset F(x) = K \cap B\left(x,\frac{\gamma_{1}}{4}\right) \subset \overline{B}(x,\gamma_{1}).
  $$
  Moreover, since $\varepsilon<\varepsilon_{0}$, the collection $\{\overline{B}_{n_{2}(y)}^{\vect{T}}(y,\varepsilon)\}_{y \in K_{2}(x)}$ is disjoint, and the integer-valued function
        $$n_{2}: K_{2}(x) \to \mathbb{N} \cap \big(M,+\infty\big)$$
  satisfies that
  \begin{equation}\label{ineq_mu1K2}
  \mu_{1}(\{x\}) < \sum_{y \in K_{2}(x)}\exp{\left(-n_{2}(y)s+S_{n_{2}(y)}^{\vect{T}}\vect{f}(y)\right)} < \Big(1+\frac{1}{2^{2}}\Big)\mu_{1}(\{x\}).
  \end{equation}

        Let $K_{2}=\bigcup_{x \in K_{1}}K_{2}(x)$. Clearly, $K_{2}$ is finite. Since the family $\mathcal{B}_{1}=\{\overline{B}(x,\gamma_{1})\}_{x \in K_{1}}$ is disjoint,
        we have that $K_{2}(x) \cap K_{2}(x^{\prime})=\emptyset$ for $x\neq x^{\prime} \in K_{1}$.
  We put the local mapping $n_{2}$ given by \eqref{def_n2x} on every $K_{2}(x)$ together to one global mapping on $K_{2}$, and it is well defined on $K_{2}$.
  Moreover, for each $x \in K_{1}$, it is clear that $K_{2}(x)=\overline{B}(x,\gamma_{1}) \cap K_{2}$, and  by inequality \eqref{ineq_mu1K2}, we have that
  $$
  \mu_{1}(\overline{B}(x,\gamma_{1})) \leq \sum_{y \in K_{2}(x)}\exp{\Big(-n_{2}(y)s + S_{n_{2}(y)}^{\vect{T}}\vect{f}(y)\Big)} \leq \big(1+\frac{1}{2^{2}}\big)\mu_{1}(\overline{B}(x,\gamma_{1})),
  $$
  which verifies that $K_{1}$, $n_{1}(\cdot)$, $\mu_{1}$ and $\gamma_{1}$ satisfy the property \eqref{ass:d}.

   We define the finite measure $\mu_{2}$ supported on $K_{2}$ by
   $$
  \mu_{2}=\sum_{y \in K_{2}}\exp{\left(-n_{2}(y)s + S_{n_{2}(y)}^{\vect{T}}\vect{f}(y)\right)}\delta_{y}.
  $$
  Recall that $\{\overline{B}_{n_{2}(y)}^{\vect{T}}(y,\varepsilon)\}_{y \in K_{2}(x)}$ is disjoint for each $x \in K_{1}$,
  and that $n_{2}(\cdot) > M$.
  Combining this with \eqref{eq:4.1}, it follows that  $\{\overline{B}_{n_{2}(y)}^{\vect{T}}(y,\varepsilon)\}_{y \in K_{2}}$ is disjoint.

  For each $0<\eta <\frac{\varepsilon}{2}$, by the uniform continuity of   $\vect{T}^{j}$,
    there exists $\delta >0$ such that  for all $z,w \in X_{0}$ satisfying $d_{X_{0}}(z,w) \leq \delta$, we have that
    $$
    d_{X_{j}}(\vect{T}^{j}z,\vect{T}^{j}w) \leq \eta \qquad \text{for all} \quad 0 \leq j \leq \max_{y \in K_{2}}n_{2}(y)-1.
    $$
  Hence, given $x\in K_2$, for all $z \in B(x,\delta)$ and all $z^{\prime} \in X_{0}$ satisfying that $  d_{X_{0}}(z,z^{\prime}) \leq \delta$, we have that
    $$
    d_{n_{2}(x)}^{\vect{T}}(z,z^{\prime}) \leq \eta ,
    $$
  and this implies that $\overline{B}(z,\delta) \subset \overline{B}_{n_{2}(x)}^{\vect{T}}(z,\eta)$.
  Immediately,  for all $x \in K_{2}$ and all $z \in B(x,\delta)$,   we obtain that
    $$
    \overline{B}(z,\delta) \subset \overline{B}_{n_{2}(x)}^{\vect{T}}(x,\varepsilon)
    $$
and
    $$
    \overline{B}_{n_{2}(x)}^{\vect{T}}(z,\varepsilon) \subset \overline{B}_{n_{2}(x)}^{\vect{T}}(x,\varepsilon_{0}).
    $$
    Hence, we choose $\gamma_{2}$ satisfying that
    $$
    0<\gamma_{2}<\min\Big\{ \frac{\gamma_{1}}{4}, \delta\Big\}
    $$
    so that for each $x \in K_{2}$ and all $z \in B(x,\gamma_{2})$,
        \begin{equation}\label{eq:4.2}
          \left(\overline{B}(z,\gamma_{2}) \cup \overline{B}_{n_{2}(x)}^{\vect{T}}(z,\varepsilon)\right) \bigcap \Big(\bigcup_{y \in K_{2}\setminus\{x\}}\Big(\overline{B}(w,\gamma_{2}) \cup \overline{B}_{n_{2}(y)}^{\vect{T}}(w,\varepsilon)\Big) \Big)= \emptyset,
        \end{equation}
   where $w \in B(y,\gamma_{2})$ for every $y \in K_{2}\setminus\{x\}$.

        We write
        $$
        \mathcal{B}_{2}=\{\overline{B}(x,\gamma_{2}): x \in K_{2}\}.
        $$
  For each $x \in K_{2}$, there exists  $y \in K_{1}$ such that $x\in K_{2}(y) \subset K \cap B\left(y,\frac{1}{4}\gamma_{1}\right)$. Since $ 0<\gamma_{2}<\frac{\gamma_{1}}{4}$, we have that
        $$
  \overline{B}(x,\gamma_{2}) \subset \overline{B}\left(y,\frac{\gamma_{1}}{2}\right),
  $$
  which verifies that $K_{1}$, $n_{1}(\cdot)$, $\mu_{1}$ and $\gamma_{1}$ satisfy the property \eqref{ass:c}.

  By \eqref{eq:4.2}, it is clear that  $\mathcal{B}_{2}$ is disjoint, and for each $x \in K_{2}$, we have that
        $$
  \overline{B}_{n_{2}(x)}^{\vect{T}}(z,\varepsilon) \bigcap \Big( \bigcup_{y \in K_{2}\setminus\{x\}}\overline{B}(y,\gamma_{2}) \Big) = \emptyset 
  $$
   for all $z \in \overline{B}(x,\gamma_{2})$, which shows that $K_{2}$, $n_{2}(\cdot)$, $\mu_{2}$ and $\gamma_{2}$ satisfy \eqref{ass:a} and \eqref{ass:b}.

  Moreover, for every $x \in K_{2}$, there exists   $y \in K_{1}$ such that  $B\left(x,\frac{\gamma_{2}}{4}\right) \cap K_{2}(y) \neq \emptyset$   and  $B\left(x,\frac{\gamma_{2}}{4}\right) \cap F^{\prime}(y) \neq \emptyset$.  By \eqref{ineq_2ptKG>0},  it implies that for every $x \in K_{2}$,
        \begin{equation}\label{eq:4.2.1}
          \msrpac_{\varepsilon_{0}}^{t}\left(\vect{T},\vect{f},K \cap B\big(x,\frac{\gamma_{2}}{4}\big)\right) \geq \msrpac_{\varepsilon_{0}}^{t}\left(\vect{T},\vect{f},F^{\prime}(x) \cap B\big(x,\frac{\gamma_{2}}{4}\big)\right) >0.
        \end{equation}

       \textbf{Step 3.} Assume that $K_{i}$, $n_{i}(\cdot)$, $\mu_{i}$ and $\gamma_{i}$ have been
        defined for $i=1,\ldots,p$ such that

  \noindent (i). $K_{i}$, $n_{i}(\cdot)$, $\mu_{i}$ and $\gamma_{i}$ satisfy \eqref{ass:a}\eqref{ass:b}\eqref{ass:c}\eqref{ass:d} for  $i=1,,2,\ldots, p-1$;

  \noindent (ii). $K_{p}$, $n_{p}(\cdot)$, $\mu_{p}$ and $\gamma_{p}$ satisfy \eqref{ass:a} and \eqref{ass:b};

  \noindent (iii).   $\gamma_{i} < \frac{1}{4}\gamma_{i-1}$ for all $i=2,\ldots,p$;

  \noindent (iv).  for each $x \in K_{p}$, we have that   for all $z \in \overline{B}(x,\gamma_{p})$,
          \begin{equation}\label{eq:4.p}
            \left(\overline{B}(z,\gamma_{p}) \cup \overline{B}_{n_{p}(x)}^{\vect{T}}(z,\varepsilon)\right) \bigcap \Big(\bigcup_{y \in K_{p}\setminus\{x\}}\Big(\overline{B}(w,\gamma_{p}) \cup \overline{B}_{n_{p}(y)}^{\vect{T}}(w,\varepsilon)\Big) \Big)= \emptyset,
          \end{equation}
       where $w \in \overline{B}(y,\gamma_{p})$ for  $y \in K_{p}\setminus\{x\}$;

  \noindent (v).  for each $x \in K_{p}$, we have that
          \begin{equation}\label{eq:4.2.2}
            \msrpac_{\varepsilon_{0}}^{t}\left(\vect{T},\vect{f},K \cap B\left(x,\frac{\gamma_{p}}{4}\right)\right)>0.
          \end{equation}

  Next, we construct $K_{p+1}$, $n_{p+1}(\cdot)$, $\mu_{p+1}$ and $\gamma_{p+1}$. Meanwhile we show that
  $K_{p}$, $n_{p}(\cdot)$, $\mu_{p}$ and $\gamma_{p}$ satisfy \eqref{ass:c} and \eqref{ass:d}  and that $K_{p+1}$, $n_{p+1}(\cdot)$, $\mu_{p+1}$ and $\gamma_{p+1}$ satisfy \eqref{ass:a} and \eqref{ass:b}.

  Similar to Step 2, fix $x \in K_{p}$, and we write $F(x)=K \cap B\left(x,\frac{\gamma_{p}}{4}\right)$,
  $$
  H(x)=\mathlarger{\bigcup_{\substack{G\ \text{is open} \\ \msrpac_{\varepsilon_{0}}^{t}(\vect{T},\vect{f},F(x) \cap G)=0}}}G,
  $$
  and  $F^{\prime}(x)=F(x) \setminus H(x)$. Similarly, for every open set $G$ such that  $G \cap F^{\prime}(x) \neq \emptyset$,  we have that
  \begin{equation}\label{ineq_p+1tKG>0}
    \msrpac_{\varepsilon_{0}}^{t}(\vect{T},\vect{f},F^{\prime}(x) \cap G)>0.
  \end{equation}

  Since $s<t$, it follows that

  $$
  \msrpac_{\varepsilon_{0}}^{s}(\vect{T},\vect{f},F^{\prime}(x))=+\infty.
  $$
  By Lemma \ref{badlem},
  for $a= \mu_{p}(\{x\})$, $b=(1+\frac{1}{2^{p+1}})\mu_{p}(\{x\})$ and $M=\max_{y \in K_{p}}n_{p}(y)$, there exists a finite disjoint collection $\{\overline{B}_{m_{i}}^{\vect{T}}(x_i,\varepsilon_{0})\}_{i \in \mathcal{I}}$ with $x_i\in F^{\prime}(x)$ and $m_i>M$ for all $i \in \mathcal{I}$ such that
  $$
  \mu_{p}(\{x\})< \sum_{i\in \mathcal{I}}\exp{\left(-m_i s + S_{m_i}^{\vect{T}}\vect{f}(x_i)\right)} < (1+\frac{1}{2^{p+1}})\mu_{p}(\{x\}).
  $$

  We write  $K_{p+1}(x)=\{x_i : i\in \mathcal{I}\}$, and we define the integer-valued function $n_{p+1}:K_{p+1}(x) \to \mathbb{N}$  by
  \begin{equation}\label{def_np+1x}
  n_{p+1}(x_i)=m_i  \qquad \textit{ for each } x_i \in K_{p+1}(x).
  \end{equation}
  It is clear that $K_{p+1}(x) \subset F^{\prime}(x)$ is finite and that
  $$
  K_{p+1}(x) \subset F(x) = K \cap B\left(x,\frac{\gamma_{p}}{4}\right) \subset \overline{B}(x,\gamma_{p}).
  $$
  Moreover, since $\varepsilon<\varepsilon_{0}$, the collection $\{\overline{B}_{n_{p+1}(y)}^{\vect{T}}(y,\varepsilon)\}_{y \in K_{p+1}(x)}$ is disjoint, and the integer-valued function
        $$n_{p+1}: K_{p+1}(x) \to \mathbb{N} \cap \Big(\max_{y \in K_{p}}n_{p}(y),+\infty\Big)$$
  satisfies that
  \begin{equation}\label{ineq_mu1Kp}
  \mu_{p}(\{x\}) < \sum_{y \in K_{p+1}(x)}\exp{\left(-n_{p+1}(y)s+S_{n_{p+1}(y)}^{\vect{T}}\vect{f}(y)\right)} < \Big(1+\frac{1}{2^{p+1}}\Big)\mu_{p}(\{x\}),
  \end{equation}

  Let $K_{p+1}=\bigcup_{x \in K_{p}}K_{p+1}(x)$. Clearly $K_{p+1}$ is finite. We piece together the mapping $n_{p+1}(\cdot)$ given by \eqref{def_np+1x} on $ K_{p+1}(x)$ for every $x \in K_{p}$ to the one on $K_{p+1}$. The new mapping $n_{p+1}(\cdot)$ on $K_{p+1}$  is well defined since $K_{p+1}(x) \cap K_{p+1}(x^{\prime}) = \emptyset$ for $x \neq x^{\prime} \in K_{p}$.
  Furthermore, for each $x \in K_{p}$, we have that $K_{p+1}(x)=\overline{B}(x,\gamma_{p}) \cap K$. Hence by the inequality \eqref{ineq_mu1K2}, we obtain that
  $$
  \mu_{p}(\overline{B}(x,\gamma_{p})) \leq \sum_{y \in K_{p+1}(x)}\exp{\Big(-n_{p+1}(y)s + S_{n_{p+1}(y)}^{\vect{T}}\vect{f}(y)\Big)} \leq \big(1+\frac{1}{2^{p+1}}\big)\mu_{p}(\overline{B}(x,\gamma_{p})),
  $$
  which verifies that $K_{p}$, $n_{p}(\cdot)$, $\mu_{p}$ and $\gamma_{p}$ satisfy the property \eqref{ass:d}.

        We define the finite measure $\mu_{p+1}$ supported on $K_{p+1}$ by
  \begin{equation}\label{def_mup+1}
  \mu_{p+1}=\sum_{y \in K_{p+1}}\exp{\left(-n_{p+1}(y)s + S_{n_{p+1}}^{\vect{T}}\vect{f}(y)\right)}\delta_{y}.
  \end{equation}

        Recall that the collection $\{\overline{B}_{n_{p+1}(y)}^{\vect{T}}(y,\varepsilon)\}_{y \in K_{p+1}(x)}$ is disjoint for each $x \in K_{p}$
        and that $n_{p+1}(\cdot)>\max_{y \in K_{p}}{n_{p}(y)}$.
        Combining this with the induction hypothesis \eqref{eq:4.p}, we have that $\{\overline{B}_{n_{p+1}(y)}^{\vect{T}}(y,\varepsilon)\}_{y \in K_{p+1}}$
        is disjoint.

  Choose $0<\eta <\frac{\varepsilon}{2}$. By the uniform continuity of  $\vect{T}^{j}$,
   there exists  $\delta>0$ such that   for all $z,w \in X_{0}$ satisfying that  $d_{X_{0}}(z,w) \leq \delta$, we have that
    $$
    d_{X_{j}}(\vect{T}^{j}z,\vect{T}^{j}w) \leq \eta
   $$
   for all  $ 0 \leq j \leq \max_{y \in K_{p+1}}n_{p+1}(y)-1$. For every $x \in K_{p+1}$, this implies that
   $\overline{B}(z,\delta) \subset \overline{B}_{n_{p+1}(x)}^{\vect{T}}(z,\eta)$ for all $z \in B(x,\delta)$.
  Since $ \eta <\frac{\varepsilon}{2}$,     for every $x \in K_{p+1}$,  we immediately obtain  that
    $$
    \overline{B}(z,\delta) \subset \overline{B}_{n_{p+1}(x)}^{\vect{T}}(x,\varepsilon),
    $$
      and
      $$
      \overline{B}_{n_{p+1}(x)}^{\vect{T}}(z,\varepsilon) \subset \overline{B}_{n_{p+1}(x)}^{\vect{T}}(x,\varepsilon_{0})
      $$
   for all $z \in B(x,\delta)$.

  Hence we choose $\gamma_{p+1}$ such that
  $$
  0<\gamma_{p+1}<\min\Big\{ \frac{\gamma_{p}}{4}, \delta \Big\},
  $$
  so that for each $x \in K_{p+1}$ and all $z \in B(x,\gamma_{p+1})$,
        \begin{equation}\label{eq:4.4}
          \left(\overline{B}(z,\gamma_{p+1}) \cup \overline{B}_{n_{p+1}(x)}^{\vect{T}}(z,\varepsilon)\right) \bigcap \Big( \bigcup_{y \in K_{p+1}\setminus\{x\}}\big(\overline{B}(w,\gamma_{p+1}) \cup \overline{B}_{n_{p+1}(y)}^{\vect{T}}(w,\varepsilon)\big) \Big)= \emptyset,
        \end{equation}
   where $w \in B(y,\gamma_{p+1})$ for every $y \in K_{p+1}\setminus\{x\}$.

        Write
        $$
        \mathcal{B}_{p+1}=\{\overline{B}(x,\gamma_{p+1}): x \in K_{p+1}\}.
        $$
  For every $x \in K_{p+1}$, there exists $y\in K_p$ such that $x\in K_{p+1}(y) \subset K \cap B\left(y,\frac{\gamma_{p}}{4}\right)$.
        Since $\gamma_{p+1}<\frac{\gamma_{p}}{4}$, we obtain  that
        $$
  \overline{B}(x,\gamma_{p+1}) \subset \overline{B}\left(y,\frac{\gamma_{p}}{2}\right),
  $$
   which verifies that $K_{p}$, $n_{p}(\cdot)$, $\mu_{p}$ and $\gamma_{p}$ satisfy the property \eqref{ass:c}.

        By \eqref{eq:4.4}, it is clear that $\mathcal{B}_{p+1}$ is disjoint, and  for every $x \in K_{p+1}$. we have that
        $$
  \overline{B}_{n_{p+1}(x)}^{\vect{T}}(z,\varepsilon) \bigcap \Big(\bigcup_{y \in K_{p+1}\setminus\{x\}}\overline{B}(y,\gamma_{2})\Big) = \emptyset ,
  $$
  which verifies that $K_{p+1}$, $n_{p+1}(\cdot)$, $\mu_{p+1}$ and $\gamma_{p+1}$ satisfy \eqref{ass:a} and \eqref{ass:b}.

  Moreover, for each $x \in K_{p+1}$, there exists   $y \in K_{p}$ such that   $B\left(x,\frac{\gamma_{p+1}}{4}\right) \cap K_{p+1}(y) \neq \emptyset$ and $B\left(x,\frac{\gamma_{p+1}}{4}\right) \cap F^{\prime}(y) \neq \emptyset$,  and by \eqref{ineq_p+1tKG>0}, this implies that
        \begin{equation}
          \msrpac_{\varepsilon_{0}}^{t}\left(\vect{T},\vect{f},K \cap B\left(x,\frac{\gamma_{p+1}}{4}\right)\right) \geq \msrpac_{\varepsilon_{0}}^{t}\Big(\vect{T},\vect{f},F^{\prime}(x) \cap B\left(x,\frac{\gamma_{p+1}}{4}\right)\Big) >0 ,
        \end{equation}
  which  completes  the induction process.

  Therefore,  we obtain the sequences  $\{K_{i}\}_{i=1}^{\infty}$, $\{n_{i}(\cdot)\}_{i=1}^{\infty}$, $\{\mu_{i}\}_{i=1}^{\infty}$ and $\{\gamma_{i}\}_{i=1}^{\infty}$  satisfying the desired properties \eqref{ass:a},\eqref{ass:b},\eqref{ass:c} and \eqref{ass:d}.
  
  We are ready to construct the desired measure $\mu$ by using these sequences.  For each integer $i>0$, it follows from \eqref{ass:a}  that
        $$
        \mu_{i+1}(\overline{B}_{i})=\sum_{\substack{\overline{B}\in\mathcal{B}_{i+1}:\\ \overline{B} \subset \overline{B}_{i}}}\mu_{i+1}(\overline{B}).
        $$
  Combining this with \eqref{ass:d}, we immediately obtain that for all $\overline{B}_{i}\in\mathcal{B}_{i}$,
        $$
        \mu_{i}(\overline{B}_{i}) \leq \mu_{i+1}(\overline{B}_{i}) \leq \left(1+\frac{1}{2^{i+1}}\right)\mu_{i}(\overline{B}_{i}).
        $$
  Applying the  inequalities repeatedly, for all $k>i$ and for all $\overline{B}_{i} \in \mathcal{B}_{i}$, we have that
        \begin{equation}\label{eq:4.5}
          \mu_{i}(\overline{B}_{i}) \leq \mu_{k}(\overline{B}_{i}) \leq \prod_{l=i+1}^{k}\left(1+\frac{1}{2^{l}}\right)\mu_{i}(\overline{B}_{i}) \leq C\mu_{i}(\overline{B}_{i}),
        \end{equation}
        where $C=\prod_{l=1}^{\infty}(1+2^{-l})<+\infty$. Let
        $$
        K^{*}=\bigcap_{l=1}^{\infty}\overline{\bigcup_{i=l}^{\infty}K_{i}}.
        $$
        It is straightforward that $K^{*} \subset K$ is compact,  and
        for every integer $i>0$,  by \eqref{ass:c} we have that
        $$
        K^{*} \subset \bigcup_{x \in K_{i}}\overline{B}\left(x,\frac{\gamma_{i}}{2}\right).
        $$

  Since all the measures $\mu_{i}$    are linear combinations of Dirac measures on $X_{0}$, they  are  Radon measures on $X_{0}$.
  By the Riesz representation theorem (see \cite[Thm.7.2]{Folland1999}), there exists a weak* limit point of $\{\mu_{i}\}_{i=1}^{\infty}$,  denoted by $\widetilde{\mu}$,
  and  $\widetilde{\mu}$ is  a Borel (in fact Radon) measure supported on $K^{*}$.

  For each integer $i>0$,  since $K^{*} \subset \bigcup_{x \in K_{i}}\overline{B}\left(x,\frac{\gamma_{i}}{2}\right)$,
  we have that $\widetilde{\mu}(\partial{B(x,\gamma_{i})})=0$ for every $x \in K_{i}$, and it follows by  \cite[Chap.1 Ex.9]{Mattila1995}  that
  $$
  \lim_{k\to \infty} \mu_{k }(B(x,\gamma_{i})) = \widetilde{\mu}(B(x,\gamma_{i})).
  $$
  For all $x \in K_{i}$, we have $\mu_{k}(\overline{B}(x,\gamma_{i})) = \mu_{k}(B(x,\gamma_{i}))$ for all $k>i$, which combined with \eqref{eq:4.5} implies that
  $$
  \mu_{i}(\overline{B}(x,\gamma_{i})) \leq \widetilde{\mu}(B(x,\gamma_{i})) \leq C\mu_{i}(\overline{B}(x,\gamma_{i})).
  $$
  By \eqref{def_mup+1}, we obtain that for each $x\in K_i$,
  \begin{equation}\label{mugmbd}
   \exp{\left(-n_{i}(x)s+S_{n_{i}(x)}^{\vect{T}}\vect{f}(x)\right)}
          \leq \widetilde{\mu}( B(x,\gamma_{i}))
          \leq  C\exp{\left(-n_{i}(x)s+S_{n_{i}(x)}^{\vect{T}}\vect{f}(x)\right)}.
  \end{equation}
  Taking $i=1$,  we immediately have that $ \widetilde{\mu}(K^{*})$ is positive and finite, i.e.,
        $$
        1 \leq \sum_{x \in K_{1}}\mu_{1}(B(x,\gamma_{1})) \leq \widetilde{\mu}(K^{*}) \leq \sum_{x \in K_{1}}C\mu_{1}(B(x,\gamma_{1})) \leq 2C.
        $$
  Combining \eqref{mugmbd} with the properties \eqref{ass:b} and \eqref{ass:c}, for each $x \in K_{i}$ and all $z \in \overline{B}(x,\gamma_{i})$, we obtain that
        $$
        \widetilde{\mu}(\overline{B}_{n_{i}(x)}^{\vect{T}}(z,\varepsilon)) \leq \widetilde{\mu}\Big(\overline{B}\big(x,\frac{\gamma_{i}}{2}\big)\Big) \leq C\exp{\left(-n_{i}(x)s+S_{n_{i}(x)}^{\vect{T}}\vect{f}(x)\right)}.
        $$

  Since  $K^{*} \subset \bigcup_{x \in K_{i}}\overline{B}(x,\frac{\gamma_{i}}{2})$ for all integers $i>0$,    we have that  for each $z \in K^{*}$ and each integer $i>0$, there exists $x \in K_{i}$ such that $z \in \overline{B}\left(x,\frac{\gamma_{i}}{2}\right)$, which implies that
  \begin{equation}\label{ineq_mutleq ce}
        \widetilde{\mu}(B_{n_{i}(x)}^{\vect{T}}(z,\varepsilon)) \leq C\exp{\left(-n_{i}(x)s+S_{n_{i}(x)}^{\vect{T}}\vect{f}(x)\right)}.
  \end{equation}

        Finally, we define a probability measure $\mu$ on $X_{0}$ by
        $$
        \mu=\frac{\widetilde{\mu}}{\widetilde{\mu}(K^{*})}.
        $$
  It is clear that $\mu \in M(X_{0})$ and  $\mu(K)=\mu(K^{*})=1$. Moreover, by \eqref{ineq_mutleq ce},  for each $z \in K^{*}$, there exists a strictly increasing sequence of integers $k_{i}$ such that
        $$
  \mu(B_{k_{i}}^{\vect{T}}(z,\varepsilon)) \leq \frac{C}{{\widetilde{\mu}(K)}}\exp{\left(-k_{i}s+S_{k_{i}}^{\vect{T}}\vect{f}(z)\right)}.
  $$
  Combining this with Proposition \ref{prop:localepsilon}, we have that
  \begin{align*}
    &\preup_{\mu}(\vect{T},\vect{f}) = \int_{X_{0}}\preup_{\mu}(\vect{T},\vect{f},x)\dif{\mu(x)} \\
          &\geq \int_{X_{0}}\preup_{\mu}(\vect{T},\vect{f},x,\varepsilon)\dif{\mu(x)} \\
          &= \int_{K^{*}}\preup_{\mu}(\vect{T},\vect{f},z,\varepsilon)\dif{\mu(z)} \\
          &\geq \int_{K^{*}}\lim_{i \to \infty}\frac{-\left(\log{C}-\log{\widetilde{\mu}(K)}-k_{i}s+S_{k_{i}}^{\vect{T}}\vect{f}(z)\right)+S_{k_{i}}^{\vect{T}}\vect{f}(z)}{k_{i}}\dif{\mu(z)} \\
          &= s.
        \end{align*}
  This implies that  for every $s<\prepac(\vect{T},\vect{f},K)$, we have constructed a measure $\mu$ with $\mu(K)=1$ such that  $\preup_{\mu}(\vect{T},\vect{f})>s$, which  completes the proof of Theorem \ref{thm:varprinciplepacking}.
    \end{proof}

\section{Acknowledge*}
On behalf of all authors, the corresponding author states that there is no conflict of interest.


\end{document}